\documentclass[a4paper]{article}
\usepackage[utf8]{inputenc}
\usepackage{amsmath}
\usepackage{amssymb}
\usepackage{amsthm}
\usepackage{dsfont}
\usepackage{xcolor}
\usepackage{soul,xcolor}
\usepackage{mathtools}
\usepackage[title]{appendix}
\usepackage{enumerate}
\usepackage{amssymb, amsmath,graphicx}
\usepackage{verbatim}
\usepackage{float}
\mathtoolsset{showonlyrefs}
\usepackage{hyperref}
\usepackage{cancel}
\usepackage{nicematrix}
\usepackage{pgfplots}
\pgfplotsset{compat=1.18}
\usepackage{esint}
\usepackage{mathrsfs}
\usepackage{soul}

\usepackage{tikz}
\usetikzlibrary{angles,quotes}
\usetikzlibrary{calc}

\theoremstyle{plain}
\newtheorem{thm}{Theorem}[section]
\newtheorem{theorem}[thm]{Theorem}
\newtheorem{proposition}[thm]{Proposition}
\newtheorem{definition}[thm]{Definition}
\newtheorem{lemma}[thm]{Lemma}
\newtheorem{corollary}[thm]{Corollary}

\newtheorem{remark}[thm]{Remark}
\newtheorem{example}[thm]{Example}
\newtheorem*{claim}{Claim}
\newtheorem*{example*}{Example}
\newtheorem*{remark*}{Remark}

\newcommand{\Ric}{\mathrm{Ric}}
\newcommand{\Sic}{\mathrm{Sic}}

\newcommand{\Hess}{\mathrm{Hess}}
\newcommand{\Div}{\mathrm{div}}

\newcommand{\diam}{\mathrm{diam}}
\newcommand{\tr}{\mathrm{tr}}
\newcommand{\mat}[1]{\left[ \begin{matrix} #1 \end{matrix} \right]}

\def\a{{\mathfrak a}}
\def\AA{{\mathfrak A}}
\def\k{{\mathfrak k}}
\def\KK{{\mathfrak K}}

\newcommand{\BE}{{\sf BE}}
\newcommand{\RCD}{{\sf RCD}}
\newcommand{\CD}{{\sf CD}}
\newcommand{\MCP}{{\sf MCP}}

\newcommand{\ind}{\mathds{1}}
\newcommand{\eps}{\varepsilon}
\newcommand{\La}{\left\langle}
\newcommand{\Ra}{\right\rangle}
\newcommand{\Rm}{\mathrm{Rm}}
\newcommand{\Id}{\mathrm{Id}}

\def\N{{\mathbb N}}
\def\Z{{\mathbb Z}}
\def\R{{\mathbb R}}

\DeclareMathOperator{\vol}{\mathsf{vol}}

\graphicspath{{images/}}

\begin{document}

\title{Gaussian Volume Functional, Integral Scalar Curvature, \\ and Minimal Super-Ricci Flows}

\date{}
\author{Marco Flaim, \quad Erik Hupp, \quad Karl-Theodor Sturm
\\[1cm]
\small Hausdorff Center for Mathematics \& Institute for Applied Mathematics\\
\small University of Bonn, Germany
}

\maketitle

\begin{abstract}
    We present a synthetic notion of scalar curvature (and its integral) for Riemannian manifolds and metric measure spaces, defined in terms of the initial slope of a Gaussian (double) integral. 

    We explicitly calculate the integral scalar curvature for Lipschitz gluings of smooth Riemannian manifolds and for cones. In dimension 2, the former coincides with the formula derived by Gauss-Bonnet, whereas the latter differs.

    The extension to the time-dependent case allows us to characterize Ricci flows as super Ricci flows with minimal integral curvature functional.
\end{abstract}

\setcounter{tocdepth}{2}
\tableofcontents

\section{Introduction}
On a smooth Riemannian manifold, lower Ricci bounds can be characterized in several ways, in particular using gamma calculus \cite{BakryEmeryDiffHypercontr}, or optimal transport \cite{CEMS01, sturmvonrenesse}. The latter made it possible to define a notion of lower Ricci bounds for metric measure spaces \cite{SturmActa1, SturmActa2, LottVillani}. Later, under suitable conditions, this notion was shown to be equivalent to the gamma calculus approach \cite{AGS1, AGS2, AGS3, EKS}. This theory turned out to be powerful also in a time-dependent setting, where these tools enable characterizations of super solution to Ricci flows \cite{mccanntopp, SturmSRF, BamlerEntropy}, and hence give a weak notion of them \cite{KopferSturm18, BamlerSRF}.

In order to characterize Einstein manifolds and Ricci flows, the challenge is to study upper Ricci bounds and sub solution of Ricci flows. In \cite{sturm2017remarks} this was done by asymptotically reversing the lower bound estimates, and in \cite{erbar2025synthetic} this was extended to the dynamic case. Another viewpoint was taken in \cite{naber2013characterizations, HN18} based on functional inequalities on the path space.

We present here an approach to synthetic characterizations of Ricci flat manifolds and of Ricci flows via the asymptotic behaviour of volumes. Let us illustrate the basic idea in the static case. In order to characterize Ricci flatness of closed Riemannian manifolds $(M,g)$, it suffices to characterize \emph{non-negative Ricci curvature} and non-positive scalar curvature or \emph{non-positive integral of scalar curvature} since
$$\Ric_x\ge0 \quad\textrm{and} \quad R(x)\le0\qquad\Longleftrightarrow\qquad \Ric_x=0$$
for every $x\in M$, and 
$$\Ric_x\ge0 \ (\forall x)\quad\textrm{and} \quad \int_M R\,d\vol\le0\qquad\Longleftrightarrow\qquad \Ric_x=0 \ (\forall x).$$
For a synthetic characterization of scalar curvature and its integral, we propose to consider the  initial negative slope 
$\k(x):=\liminf_{s\searrow0} \frac1s\big( \a_0(x)-\a_s(x)\big)$ of the Gaussian volume functional
$$\a_s(x):=(12\pi s)^{-n/2}\int_{M} e^{-\frac{d^2(x,y)}{12s}}m(dy)$$
as well as the initial negative slope ${\KK}:=\liminf_{s\searrow0} \frac1s\big( {\AA}_0-{\AA}_s\big)$
of the Gaussian double integral
${\AA}_s:= \int_{M}\a_s(x)\, m(dx)$. In the smooth setting $\k$ and $\KK$ coincide with the scalar curvature and the total scalar curvature, respectively, see Theorem \ref{thm:smoothcase}. This suggests a possible definition of total scalar curvature for metric measure spaces, Definition \ref{def:curvaturefunctional}. In \cite{KLP21} the asymptotic behaviour of volume of balls is used to define the \textit{mm boudary} and the \textit{mm curvature} of an Alexandrov space.

In Section \ref{sec:examples} we compute $\k$ and $\KK$ for several spaces beyond smooth closed Riemannian manifolds; we report here the most interesting cases. Note that the proofs of the last two examples are already rather involved and can be found in the appendix.
\begin{itemize}
    \item In the case of weighted Riemannian manifold $(M,g,e^f d\vol_g)$ we get $\k(x)=(R(x)-3|\nabla f(x)|^2-3\Delta f(x))e^{f(x)}$, which we relate to Perelman's weighted scalar curvature $R-2\Delta f-|\nabla f|^2$, introduced in \cite{perelman2002entropy}, and studied in the context of spin geometry and mathematical relativity in \cite{BaldaufOzuch22, ChuZhu24}. 
    \item Let $\hat M$ be the gluing of two smooth manifolds along a common boundary (across which the metric will a priori be only Lipschitz). In Theorem \ref{thm:C0gluing} we show that $\widehat{\KK} =\int_{\hat M} R(x)d\vol(x)+2(n-1)\int_{\partial M} (H_1(z) + H_2(z))d\sigma(z)$, where $H_i$ are the mean curvatures with respect to the two different metrics. This corresponds both to an approximating procedure \cite{Miao02} and to a notion of distributional scalar curvature \cite{LeeLeFloch15} already considered in mathematical relativity, and related to the Brown-York mass \cite{BY91}.
    \item Finally, we consider cones. If the cone is of $\dim\geq 3$, $\KK$ does not see the tip singularity. For $\dim=2$, $\KK$ is positive (resp. negative) for a cone over a sphere of radius $r<1$ (resp. $r>1)$, as expected. However, the numerical values do not coincide with the one provided by the Gauss--Bonnet formula. This is the same phenomenon as \cite[Ex.~1.14]{KLP21}.
\end{itemize}

In Section \ref{sec:timedep} we study time-dependent metric measure spaces. In order to characterize Ricci flows, exploiting
$$\partial_t g_x \ge -2\Ric_x \ (\forall x)\quad\textrm{and} \quad \int_M \left[ \frac{1}{2}\tr(\partial_t g_x) +R \right] d\vol \leq 0\qquad\Longleftrightarrow\qquad \partial_t g_x=-2\Ric_x \ (\forall x)$$
we are left to give meaning to the integral expression. This is done similarly as above, but now the distance is varying in time: 
$$\a_{s}(t,x):=(12\pi s)^{-n/2}\int e^{-\frac{d^2_{t+s}(x,y)}{12s}}m_t(dy)\,.$$
Also in this setting we study the case of weighted Riemannian manifolds, and then some examples of non smooth time-dependent spaces, Section \ref{sec:time-dep-ex}.

In both the static and the time-dependent case, we also include a discussion where we consider a directional version of the Gaussian volume. This is a stronger object, that allows us to characterize the full tensor ($\Ric$ or $\partial_tg+2\Ric$), but which requires more structure on the underlying space (in particular the existence of a tangent space). This notion of Ricci curvature was discussed in \cite[Appendix 2]{CC97}, and studied in \cite{ZZ10} on Alexandrov spaces. Along the way, we also prove a characterization of super Ricci flows via a new time-dependent Laplace comparison, see Theorem \ref{thm:charactSRF}$(ii)$.

\paragraph{Acknowledgements}
This work was partially funded by the Deutsche Forschungsgemeinschaft (DFG, German Research Foundation), projects CRC 1720 – 539309657, EXC-2047/2 - 390685813, and SPP 2026 “Geometry at Infinity”.

\section{The Gaussian volume functional and (integral) scalar curvature}
\setstcolor{blue}

\subsection{The Gaussian volume functional and the Gaussian double integral}
Let $(X,d,m)$ be a metric measure space (mm-space), i.e. $(X,d)$ is a complete, separable metric space, and $m$ is a Radon measure on $\mathcal{B}(X)$, and for $A \in \mathcal{B}(X)$ we write $|A|:= m(A)$. We always assume that for all $x\in X$ there is $C>0$ s.t.
\begin{align}\label{eq:expgrowth}
    |B_r(x)| \leq e^{Cr^2}.
\end{align}
This assumption is for example satisfied on $\CD$ spaces (introduced in \cite{SturmActa1, SturmActa2}, \cite{LottVillani}), as well as when $m$ is finite, see also \cite[pag.~321]{AGS2} for a detailed discussion on this condition. 

\paragraph{Assumption:} For $n\in\N$ given, assume that for every $x\in X$ the limit
\begin{equation}\label{density}
\rho(x):=\lim_{r\searrow0}\frac{m(B_r(x))}{\omega_nr^n}\end{equation}
exists, is finite, and the ratio $\frac{m(B(r,x))}{\omega_nr^n}$ is bounded uniformly in $(x,r)$ for $x \in X$, $r \leq 1$. Here $\omega_n:=\frac{\pi^{n/2}}{\Gamma(n/2+1)}$ denotes the volume of the unit ball in $\R^n$.

Recall that on a $\CD(-K,n)$ space, by the Bishop-Gromov theorem, the ratio $\frac{|B_r(x)|}{|\mathbb B^{K,n}_r(o)|}$ is non increasing in $r$. In particular the limit in \eqref{density} exists
\begin{align}
    \lim_{r\searrow0}\frac{|B_r(x)|}{|\mathbb B^{K,n}_r(o)|} = \lim_{r\searrow0}\frac{|B_r(x)|}{\omega_n r^n} =:\rho(x)
\end{align}
but might be infinite. If we assume it to be bounded, then the ratio is uniformly bounded in $(x,r)\in X\times [0,1]$:
\begin{align}
    \frac{|B_r(x)|}{\omega_nr^n}  \leq C(\diam M,K,n)\frac{|B_r(x)|}{|\mathbb B^{K,n}_r(o)|} \leq C\lim_{r \searrow 0} \frac{|B_r(x)|}{|\mathbb B^{K,n}_r(o)|} = C\lim_{r \searrow 0} \frac{|B_r(x)|}{\omega_nr^n} = C\rho(x) \,.
\end{align}

Moreover, on a non collapsed $\RCD(K,n)$ space (i.e. with $m=\mathcal H^n$) $\rho(x)\leq1$ for every $x\in X$, see \cite[Cor.~2.14]{DePhilippisGiglincRCD}. From \cite{BGHZ23}, if on an $\RCD(K,n)$ space the $n$-density $\rho$ is finite on a set of positive measure, then $m=c\mathcal H^n$, $c\in \R_+$.

In section \ref{sec:examples} we will also study spaces that are not $\CD$, but that still satisfy \eqref{density}.

Under these assumptions we introduce the following quantities.
\begin{definition}
Define:
\begin{itemize}
    \item The \emph{Gaussian volume functional}
$$\a_s(x):=(12\pi s)^{-n/2}\int_{X} e^{-\frac{d^2(x,y)}{12s}}m(dy).$$

\item If $m$ is finite, The \emph{Gaussian double integral} 
$${\AA}_s:=(12\pi s)^{-n/2}\int_{X} \int_{X}e^{-\frac{d^2(x,y)}{12s}}m(dy) m(dx).$$
\end{itemize}
\end{definition}

For the case $s=0$ we make use of the following statements (they also show that, for $s$ small enough, the quantities above finite).
\begin{lemma}\label{lem:limitsto0}
Assume $(X,d,m)$ satisfies \eqref{density}. Then

(i) The limit
$\a_0(x):=
\lim_{s\searrow0}\a_s(x)$ exists and coincides with $\rho(x)$. 

(ii) Similarly, when $m$ is finite, $\AA_0:=\lim_{s\searrow0}\AA_s=\int_X\rho\,dm$.

(iii) Moreover,
$$\lim_{s\searrow 0}(12\pi s)^{-n/2}\frac1{6s}\int_X  d^2(x,y) e^{-\frac{d^2(x,y)}{12s}}m(dy) = n \cdot \rho(x).$$
\end{lemma}

\begin{proof}
(i) For fixed $s$ put $v(r):=m(B_r(x))$. Then, by integration by parts for the Stieltjes integrals, 
\begin{align*}
\lim_{s\searrow 0}&(12\pi s)^{-n/2}\int_X e^{-\frac{d^2(x,y)}{12s}}m(dy) \\
&=
\lim_{s\searrow 0}(12\pi s)^{-n/2}\int_0^\infty e^{-r^2/12s} dv'(r)\\
&=
\lim_{s\searrow 0} \lim_{R\to\infty}(12\pi s)^{-n/2}\int_0^R e^{-r^2/12s} dv'(r)\\
&= \lim_{s\searrow 0} \lim_{R\to\infty} (12\pi s)^{-n/2} \left[ e^{-r^2/12s} v(r) \bigg|_{r=0}^R + \int_0^R e^{-r^2/12s} \frac{2r}{12s}v(r) dr \right] \\
&= \lim_{s\searrow 0}(12\pi s)^{-n/2}\int_0^\infty e^{-r^2/12s} \frac{2r}{12s}v(r) dr
\end{align*}
since from \eqref{density} we have that $v(0)=0$, and from \eqref{eq:expgrowth} that $0\leq e^{-r^2/12s} v(r) \leq e^{-\frac{r^2}{12s} + Cr^2} \to 0$
as $R\to \infty$ for $s<\frac{1}{12C}$. Now thanks again to \eqref{eq:expgrowth}, we can just consider the expansion of $v(r)$ around 0 from \eqref{density}:
\begin{align*}
\lim_{s\searrow 0}&(12\pi s)^{-n/2}\int_0^\infty e^{-r^2/12s} \frac{2r}{12s}v(r) dr \\
&=\lim_{s\searrow 0}(12\pi s)^{-n/2}\int_0^\infty e^{-r^2/12s} \frac{2r}{12s} \left[\rho(x)\omega_nr^n+o(r^n)\right] dr \\
&= \rho(x) \,.
\end{align*}
In the last line we explicitly computed the Gaussian integral using the useful identity:
\begin{align}\label{eq:gaussianint}
    F(s):=(12\pi s)^{-n/2}\int_{0}^{\infty}\frac{r^{\alpha}}{s^{\beta}}e^{-r^{2}/(12s)}n\omega_{n}r^{n-1}\,dr=\frac{n}{2}\,12^{\alpha/2}s^{\alpha/2-\beta}\frac{\Gamma\!\left(\tfrac{n+\alpha}{2}\right)}{\Gamma\!\left(\tfrac{n}{2}+1\right)}\,.
\end{align}

(ii) Using boundedness of the ratio in \eqref{density}, with the same computations as above, $\a_s(x)\leq C$, $C$ independent of $x$. Then by finiteness of $m$ we use dominated convergence to pass the limit inside the integral:
\begin{align}
    \lim_{s\searrow0}\AA_s = \lim_{s\searrow0}\int_X\a_s(x)dm(x) = \int_X\rho(x)dm(x)\,.
\end{align}

(iii) As above,
\begin{align*}
\lim_{s\searrow 0}(12\pi s)^{-n/2}&\frac1{6s}\int  d^2(x,y) e^{-\frac{d^2(x,y)}{12s}}m(dy) =
\lim_{s\searrow 0}(12\pi s)^{-n/2}\frac1{6s}\int_0^\infty r^2 e^{-r^2/12s} dv'(r)\\
&=-\lim_{s\searrow 0}(12\pi s)^{-n/2}\frac1{6s}\int_0^\infty \Big[2r-\frac{2r^3}{12s}\Big] e^{-r^2/12s} v(r)dr\\
&=-\lim_{s\searrow 0}(12\pi s)^{-n/2}\frac1{6s}\int_0^\infty \Big[2r-\frac{2r^3}{12s}\Big] e^{-r^2/12s} [\omega_nr^n \rho(x) + o(r^n)]dr\\
&=n\ \rho(x)
\end{align*}
by using that
$$\limsup_{s\searrow 0}(12\pi s)^{-n/2}\frac1{6s}\int_0^\infty \Big|2r-\frac{2r^3}{12s}\Big| e^{-r^2/12s} r^n\,dr<\infty.$$
\end{proof}

We now define the principal geometric objects analysed in this paper.

\begin{definition}\label{def:curvaturefunctional}
Under \eqref{density}, define:
\begin{itemize}
    \item The \emph{initial negative slope} of $\a$:  
$$\k(x):=\liminf_{s\searrow0}\frac1s\Big( \a_0(x)-\a_s(x)\Big).$$

\item If $m$ is finite, the \emph{initial negative slope} of $\AA$:
\begin{align}\label{eq:defK}
    \KK :=\liminf_{s\searrow0} \frac1s\Big( {\AA}_0-{\AA}_s\Big),
\end{align}
called \emph{curvature functional}.
Note that all these quantities depend on $n$. To emphasize this dependence,
we occasionally write $\k^{(n)}$ and ${\KK}^{(n)}$.
\end{itemize}
\end{definition}

\begin{remark}\label{rmk:scalmeasure}
    Assume $X$ compact, for any $U\subseteq X$ Borel, let 
    \begin{align}
        \AA_s(U):=(12\pi s)^{-n/2}\int_{U} \int_{X}e^{-\frac{d^2(x,y)}{12s}}m(dy) m(dx)
    \end{align} and
    \begin{align}
    \KK_s(U)=\frac1s\Big( {\AA}_0(U)-{\AA}_s(U)\Big)\,.
    \end{align}
    Then $\KK_s$ are signed measures on $\mathcal B(X)$. A more refined definition than \eqref{eq:defK} requires that there exists a signed measure $\KK$ s.t. $\KK_s\rightharpoonup \KK$ weakly. In this case, by compactness, $\KK(X)$ is the number obtained in \eqref{eq:defK}.
\end{remark}

\begin{lemma}\label{upp/low}
(i) For $s>0$, the Gaussian volume functional is continuously differentiable and
$$     \partial_s \a_s(x)=
    (12\pi s)^{-n/2}\frac1{2s}\int \Big[-n+\frac1{6s}d^2(x,y) \Big]e^{-\frac{d^2(x,y)}{12s}}m(dy).$$
    
(ii) Moreover,
$$-\limsup_{s\searrow0} \partial_s\a_s(x)\le \k(x)\le -\liminf_{s\searrow0} \partial_s\a_s(x)\,.$$

(iii) Similarly if $m$ is finite,
$$-\limsup_{s\searrow0} \int_X\partial_s\a_s(x)dm(x)\le \KK\le -\liminf_{s\searrow0} \int_X \partial_s \a_s(x) dm(x)\,.$$
\end{lemma}

\begin{remark} The previous lemma will be useful to compute $\k(x)$ and $\KK$. Indeed, it shows that if $\lim_{s\to 0} \partial_s\a_s(x)$ exists, then $\k(x)=-\lim_{s\to 0} \partial_s\a_s(x)$ and if $\lim_{s\to 0} \int_X \partial_s\a_s(x) dm(x)$ exists, then $\KK=\lim_{s\to0} \int_X \partial_s \a_s(x) dm(x)$.
\end{remark}

\begin{proof} (i) Follows by differentiating, and by \eqref{eq:expgrowth} in order to to take the difference quotient limit under the integral.

(ii) Let us prove the lower estimate. Put $\k_*(x):=-\limsup_{s\searrow0} \partial_s\a_s(x)$. Then 
for all $\epsilon>0$ there exists $s>0$ such that:
\begin{itemize}
    \item $\k(x)+\epsilon\ge \frac1s (\a_0(x)-\a_s(x))$
    \item for all $r\in (0,s)$,
$ \partial_r\a_r(x) \le -{\k}_*(x)+\epsilon$.
\end{itemize}
Moreover, by continuity (Lemma \ref{lem:limitsto0}) there exists  $\delta_0<s$ such that $\a_\delta(x)-\a_0(x)\le\epsilon s$ for all $\delta\in (0,\delta_0)$.
Thus
\begin{align*}
\k(x)+\epsilon&\ge \frac1s \Big(\a_0(x)-\a_s(x)\Big)
\ge  \frac1s \Big(\a_\delta(x)-\a_s(x)\Big)-\epsilon\\
&=-\frac1s\int_\delta^s \partial_r\a_r(x)dr-\epsilon
\ge \Big(1-\frac\delta s\Big)\k_*(x)-2\epsilon.
\end{align*}
Letting $\delta\to0$ and then $\epsilon\to0$, the claim follows. The upper bound and (iii) follow analogously.
\end{proof}

\subsection{Scalar curvature}
We can now justify Definition \ref{def:curvaturefunctional}.

\begin{theorem}\label{thm:smoothcase}
Let $(M,g)$ be a smooth complete $n$-dimensional Riemannian manifold without boundary satisfying \eqref{eq:expgrowth}. Denote its scalar curvature by 
$R=R_g$ and put $d:=d_g, m:=\vol_g$. Then \[\k(x)=R(x)\]
and if $M$ is compact
\[ {\KK}=\int_M R\,d\vol_g.\]
\end{theorem}
\begin{proof}[Proof I (expanding the volume of spheres)]
Set $v(r)=m(B_r(x))$, recall from \cite{GrayVolume} that on a Riemannian manifold
\begin{align*}
    v'(r) = n\omega_n r^{n-1} \left[ 1 - \frac{R}{6n}r^2 + O(r^4)\right].
\end{align*}
Therefore
\begin{align}\label{eq:vol_and_scal_curv}
    \k(x) &= -\lim_{s\searrow0} \partial_s\a_s(x) \\
    &=-\lim_{s\searrow0}(12\pi s)^{-n/2}\frac1{2s}\int \Big[-n+\frac1{6s}d^2(x,y) \Big]e^{-\frac{d^2(x,y)}{12s}} d\vol(y) \\
    &= -\lim_{s\searrow0}(12\pi s)^{-n/2}\frac1{2s}\int_0^\infty \Big[-n+\frac{r^2}{6s} \Big]e^{-\frac{r^2}{12s}} n\omega_n r^{n-1} \left[ 1 - \frac{R(x)}{6n}r^2 + O(r^4)\right] dr \\
    &= \lim_{s\searrow0} \frac{n}{2s} - \frac{n}{2s} - \frac{nR(x)}{2}+\frac{(n+2)R(x)}{2} + O(s) \\
    &= R(x)
\end{align}
using the identity \eqref{eq:gaussianint}.

To prove the integral statement, we recall that the term $O(r^4)$ depends on the geometry of $M$ at $x$, which is uniformly bounded when $M$ is compact, hence the convergence above is uniform in $x$, and ${\KK}=\int_M R\,d\vol_g$.
\end{proof}

\begin{proof}[Proof II (expanding the Laplacian of the distance)]
By integration by parts once and $|\nabla d|^2=1$ a.e. we have
\begin{align}
    \int e^{-\frac{d^2(x,y)}{12s}}\frac{\Delta d^2}{2} d\vol(y) = \int e^{-\frac{d^2(x,y)}{12s}}
   \frac{d^2}{6s}d\vol(y),
\end{align}
and hence
\begin{align}\label{eq:IBP}
    \partial_s \a_s(x)&=
    (12\pi s)^{-n/2}\frac1{2s}\int_M \bigg[-n+\frac1{6s}d^2(x,y) \bigg]e^{-\frac{d^2(x,y)}{12s}}d\vol(y)\\
    &=(12\pi s)^{-n/2}\frac1{2s}\int_M \bigg[-n+\frac12\Delta d^2(x,y) \bigg]e^{-\frac{d^2(x,y)}{12s}}d\vol(y)\,.
\end{align}
Set
\begin{align}\label{eq:tau_tilde}
    \tilde \tau_{K,n}(r) =\begin{cases}
        1 + r\sqrt{K/(n-1)} \cot(r\sqrt{K/(n-1)}) & K>0 \\
        n & K=0 \\
        1 + r\sqrt{-K/(n-1)} \coth(r\sqrt{-K/(n-1)}) & K<0
    \end{cases}
\end{align}
then from the Taylor expansion of the metric \cite{GrayVolume} we get
\begin{align}\label{eq:laplcomp}
    \frac{1}{2}\Delta_y d^2(x,y) &= n\tilde\tau_{K,n}(d(x,y)) + O(d^3(x,y)) \\
    &= n - \frac{K_{x,y}}{3}d^2(x,y) + O(d^3(x,y))
\end{align}
where $K_{x,y}=\Ric_x(\dot\gamma_0,\dot\gamma_0)$, with $\gamma$ the unit speed geodesic from $x$ to $y$, where the error term is a signed measure whose total mass is $O(d^3)$. Then
\begin{align}\label{eq:GaussintwithLaplcomp}
    -\partial_s \a_s(x)
    &=(12\pi s)^{-n/2}\frac1{2s}\int_M \bigg[\frac{K_{x,y}}{3}d^2(x,y) + O(d^3(x,y)) \bigg]e^{-\frac{d^2(x,y)}{12s}}d\vol(y) \\
    &= (12\pi s)^{-n/2}\frac1{2s}\int_0^\infty \int_{S^n_xM} \bigg[\frac{\Ric_x(\theta)}{3}\rho^2 + O(\rho^3) \bigg]e^{-\frac{\rho^2}{12s}}\left(\rho^{n-1}+O(\rho^n)\right) d\sigma(\theta) d\rho \\
    &= (12\pi s)^{-n/2}\frac1{2s}\int_0^\infty \left[\frac{R_x}{3n}\rho^{n+1} n \omega_n + O(\rho^{n+2})\right] e^{-\frac{\rho^2}{12s}} d\rho \\
    &= R_x + O(\sqrt s)
\end{align}
where we used \eqref{eq:gaussianint} and $\int_{S^n_xM} \Ric_x(\theta) d\theta = \frac{R_x}{n}|S^n| = R_x\omega_n$ by properties of bilinear forms. Also in this case when $M$ is compact the error $O(\sqrt s)$ is uniform in $x$ and we can integrate.
\end{proof}

The function $\k$ provides a synthetic way to define the scalar curvature.
For singular spaces, however, the quantity $\KK$ turns out to provide more detailed information on the ``singular contributions to the curvature''. We will see in numerous non smooth examples that
${\KK}\not= \int_M \k\,dm$ (because dominated convergence is violated).

We can now characterize Einstein manifolds.
\begin{corollary}\label{cor:minimal} Assume  $(M,g)$ is a closed Riemannian manifold with  $\Ric\ge K$. Then $M$ is Einstein, i.e. $\Ric=K g$ if and only if $M$ is \emph{minimal} in the sense that
$$\frac1{|M|}{\KK}\le n K.$$
\end{corollary}

\subsection{The directional Gaussian volume}
A directional version of the Gaussian volume functional also recovers the Ricci curvature (rather than scalar curvature).

Let $(M,g)$ be a complete $n$-dimensional Riemannian manifold without boundary. For $x\in M$, let $\delta_x>0$ be so that $\exp_x:(0,\delta_x)\times S_xM\to M$ is injective.
Given an open subset $Z\subset S_xM$ with smooth boundary, define the set 
$$Z(x):=\Big\{ \exp_x(rz): z\in Z, r\in [0,\delta_x]\Big\},$$
and put
$$\a_s(x, Z):=(12\pi s)^{-n/2}\int_{Z(x)} e^{-\frac{d^2(x,y)}{12s}}m(dy)$$
as well as $$\k(x,Z):=\liminf_{s\to0} \frac1s\Big( \a_0(x, Z)-\a_s(x, Z)\Big).$$

\begin{remark}
    In this case the Gaussian integral is performed on the bounded region $Z(x)$, hence the sub-exponential volume growth assumption \eqref{eq:expgrowth} is not necessary.
\end{remark}

\begin{theorem}\label{k-ric} For  $(M,g)$ complete $n$-dimensional Riemannian manifold without boundary, it holds
$$ \k(x,Z) = \frac 1{\omega_n}\int_{Z}\Ric(z)\,d\sigma(z).$$
\end{theorem}
\begin{proof}
    The proof is similar to Theorem \ref{thm:smoothcase}. In particular we follow Proof II, as this approach will be useful in the next theorem. Integration by parts yields
    \begin{align}
        \int_{Z(x)} e^{-\frac{d^2(x,y)}{12s}}\frac{\Delta d^2}{2} d\vol(y) = \int_{Z(x)} e^{-\frac{d^2(x,y)}{12s}}
   \frac{d^2}{6s}d\vol(y) + \int_{\partial Z(x)}e^{-\frac{d^2(x,y)}{12s}}\frac{\nabla d^2(x,y)}{2} \cdot \nu(y) d\sigma(y)
    \end{align}
    Then 
    \begin{align}
        -\partial_s \a_s(x)&=
    (12\pi s)^{-n/2}\frac1{2s}\int_{Z(x)} \bigg[n-\frac1{6s}d^2(x,y) \bigg]e^{-\frac{d^2(x,y)}{12s}}d\vol(y)\\
    &=(12\pi s)^{-n/2}\frac1{2s}\int_{Z(x)} \bigg[n-\frac12\Delta d^2(x,y) \bigg]e^{-\frac{d^2(x,y)}{12s}}d\vol(y) \\
    &\qquad + (12\pi s)^{-n/2}\frac1{4s}\int_{\partial Z(x)} e^{-\frac{d^2(x,y)}{12s}} \nabla d^2(x,y)\cdot \nu(y) d\sigma(y) \\
    &= I + II\,.
    \end{align}
    Using \eqref{eq:laplcomp}, with $K_{x,y}=\Ric_x(\dot\gamma_0,\dot\gamma_0)$, $\gamma$ the unit speed geodesic from $x$ to $y$, the first integral gives the claimed quantity
    \begin{align}
        I
    &=(12\pi s)^{-n/2}\frac1{2s}\int_{Z(x)} \bigg[\frac{K_{x,y}}{3}d^2(x,y) + O(d^3(x,y)) \bigg]e^{-\frac{d^2(x,y)}{12s}}d\vol(y) \\
    &= (12\pi s)^{-n/2}\frac1{2s}\int_0^{\delta(x)} \int_{Z} \bigg[\frac{\Ric_x(z)}{3}\rho^2 + O(\rho^3) \bigg]e^{-\frac{\rho^2}{12s}}\left(\rho^{n-1}+O(\rho^n)\right) d\sigma(z) d\rho \\
    &= \left[\int_Z\Ric_x(z)d\sigma(z) \right] (12\pi s)^{-n/2}\frac1{2s}\int_0^\infty \left[\frac{1}{3}\rho^{n+1} + O(\rho^{n+2})\right] e^{-\frac{\rho^2}{12s}} d\rho \\
    &= \frac{1}{\omega_n}\int_Z\Ric_x(z)d\sigma(z) + O(\sqrt s)\,.
    \end{align}
    On the other hand, let $\partial Z(x) = L\sqcup C$, where $L = \{\exp_x(tz), \ t\in[0,\delta_x), \ z\in Z \}$ and $C = \{ \exp_x(\delta_x z) , \ z\in \overline{Z} \}$. Then $\nabla d\cdot \nu = \ind _{C}$ on $\partial Z(x)$ a.e. and
    \begin{align}
        II = (12\pi s)^{-n/2}\frac1{4s}|C| e^{-\frac{\delta_x^2}{12s}} 2\delta_x =O(s)\,.
    \end{align}
\end{proof}

\begin{theorem}[Novel Characterizations of Ricci lower bounds] For $(M,g)$ complete $n$-dimensional Riemannian manifold without boundary and a number $K\in\R$, the following are equivalent:
\begin{enumerate}[(i)]
\item Lower Ricci bound: $\Ric_x\ge K g_x$ for all $x$
\item Bakry-\'Emery condition $\BE(K,n)$
\item Curvature-dimension condition $\CD(K,n)$
\item Measure contraction property $\MCP(K,n)$
\item Sharp Laplace comparison: for all $x$ and a.e. $y\in M$,
$$\frac12\Delta d^2(x,y) \le n\tilde \tau_{K,n}(d(x,y))\,,
$$
where $\tilde \tau_{K,n}$ was defined in \eqref{eq:tau_tilde}.
\item Weak Laplace comparison: for all $x$ and a.e. $y\in M$,
$$\frac12 \Delta_y d^2(x,y)\le  n-\frac13 K d^2(x,y)$$
\item Perceived lower Ricci bound: for all $x$ and $Z$,
$$\k(x,Z)\ge {Kn}\,\frac{|Z|}{|S_x|} = K\frac{|Z|}{\omega_n}.$$
\end{enumerate}

\end{theorem}
\begin{proof} Equivalences $(i)-(iv)$ follow from \cite{BakryEmeryDiffHypercontr, sturmvonrenesse, SturmActa2}, $(iv)\implies (v)$ from  \cite{giglionthediffstructure}. 

$(v)\implies (vi).$ It follows by noticing that $n\tilde \tau_{K,n}(d(x,y)) \leq n - \frac{1}{3}Kd^2(x,y)$.

$(vi)\implies (vii).$ From the proof of Theorem \ref{k-ric}, 
\begin{align}
    -\partial_s \a_s(x,Z) &=(12\pi s)^{-n/2}\frac1{2s}\int_{Z(x)} \bigg[n-\frac12\Delta d^2(x,y) \bigg]e^{-\frac{d^2(x,y)}{12s}}d\vol(y) + O(s) \\
    &\geq (12\pi s)^{-n/2}\frac1{2s}\int_{Z(x)} \frac{1}{3}Kd^2(x,y)e^{-\frac{d^2(x,y)}{12s}}d\vol(y) + O(s) \\
    &= K\frac{|Z|}{\omega_n} +O(\sqrt s)\,.
\end{align}

$(vii)\implies (i).$ Consider $\theta\in S_xM$, $Z_\eps := B_\eps(\theta)\subset S_x M$, and $Z_\eps (x)=\{ \exp_x(rz): z\in B_\eps(\theta), r\in [0,\delta_x]\}$. Then by Theorem \ref{k-ric},
\begin{align}
    K \leq \frac{1}{|Z_\eps|}\int_{Z_\eps}\Ric(z)d\sigma(z) \to \Ric(\theta) \qquad \text{as $\eps\to0$.}
\end{align}
\end{proof}
 
\section{Examples of Gaussian volume functional}\label{sec:examples}
 
\subsection{Weighted Riemannian manifolds}\label{subsec:weighted}
Let $(M,g)$ be a smooth, complete $n$-dimensional Riemannian manifold without boundary and put $d=d_g$,  $dm=\rho\,d\vol_g$ for some $\rho=e^f\in C^2(M)$. The triple $(M,g,m)$ is called a {\it weighted Riemannian manifold}. On these spaces the weighted Laplacian $\Delta_{f} u=\Delta u+\nabla f\nabla u$ is formally self-adjoint on $L^2(dm)$. Moreover, a lower bound on the weighted Ricci curvature 
$\Ric^{f}=\Ric- \Hess f$ 
ensures that $\Delta_f$ enjoys similar analytic properties to the classical Laplacian 
$\Delta$ under a lower Ricci bound.

Note that in this setting Equation \eqref{density} is satisfied, in particular
\begin{align}
    \a_0(x)=\rho(x) \qquad \text{and} \qquad {\AA}_0=\int \rho^2d\vol\,.
\end{align}

\begin{theorem}\label{thm:weightedRm}
Let $(M,g,m=\rho d\vol_g)$ be a weighted smooth complete $n$-dimensional Riemannian manifold without boundary satisfying \eqref{eq:expgrowth}, set $\rho=e^f$, then
\begin{align}
{\k}^{(n)}(x)
=R(x)\rho(x)-3\Delta \rho(x) = (R(x)-3|\nabla f(x)|^2-3\Delta f(x))e^{f(x)} 
\end{align}
and if $M$ is closed
\begin{align}
{\KK} =
\int \left[\rho^2 R+  3|\nabla \rho|^2\right]d\vol_g = \int\left[R+3|\nabla f|^2\right]e^{2f}d\vol_g\,.
\end{align}
\end{theorem}

\begin{proof}
    Fix $x$, consider $y$ in a neighbourhood of $x$ with $\gamma$ unit speed geodesic from $x$ to $y=\gamma_r$. Combining the classical Laplacian expansion \eqref{eq:laplcomp} with 
    \begin{align}
        \frac{1}{2}\left\langle \nabla f(y),\nabla_yd^2(x,y)\right\rangle = d(x,y) \frac{d}{dr}f(\gamma_r)= \nabla_{\dot\gamma_0}f(x) d(x,y) + \Hess f(x)(\dot\gamma_0,\dot\gamma_0)d^2(x,y) + O(d^3)
    \end{align}
    we obtain
    \begin{align}
    \frac{1}{2}(\Delta_f)_y d^2(x,y)
    &= n + \nabla_{\dot\gamma_0}f(x) d(x,y) - \frac{1}{3}(\Ric -3\Hess f)(\dot\gamma_0,\dot\gamma_0)d^2(x,y) + O(d^3)
\end{align}
Then
\begin{align}
    &-\partial_s \a_s(x) \\
    &=
    (12\pi s)^{-n/2}\frac1{2s}\int_M \bigg[n-\frac1{6s}d^2(x,y) \bigg]e^{-\frac{d^2(x,y)}{12s}}e^{f(y)} d\vol(y)\\
    &=(12\pi s)^{-n/2}\frac1{2s}\int_M \bigg[n-\frac12\Delta_f d^2(x,y) \bigg]e^{-\frac{d^2(x,y)}{12s}} e^{f(y)} d\vol(y) \\
    &=(12\pi s)^{-n/2}\frac1{2s}\int_M \bigg[-\nabla_{\dot\gamma_0}f(x) d(x,y) + \frac{1}{3}(\Ric -3\Hess f)(\dot\gamma_0,\dot\gamma_0)d^2(x,y) + O(d^3(x,y)) \bigg]\\ & \hspace{5cm} \times e^{-\frac{d^2(x,y)}{12s}}\left[e^{f(x)}+e^{f(x)}\nabla_{\dot\gamma_0} f(x)d(x,y) + O(d^2(x,y))\right]d\vol(y) \\
    &= (12\pi s)^{-n/2}\frac1{2s}\int_0^\infty \int_{S^n} \bigg[-\nabla_{\theta}f(x) r + \frac{1}{3}(\Ric -3\Hess f-3\nabla f\otimes\nabla f)(\theta,\theta)r^2 + O(r^3) \bigg]e^{f(x)} \\
    &\hspace{11cm} \times e^{-\frac{r^2}{12s}}e^{f(x)}r^{n-1} d\sigma(\theta) dr \\
    &= (12\pi s)^{-n/2}\frac1{2s}\int_0^\infty \left[\frac{\tr (\Ric -3\Hess f -3\nabla f\otimes \nabla f)(x)}{3n}e^{f(x)}r^{n+1} n \omega_n + O(r^{n+2})\right] e^{-\frac{r^2}{12s}} dr \\
    &= (R(x)-3\Delta f(x) -3|\nabla f|^2)e^{f(x)} + O(\sqrt s)
\end{align}
Compactness ensures uniformity of the error in $x$ and the integral statement follows after integrating by parts.
\end{proof}
 
\begin{remark} 
\begin{enumerate}[(i)]
\item Denote with $R^\rho=R-\Delta \log\rho$ the trace of the Bakry-\'Emery Ricci tensor, then 
\[ {\KK}^{(n)} =\int \left[\rho^2 R^\rho + |\nabla \rho|^2\right]d\vol_g.
\]
\item On the other hand, in \cite{perelman2002entropy} Perelman considered the weighted scalar curvature 
\begin{align}
\widetilde R^{e^f} :=R-2\Delta f-|\nabla f|^2 = \widetilde R^\rho= R + \frac{|\nabla \rho|^2}{\rho^2}- 2\frac{\Delta\rho}{\rho}
\end{align}
which solves the weighted Bianchi identity $\Div_f \Ric^{f} = \frac{1}{2}\nabla\widetilde R^{e^f}$. Then one can see that
\begin{align}
   \int \widetilde R^\rho \rho^2 d\vol_g = \int (R\rho^2 + |\nabla \rho|^2 - 2\rho\Delta\rho)d\vol_g = \int (R\rho^2 + 3|\nabla\rho|^2)d\vol_g = {\KK} 
\end{align}
This notion of weighted scalar curvature was also used in \cite{BaldaufOzuch22,ChuZhu24} in relation to a (weighted) positive mass theorem.

\item Moreover, let 
\begin{align} 
R^{\rho,N} &= \tr (\Ric -\Hess f - \frac{1}{N-n}\nabla f\otimes\nabla f)=R-\Delta f -\frac{1}{N-n}|\nabla f|^2  \\
&= R-\frac{\Delta \rho}{\rho} + \left(1-\frac{1}{N-n}\right)\frac{|\nabla\rho|^2}{\rho^2}
\end{align}
denote the trace of the Bakry-\'Emery $N$-Ricci tensor, then
\begin{align}
    \int R^{\rho,N}\rho^2d\vol_g = \int \left[ R\rho^2 +\left(2-\frac{1}{N-n}\right) |\nabla\rho|^2 \right] d\vol_g \,.
\end{align}
In particular, if we choose $N=n-1$:
\begin{align}
    \KK^{(n)} = \int R^{\rho,n-1}\rho^2 d\vol_g\,.
\end{align}
Note however that the choice $N=n-1$ is neither in the classical regime $\{N\geq n\}$, nor in the negative one $\{N<0\}$ studied by Ohta \cite{OhtaNegative}.
\end{enumerate}
\end{remark}

\begin{corollary} Assume  $(M,g, \rho)$ is a closed weighted $n$-dimensional Riemannian manifold with $\Ric\ge 0$ or with $\Ric^\rho\ge 0$ and that is \emph{minimal} in the sense that
${\KK}^{(n)}\le 0$.
Then $\rho$ is constant and $M$ is \emph{Ricci flat}: $$\Ric_x=0 \quad\text{for all $x$.}$$
\end{corollary}

\begin{example}
    Consider the Gaussian space $(\R^n,d_E,e^{-|x^2|/2}d\mathcal L^n)$, which satisfies $\CD(1,\infty)$. We compute
    \begin{align}
        \k^{(n)} = -3e^{-|x|^2/2}(|x|^2-n)
    \end{align}
    and
    \begin{align}
        \KK^{(n)} = \frac{3n}{2}\pi^{n/2}\,.
    \end{align}
\end{example}

\subsection{RCD Spaces}
We now consider the case when $(X,d,m)$ is an $\RCD(K,n)$ space, \cite{AGS1,AGS2,AGS3,EKS}. Recall the discussion after \eqref{density} for the existence of the density $\rho(x)$.
\begin{theorem} Let $(X,d,m)$ be a noncollapsed $\RCD(K,n)$ space. Then
$$ {\k}^{(n)}(x)
\ge nK\rho(x). $$
If $m=c \mathcal{H}^n$ is finite, then
$$ {\KK}^{(n)}\ge n K \int_X \rho dm =nKc^2\mathcal H^n(X) \,.
$$
\end{theorem}
\begin{proof}
Recall that on these spaces the Laplace comparison in distributional sense holds \cite{giglionthediffstructure}:
\begin{align*}\frac12\Delta d^2(x,y) &\le n\tilde \tau_{K,n}(d(x,y)) \\
&\le
1+(n-1)\Big(1-\frac13 K/(n-1)d^2(x,y)\Big)\\
&=n-\frac13 K d^2(x,y)
\end{align*}
Using integration by parts as in \eqref{eq:IBP},
\begin{align}
     \partial_s \a_s(x)&=
    (12\pi s)^{-n/2}\frac1{2s}\int \bigg[-n+\frac1{6s}d^2(x,y) \bigg]e^{-\frac{d^2(x,y)}{12s}}m(dy)\\
    &=(12\pi s)^{-n/2}\frac1{2s}\int \bigg[-n+\frac12\Delta d^2(x,y) \bigg]e^{-\frac{d^2(x,y)}{12s}}m(dy)\\
    &\le-\frac K3 \, (12\pi s)^{-n/2}\frac1{2s}\int  d^2(x,y) e^{-\frac{d^2(x,y)}{12s}}m(dy) \\
   &=-K \,\left[n \rho(x)+ o(1)\right]. \label{eq:AforRCD}
   \end{align}
In the non compact case the chain rule and the integration by parts is justified by considering cut-off functions on an exhaustion of the space $X$, see \cite[Lem.~6.7]{AMS16}.
The inequality follows from the Laplace comparison and the last line from Lemma \ref{lem:limitsto0} $(iii)$. Using Lemma \ref{upp/low} we obtain the conclusion.

The integrated statement is justified by the uniformity in $x\in X$ of $ o(1)$ in \eqref{eq:AforRCD}, and boundedness of $\rho$. The form $m=c\mathcal{H}^n$ is guaranteed by \cite{BGHZ23}.
\end{proof}

\begin{definition}\label{def:minimalRCD}
    A noncollapsed $\RCD(K,n)$ space is said minimal if $\KK^{(n)} = nK\int_X \rho dm$.
\end{definition}

\subsection{Manifolds with boundary, doubling, tripling etc.}
Consider a smooth Riemannian manifold $(M,g)$ with non empty boundary $\partial M$. One can consider, for $k\in \N_0$ the space $\widehat M^{\cup k}$ by gluing $k$ copies of $M$. This results in a topological space (a differentiable manifold with nonempty boundary iff $k=1$ and a differentiable manifold without boundary iff $k=2$) on which we can define a metric $d$ and a measure $m$.

Note that for any $n$, assumption \eqref{density} is satisfied ($\rho=1$ for interior points, and $\rho=\frac{k}{2}$ for points on $\partial M$), and $\widehat{\AA}_0 =k|M|$.

\subsubsection{Manifolds with boundary}
In the following theorem, the $\sqrt s$ term should be compared with the notion of \textit{mm boundary}, see \cite[Ex.~1.4]{KLP21}.
\begin{theorem}\label{thm:boundary}
Let $M$ be a compact smooth manifold with boundary. Then
 $${\AA}_s= |M| -\sqrt{\frac{3s}\pi} |\partial M|+O(s).$$
In particular, $\KK=+\infty$, but $\int \k d\vol = \int_{M\setminus\partial M}Rd\vol$.
\end{theorem}

\begin{proof}
    We integrate by parts as in \eqref{eq:IBP}, but taking care of the boundary term
    \begin{align}
    \int_M\partial_s \a_s(x) d\vol(x)&= \int_M
    (12\pi s)^{-n/2}\frac1{2s}\int_M \bigg[-n+\frac1{6s}d^2(x,y) \bigg]e^{-\frac{d^2(x,y)}{12s}}d\vol(y)d\vol(x)\\
    &= \int_M (12\pi s)^{-n/2}\frac1{2s}\int_M \bigg[-n+\frac12\Delta d^2(x,y) \bigg]e^{-\frac{d^2(x,y)}{12s}}d\vol(y)d\vol(x) \\
    &\qquad - \int_M (12\pi s)^{-n/2}\frac1{2s} \int_{\partial M} d(x,y)\langle \nabla_y d(x,y),\nu\rangle e^{-\frac{d^2(x,y)}{12s}} d\sigma(y)d\vol(x) \\
    &= I_1 + I_2
    \end{align}
    with $\nu$ the outer unit normal and $\sigma$ the volume measure on $\partial M$.
    
First, proceeding as in \eqref{eq:GaussintwithLaplcomp}
\begin{align}
    I_1 = -\int_M R(x)d\vol(x)=O(1).
\end{align}

Writing the second integral as (for any $r>0$)
\begin{align}
    I_2 = -\int_{\partial M} (12\pi s)^{-n/2}\frac{1}{2s}\int_{B_r(y)}d(x,y) \langle \nabla_y d(x,y),\nu\rangle e^{-\frac{d^2(x,y)}{12s}} d\vol(x) d\sigma(y) + O(s)
\end{align}
we would like to pass to polar coordinates around a point $y\in \partial M$. However, the exponential map $\exp_y:V\subset HB_r(0)\to M$, with $HB_r(0):=HS^{n-1}\times [0,r)$ the half ball on the tangent space, $HS^{n-1} :=\{\theta\in S^{n-1} : \langle \theta, \nu\rangle \leq 0\}$, $V$ the biggest subset of $HB_r(0)$ where $\exp_y$ is defined, does not in general provide a diffeomorphism to a neighbourhood of $y$ in $M$. Nevertheless, letting $\phi:B_r(y)\to U\subset H\R^n$, $0\in U$, be a chart around $y$, we consider $\Delta := U\setminus \phi(\exp_y(HB_r(0)))$. Because $\phi(\exp_y(V))$ can be seen as the upper graph of a smooth function $f:\R^{n-1}\to \R$ with $\nabla f=0$, $|\Delta\cap \partial B_r(0)|=O(r^n)$. Similarly, an error of the same order is made if one extends $\phi\circ\exp_y$ from $V$ to a smooth diffeomorphism from $HB_r(0)$ to $\tilde U$, $(\phi\circ \exp_y)(V) \subset \tilde U \subset \R^{n}$, and extends smoothly the integrand on this region. In these coordinates $\theta=-\nabla_yd(x,y)$, hence
\begin{align}
    I_2 &= -\int_{\partial M} (12\pi s)^{-n/2}\frac{1}{2s} \int_0^r \int_{HS^{n-1}} \rho \langle -\theta,\nu\rangle e^{-\frac{\rho^2}{12s}} \left[d\sigma_\rho(\theta) + O(\rho^n)\right]d\rho d\sigma(y) + O(s) \\
    &= \int_{\partial M} (12\pi s)^{-n/2}\frac{1}{2s} \int_0^r \int_{HS^{n-1}} \rho \langle \theta,\nu\rangle e^{-\frac{\rho^2}{12s}} \left[\rho^{n-1} d\vol_{S^{n-1}}(\theta) + O(\rho^n)\right]d\rho d\sigma(y) + O(s) \\
    &= \int_{\partial M} (12\pi s)^{-n/2}\frac1{2s} \int_0^\infty \int_{HS^{n-1}} \rho \langle \theta,\nu\rangle e^{-\frac{\rho^2}{12s}} \rho^{n-1}d\vol_{S^{n-1}}(\theta)d\rho +O(1)
\end{align}

We compute, for $n\geq 2$,
\begin{align}
    \int_{HS^{n-1}} \langle \theta, e_n\rangle dg_{S^{n-1}} &= \int_0^{\pi/2} \int_{S^{n-2}}-\cos\varphi(\sin{\varphi})^{n-2} dg_{S^{n-2}}d\varphi
    \\
    &= -\frac{|S^{n-2}|}{n-1} \,.
\end{align}
Then
\begin{align}
    I_2 &= -|\partial M|(12\pi s)^{-n/2}\frac1{2s} \int_0^\infty \frac{|S^{n-2}|}{n-1}\rho e^{-\frac{\rho^2}{12s}} \rho^{n-1}d\rho +O(1) \\
    &= -|\partial M|\frac{|S^{n-2}|}{(n-1)|S^{n-1}|} \frac{n}{4}\sqrt{12}\frac{1}{\sqrt s}\frac{\Gamma(\frac{n+1}{2})}{\Gamma(\frac{n+2}{2})} +O(1)\\
    &= -|\partial M|\sqrt\frac{3}{4\pi} \frac{1}{\sqrt{s}}+O(1).
\end{align}
Integrating in $s$ we get
\begin{align}
    \AA_s = |M|-|\partial M|\sqrt\frac{3s}{\pi} + O(s)\,.
\end{align}
\end{proof}

\subsubsection{Doubling} 
\begin{theorem}\label{thm:doubling}  Let  $(M,g)$ be a compact Riemannian manifold with smooth boundary.  
Assume that $(\hat M, \hat g)$ is obtained by gluing two copies of $(M,g)$ along the boundary $\partial M$.
Then
$$\widehat{\KK}
=2\,\int_M R(x)d\vol(x)+4(n-1)\, \int_{\partial M} H(z)d\sigma(z)$$
where $H(z)$ is the mean curvature of $\partial M$ at $z\in\partial M$ with respect to the ``outward''-pointing unit normal (in our convention  $(n-1)H$ is the trace of the second fundamental form).
 
Instead, $\int_{\hat M}\hat\k(x) dm(x)=2\,\int_M R(x)d\vol(x)$.
\end{theorem}

\begin{proof}
The theorem is a special case of Theorem \ref{thm:C0gluing}, see Appendix \ref{app:C0} for the proof.
\end{proof}

We now show that the value of $\widehat{\KK}$ obtained above is the limit of $\widehat{\KK}_\epsilon$ coming from a suitable smooth sequence of $M_\epsilon$ approximating $M$. We first give two easy examples for simplicity: the case of surfaces, which only uses Gauss-Bonnet formula, and the case of $\mathbb B^n$. We then treat the most general case relying on a construction of Miao \cite{Miao02}, see also \cite{Perelman97}. 

\begin{example} Suppose $(M,g)$ is a compact two-dimensional Riemannian manifold with boundary, and that $(\hat M, \hat g_\epsilon)$, $\epsilon>0$, is a family of Riemannian surfaces obtained from the doubling $(\hat M, \hat g)$ by smoothing the metric in the neighborhood of the rim without changing the topology. By the Gauss-Bonnet Theorem applied to $\widehat M$, denoting with $\chi(M)$ the Euler characteristic, for all $\epsilon$
$$
\widehat{\KK}_\epsilon = \int_{\hat M} R_\eps d\vol_\eps = 4\pi\, \chi(\hat M)\,.
$$
On the other hand, by Theorem \ref{thm:doubling} and then Gauss-Bonnet applied to $M$
$$
\widehat{\KK} = 2 \left( \int_M Rd\vol + 2\int_{\partial M}k_{\partial M}(z) d\sigma \right) = 8\pi\chi(M)
$$
Since $\chi(\hat M) = \chi (M_1 \cup_{\partial M} M_2) = \chi(M_1)+\chi(M_2)-\chi(\partial M) = 2\chi(M)$ (using $\chi(\partial M) = \chi(\bigsqcup_{j = 1}^NS^1) = 0$), we have $\widehat{\KK}_\eps =\widehat{\KK}$.
\end{example}

 \begin{example} Let $(M,g)$ be the closed ball of radius $\rho$ in $\R^n$ and let $(\hat M,\hat g)$ be its doubling which can be represented as warped product $[-\rho,\rho]\times_f {\mathbb S}^{n-1}$
with warping function
$f(r)=\rho-|r|$.
Approximate $(\hat M,\hat g)$ by $C^2$ Riemannian manifolds $(\hat M_\epsilon, \hat g_\epsilon)$ given as warped products $[-\rho,\rho]\times_{f_\epsilon} M$
with 
$$f_\epsilon(r)=\begin{cases}\rho-|r|, &\qquad |r|\in [\epsilon\frac\pi2,\rho]\\
\epsilon\cos\big(r/\epsilon\big)+\rho-\epsilon\frac\pi2, &\qquad |r|\in [0,\epsilon\frac\pi2].
\end{cases}$$
Let $\widehat{\KK}_\epsilon$ be the curvature functional for $(\hat M_\epsilon, \hat g_\epsilon)$. Then
$$R_\epsilon(x)=-2(n-1)\frac{f''_\epsilon}{f_\epsilon}(r)+(n-1)(n-2)\frac{1-|f'_\epsilon|^2}{f_\epsilon^2}(r)$$
and
\begin{align*}
\int_{\hat M_\epsilon}R_\epsilon(x)\,dm_\epsilon(x)&=n\omega_n\,
\int_{-\epsilon\frac\pi2}^{\epsilon\frac\pi2}\bigg[-2(n-1)\frac{f''_\epsilon}{f_\epsilon}(r)+(n-1)(n-2)\frac{1-|f'_\epsilon|^2}{f_\epsilon^2}(r)
\bigg]f_\epsilon^{n-1}(r)dr\\
&=n\omega_n\,
\int_{-\epsilon\frac\pi2}^{\epsilon\frac\pi2}\frac{2(n-1)}{\epsilon\rho}\cos\big(r/\epsilon\big)
\rho^{n-1}dr+O(\epsilon) \\
&=4n(n-1)\omega_n\rho^{n-2}+O(\epsilon) \\
&= \widehat{\KK} + O(\epsilon)
\end{align*}
\end{example}

\begin{proposition}
    Consider $(M,g)$ smooth, compact with nonempty boundary. We can construct a family $(\hat M_\epsilon, \hat g_\epsilon)$, $\epsilon>0$ of smooth Riemannian manifolds obtained from the doubling $(\hat M, \hat g)$ by smoothing the metric in the neighborhood of the rim, so that \begin{itemize}
        \item $(\hat M_\epsilon, \hat g_\epsilon) \to (\hat M, \hat g)$ in $C^0$;
        \item $\widehat{\KK}_\epsilon \to \widehat{\KK}=2\int R d\vol + 4(n-1)\int H d\sigma$.
    \end{itemize}
\end{proposition}
\begin{proof}
    See Proposition \ref{prop:gluingapprox}.
\end{proof}

\begin{example}
Let $F$ be the closure of an open connected  set in $\R^n, n\ge2$, with smooth boundary and $\emptyset\not= F\not=\R^n$. Let $d$ be the intrinsic distance induced by the euclidean metric and $m$ the Lebesgue measure on $F$. Let  $(\hat F,\hat d,\hat m)$ denote the doubling of  $(F,d,m)$ along $\partial F$.

(i) The space  $(\hat F,\hat d,\hat m)$ is $\RCD$ if and only if $F$ is convex, in which case it is $\RCD(0,n)$.

(ii) For $F$ convex, the space  $(\hat F,\hat d,\hat m)$ is minimal , in the sense of $\KK=0$, if and only if $F$ is a halfspace or a strip, that is, isometric either to $\R_+\times \R^{n-1}$ or to $[0,L]\times \R^{n-1}$ for some $L>0$.
\end{example}

\begin{proof}
Claim (i) follows from \cite{KapKetSt}. For the second claim, we first clarify that for spaces of infinite mass, minimality corresponds to $\KK(U)=0$ for all bounded $U$, see Remark \ref{rmk:scalmeasure}. Fixing such a $U$ we then have $0\geq \KK(U)=4(n-1)\int_{U \cap\partial F}Hd\sigma$ and by convexity $H\equiv0$. But again by convexity $\partial F$ must have second fundamental form $II\equiv 0$, hence it is a hyperplane. The only possibilities left are a halfspace or a strip.
\end{proof}

\subsubsection{Gluing multiple copies}   
\begin{theorem}\label{thm:multiple-copies}  Let  $M$ be a smooth, compact Riemannian manifold with nonempty boundary.
Assume that $\hat M$ is obtained by gluing $k$ copies of $M$ along $\partial M$.
Then $$\widehat{\AA}_s=k |M|+k(k-2)\sqrt{\frac{3s}\pi} |\partial M|+O(s).$$
Therefore, $\widehat{\KK}=-\infty$ whenever $k>2$.
\end{theorem}

\begin{proof}
    We will denote by $i$ the quantities on the $i$-th copy of $M$. In the next computation we will integrate by parts formally as if $d(x,\cdot)$ were smooth on $\bigsqcup_j M_j$ and $C^1$ across the boundary between any two copies. This issue is treated carefully in Appendix \ref{app:C0}.
    \begin{align}
        &\int_{\hat M}\partial_s \a_s(x) d\vol(x) \\&= \sum_{i,j=1}^k\int_{M_i}
    (12\pi s)^{-n/2}\frac1{2s}\int_{M_j} \bigg[-n+\frac1{6s}d^2(x,y) \bigg]e^{-\frac{d^2(x,y)}{12s}}d\vol_j(y)d\vol_i(x)\\
    &= \sum_{i,j=1}^k\int_{M_i} (12\pi s)^{-n/2}\frac1{2s}\int_{M_j} \bigg[-n+\frac12\Delta d^2(x,y) \bigg]e^{-\frac{d^2(x,y)}{12s}}d\vol_j(y)d\vol_i(x) \\
    &\ \ - \sum_{i,j=1}^k\int_{M_i} (12\pi s)^{-n/2}\frac1{2s} \int_{\partial M_j} d(x,y)\langle \nabla_y d(x,y),\nu_j\rangle e^{-\frac{d^2(x,y)}{12s}} d\sigma_j(y)d\vol_i(x) \\
    &= O(1) - \sum_{i=1}^k\int_{M_i} (12\pi s)^{-n/2}\frac1{2s} \sum_{j=1}^k\int_{\partial M_j} d(x,y)\langle \nabla_y d(x,y),\nu_j\rangle e^{-\frac{d^2(x,y)}{12s}} d\sigma_j(y)d\vol_i(x) \\
    &=O(1) - \sum_{i=1}^k \int_{M_i} (12\pi s)^{-n/2}\frac1{2s} (1-(k-1))\int_{\partial M_i} d(x,y)\langle \nabla_y d(x,y),\nu_i\rangle e^{-\frac{d^2(x,y)}{12s}} d\sigma_i(y)d\vol_i(x) 
    \end{align}
    where in the last equality we used that $\langle \nabla d,\nu_j\rangle = -\langle\nabla d,\nu_i\rangle$ for $j\neq i$. Using the computations from Theorem \ref{thm:boundary}, we conclude:
    \begin{align}
    ...&= O(1) - k(2-k) \int_{M} (12\pi s)^{-n/2}\frac1{2s} \int_{\partial M} d(x,y)\langle \nabla_y d(x,y),\nu\rangle e^{-\frac{d^2(x,y)}{12s}} d\sigma(y)d\vol(x) \\
    &= O(1) - k(2-k)|\partial M|\sqrt\frac{3}{4\pi} \frac{1}{\sqrt{s}}\,.
    \end{align}
\end{proof}

\subsection{Lipschitz gluing of manifolds}
Here we extend the results of the doubling case to the case of Lipschitz gluing along a hypersurface. Equation \eqref{density} is satisfied with $\rho=1$.
\begin{theorem}\label{thm:C0gluing} Let $(M_1,g_1)$, $(M_2,g_2)$ be two smooth, compact Riemannian manifolds with isometric boundary $\partial M_1=\partial M_2=(\Sigma,\bar g)$, and consider the (smooth manifold with Lipschitz Riemannian metric) $(\hat M, \hat g)$, obtained by gluing of $M_1$, $M_2$ along $\Sigma$.
Then the curvature functional of $\hat M$ is
$$\widehat{\KK}
=\int_{\hat M\setminus \Sigma} R(x)d\vol(x)+2(n-1)\int_{\Sigma} (H_1(z) + H_2(z))d\sigma(z)$$
where $H_1$ and $H_2$ are the mean curvatures of $\Sigma$ induced respectively by the immersion in $(M_1,g_1)$, $(M_2,g_2)$ with respect to the ``outward''-pointing unit normal vector (in our convention $(n-1)H$ is the trace of the second fundamental form).
 
Again, $\int_{\hat M}\hat\k(x) dm(x)=\int_{\hat M \setminus \Sigma} R(x)d\vol(x)$.
\end{theorem}

We give the proof in Appendix \ref{app:C0}.

\begin{remark}
    In the case where $(M_1,g_1)$ is bounded and $(M_2,g_2)$ is unbounded and Euclidean, the term $\frac{1}{8\pi}\int (H_1+H_2) d\sigma$ is called Brown-York mass, see \cite[Eq.~4.8]{BY91}, \cite{ShiTam02}.
\end{remark}

\begin{proposition}\label{prop:gluingapprox}
    Given $(\hat M, \hat g)$ as in Theorem \ref{thm:C0gluing}, we can construct a family $(\hat M_\epsilon, \hat g_\epsilon)$ $\epsilon>0$ of smooth Riemannian manifolds, so that \begin{itemize}
        \item $(\hat M_\epsilon, \hat g_\epsilon) \to (\hat M, \hat g)$ in $C^0$;
        \item $\widehat{\KK}_\epsilon \to \widehat{\KK}=\int_{M_1} R_{g_1}d\vol_{g_1} + \int R_{g_2}d\vol_{g_2}+ 2(n-1)\int_\Sigma (H_1 + H_2) d\sigma$.
    \end{itemize}
\end{proposition}

\begin{proof}
    We only need to control the behaviour around $\Sigma$.
    In \cite{Miao02} Miao constructs smooth approximating metrics by mollification around the hypersurface $\Sigma$. Integrating his Equation (32) we obtain (note the different normalization convention for $H$, the factor of 2 missing after plugging Equation (30) into (17), and the different normal vector considered), as $\delta\to 0$,
    \begin{align}
        &\int_{\Sigma\times [-\frac{\delta^2}{100},\frac{\delta^2}{100}]}O(1)+ 2(n-1)(H_1+H_2)\frac{100}{\delta^2}\phi\left(\frac{100}{\delta^2} t\right)d\sigma_t(x)dt  \longrightarrow 2\int_\Sigma (n-1)(H_1+H_2)d\sigma(x)
    \end{align}
    since $d\sigma_t=d\sigma_0 + O(1)$ by continuity of the gluing and $\int_{-1}^1\phi=1$.
\end{proof}

\subsection{Cones} 
We now consider the case when $X$ is the metric completion of the cone over a smooth manifold M. In this case \eqref{density} is satisfied with $\rho=\frac{\vol(M)}{\omega_n}$ at the tip and $\rho=1$ otherwise.
\subsubsection{Two-dimensional cones}
We define the function
\begin{align}
    \mathfrak C(\alpha)=
    \begin{cases}
    2\alpha\left(1-\frac{\alpha}{2\pi} \right) \, , & \text{if $\alpha\leq 0$}\\
    3\left(1-\frac\alpha{2\pi}\right)
    \frac{2\alpha-2\sin\alpha}{1-\cos\alpha} \, , &\text{if $\alpha\in[0,\pi]$} \\
    3\left(1-\frac{\alpha}{2\pi}\right)\left[ \frac{3\alpha-\pi-\sin\alpha}{1-\cos\alpha}\right] \, & \text{if $\alpha\in [\pi,2\pi)$} \,.
    \end{cases}
\end{align}

\begin{remark}
    Note that $\mathfrak C$ is $C^2$ and $\mathfrak C(\alpha)\to\infty$ as $\alpha\to 2\pi^-$.
\end{remark}
\begin{figure}[ht]
\centering
\includegraphics[width=10cm]{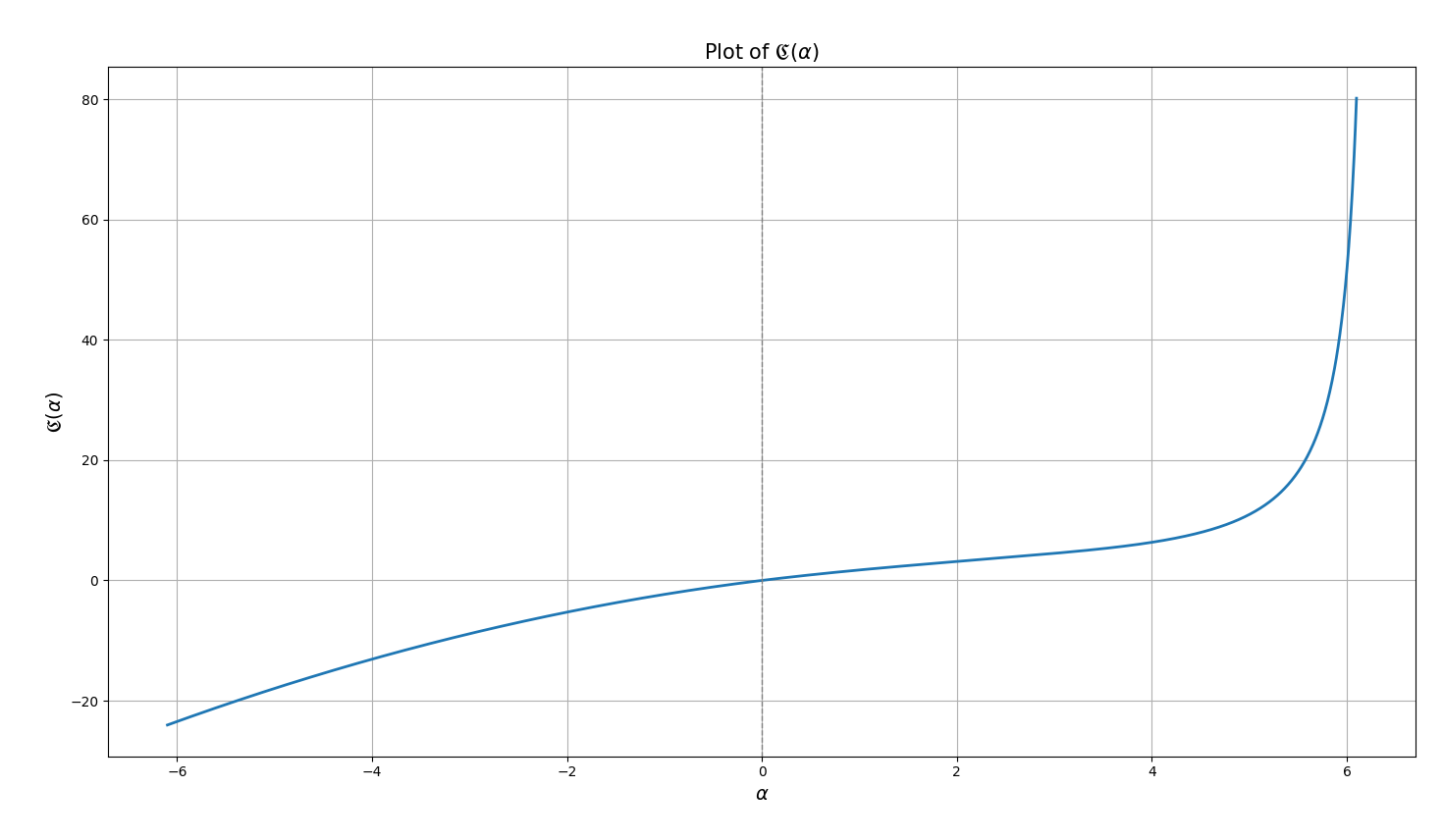}
\caption{Graph of $\mathfrak C$.}
\end{figure}

In the following, we call $x\in X$ a conical singularity if there exists a neighbourhood $B_\eps(x)$ that is isometric to a ball on a cone over $S^1_{2\pi-\alpha}$, $\alpha\in(-\infty,2\pi)$, centred at the vertex.
\begin{theorem}\label{2d-cone}
Assume that $M$ is a closed surface, smooth outside finitely many conical singularities $z_1, \ldots, z_k$ with conical angles $2\pi-\alpha_1,\ldots, 2\pi-\alpha_k$ for some $\alpha_i\in (-\infty,2\pi)$.
Then
   $${\KK}=
   \int_{M\setminus \{z_1,\ldots,z_k\}} R\,d\vol+\sum_i \mathfrak C(\alpha_i)\,.$$
On the other hand,
   ${\KK}\not=\int\k d\vol$
   and $\k(z_i)=0$ for $i=1,\ldots,k$.
   \end{theorem}
For the proof see Appendix \ref{app:cone}.
\begin{corollary}
    The 2-dimensional cone with $\KK=0$ is $\R^2$.
\end{corollary}

\begin{proposition}\label{prop:smoothing2dcone}
Let $M$ be as in Theorem \ref{2d-cone}. If $(M^\epsilon,g^\epsilon)$ denotes a smoothing that preserves the topology of $(M,g)$ and $R^\epsilon$ its scalar curvature, then in general
   $${\KK}\not=\lim_{\epsilon\to0}{\KK}^\epsilon.$$
    In other words,
    the correction term $\sum_i \mathfrak C(\alpha_i)$ does not coincide with 
   $$\lim_{\epsilon\to 0}\int_{M^\epsilon} R^\epsilon(x)d\vol_\eps(x)-\int_{M\setminus \{z_1,\ldots,z_k\}} R(x)d\vol(x)=\sum_i 2\alpha_i,$$
   as similarly noticed in \cite[Ex.~1.14]{KLP21}.
\end{proposition}

\begin{proof}[Proof of Proposition \ref{prop:smoothing2dcone}]
Note first that the claim does not depend on the smoothing, as long as the topology is preserved, since by Gauss-Bonnet
\begin{align}
    \lim_{\epsilon\to 0}\int_{M^\epsilon} R^\epsilon(x)d\vol_\eps(x) = 4\pi\chi(M_ \eps)=4\pi\chi(M)\,.
\end{align}
Consider then a smoothing $M_\eps$ such that $M \setminus \bigcup_{j = 1}^n B_\eps(p_j) =M_\eps \setminus \bigcup_{j = 1}^n B_\eps(p_j)$, where $B_\eps(p_j)$ is the ball of radius $\eps$ centered at $p_j$ in the metric of $M$. For a construction see the proof of Proposition \ref{prop:higherdconeapprox}. Write
\[
\lim_{\eps \to 0} \int_{M_\eps} R_\eps(x)\,d\vol_\eps(x) = \lim_{\eps \to 0}\int_{M \setminus \bigcup_{j = 1}^n B_{2\eps}(p_j)} R(x)\,d\vol(x) + \lim_{\eps \to 0}\sum_{j = 1}^n \int_{B_{2\eps}(p_j)} R_\eps(x)\,d\vol_\eps(x)\,.
\]
By the bounded curvature of $M$, the first term of the RHS converges to $\int_{\Sigma \setminus \{p_1,\ldots,p_n\}} R(x)\,d\vol(x)$. The other term can be computed using Gauss-Bonnet:
\begin{align}
\lim_{\eps \to 0} \int_{B_{2\eps}(p_j)} R_\eps(x)\,d\vol_\eps(x) &= 4\pi\chi(B) - 2\lim_{\eps\to0}\int_{\partial B_{2\eps}(p_j)} \kappa(\theta)d\sigma(\theta) \\
&=4\pi - 2(2\pi-\alpha_j) = 2\alpha_j
\end{align}
where we used that by construction $\partial B_{2\eps}(p_j)$ is the distance sphere inside a cone over $S^1_{2\pi-\alpha_j}$, hence $\kappa=1/2\eps$, $\sigma(\partial B_{2\eps})= (2\pi-\alpha_j)2\eps$. 
\end{proof}

The following extends Theorem \ref{2d-cone} and will be used in the time-dependent setting. For the proof see Appendix \ref{app:sphsusp}.
\begin{theorem}\label{2d-sphsusp}
    The same conclusion of Theorem \ref{2d-cone} holds if $z_i$ is allowed to be a spherical suspension singularity of angle $2\pi-\alpha_i$, i.e. if for some $\eps>0$, $B_\eps(z_i)$ is isometric to 
    \[ (([0,\eps] \times S^1_{2\pi-\alpha_i})/\sim\,, \, g(r,\theta) = dr^2 \oplus \sin^2(r)g_{S^1_{2\pi - \alpha_i}})\, . \]
\end{theorem}

\subsubsection{Higher-dimensional cones}
\begin{theorem}\label{thm:higherdcone}
Assume that $X$ is a compact mm-space that is isometric to a smooth Riemannian manifold of dimension $n>2$ outside of finitely many singularities $z_1, \ldots, z_k$, and that  
in the $\delta_i$-neighborhood of $z_i$ it is isometric to the cone over some closed manifold $(M_i,g_i)$. 
Then
$${\KK}_X=\int_{X\setminus \{z_1,\ldots,z_k\}}  R\,d\vol=\int_X\k \,d\vol.$$
\end{theorem}

\begin{proposition}\label{prop:higherdconeapprox}
Consider $X$ a mm-space as in Theorem \ref{thm:higherdcone}, and assume furthermore that every $(M_i,g_i)_{i=1,...z}$ is isometric to a sphere. Then there exist $(M_\eps,g_\eps)$ closed, $C^2$ regular Riemannian manifolds such that
\begin{itemize}
    \item $(M_\eps,g_\eps)\to X$ in Gromov--Hausdorff topology;
    \item ${\KK}(X)=\lim_{\epsilon\to 0}{\KK}(M_\epsilon)=
\lim_{\epsilon\to 0}\int_{M_\epsilon} R_\epsilon\, d\vol_\epsilon.$
\end{itemize}
\end{proposition}

It will be convenient to prove a more general statement than Theorem \ref{thm:higherdcone}, that covers arbitrary suspensions of a smooth cross-section over an interval.

\begin{lemma}\label{lem:n-geq-3-susp}
Given a smooth function $f: [0,R] \to [0,\infty)$ satisfying $R > 0$, $f(r) = 0$ iff $r = 0$, $f'(0) > 0$, and a closed Riemannian manifold $(N^{n - 1},g_N)$ with $n > 2$, consider the space $(X,d,m)$ given by adding a point $o$ to complete 

\begin{equation}
(M, h, d\vol_h)=((0,R]_r\times N, h:= dr^2 + f^2(r)\cdot g_N, f^{n - 1}(r)\cdot d\rho\wedge d\vol_{N})\,.
\end{equation}

Then for any compact subset $K \subseteq \{r < R\}\subseteq X$ 

\begin{equation}\label{eq:susp-scalar-int}
{\KK}(K) = \int_{K\setminus \{o\}} R_h d\vol_h =\int_{K\setminus \{o\}}\frac{R_N\circ \pi_N}{(f\circ r)^2} -\left(2(n - 1)\frac{f''}{f} +(n - 1)(n - 2)\frac{(f')^2}{f^2}\right)\circ r \,d\vol_h \,.
\end{equation}

Moreover, convergence to the integrand of ${\KK}(K)$ is controlled in the sense that given a fixed $(N, g_N)$ and uniform bounds

\begin{align}\label{eq:warping-factor-bounds}
\begin{cases}
0 < C_0^{-1} \leq f'\,,& \forall r \in [0,R/2]\,;\\
0 \leq f',\, r^2\left|\frac{f''}{f}\right|\ \leq C_0\,,& \forall r \in [0,R]\,,
\end{cases}
\end{align}

we have the uniform estimates near the tip

\begin{align}\label{eq:susp-unif-conv}
|\partial_s\a_s(x)| \leq C(N,C_0)\cdot d_X^{-2}(o,x)\,,\quad \forall s > 0\,, x \in B_{1 \wedge R/2}(o)\,.
\end{align}
\end{lemma}

\begin{proof}
For conciseness, we will denote the rescaling operation $\lambda x = (\lambda r,\theta) \in (0,\lambda R) \times N$ for $x = (r,\theta) \in X$ and any scaling factor $\lambda > 0$. Now fix any $x \in B_{1 \wedge R/2}(o)$, choose the specific scaling factor $\lambda := d_X^{-1}(o,x)$, and rescale

\begin{align}
    \partial_s\a_s^{X}(x) &= \frac{1}{(12\pi s)^{n/2}} \frac{1}{2s} \int_{(0,R) \times N} \left[-n + \frac{d^2_X(x,(r,\theta))}{6s}\right]\exp\left({-\frac{d^2_X(x,(r,\theta))}{12s}}\right)\,f^{n - 1}(r)d\vol_{g_N}(\theta)dr \\
    &= \lambda^2\frac{1}{(12\pi \lambda^2 s)^{n/2}} \frac{1}{2\lambda^2s} \int_{(0,\lambda R) \times N} \left[-n + \frac{(\lambda d_X)^2(x,(\lambda^{-1}r,\theta))}{6\lambda^2s}\right]\\
    &\hspace{4.8cm}\times \exp\left(-\frac{(\lambda d_X)^2(x,(\lambda^{-1}r,\theta))}{12 \lambda^2 s}\right)\,\big(\lambda f(\lambda^{-1}r)\big)^{n - 1}d\vol_{g_N}(\theta)dr \\
    &= \lambda^2\cdot\left.\partial_\tau \a_{\tau}^{\lambda X}(\lambda x) \right|_{\tau = \lambda^2 s}\,.\label{eq:scaling-almost-cone}
\end{align}

Here, $\lambda X$ is the metric completion of the Riemannian manifold

\begin{equation}
\Big((0,\lambda R]_\rho \times N, 
d\rho^2 + f^2_\lambda(\rho)\cdot g_N, f_\lambda^{n - 1}(\rho)\cdot d\rho\wedge d\vol_N\Big) \cong ((0,R]\times N, \lambda^2 h, \lambda^{n}d\vol_h)\,.
\end{equation}

where $f_\lambda := \lambda f(\cdot/\lambda):[0,\lambda R] \to [0,\infty)$. Note that the bounds \eqref{eq:warping-factor-bounds} are preserved under such rescalings (though the right endpoints of \eqref{eq:warping-factor-bounds} transform as $R/2 \rightsquigarrow\lambda R/2$, $R\rightsquigarrow\lambda R$). 

In order to continue as in the proof of Theorem \ref{thm:smoothcase}, we need the following technical claim.\\

\textbf{Claim}: For any $x$ and $\lambda$ as above, there is a radius $r_0 = r_0(N,C_0) > 0$ such that $|\sec_{\lambda M}| \leq K(N,C_0)$ and $d_{\lambda X}^2(\lambda x,\cdot) $ is smooth on $B_{r_0}^{\lambda X}(\lambda x)$. 

\textit{Proof.} We will now produce a warped product metric on $[0,2]\times N$ that has sectional curvature $\leq K(N,C_0)$, is a product metric on a neighborhood of each boundary component $\{0,2\}\times N$, and agrees with $d\rho^2 + f_\lambda^2(\rho)\cdot g_N$ on $[1/2,3/2]\times N$. Gluing two copies along their boundary, we obtain a smooth closed $n$-manifold with sectional curvature $\leq K(N,C_0)$ and 

\begin{equation}
\vol(B_{1/2}(\lambda x)) \geq \vol_{\lambda X}\big([3/4,5/4]\times B^N_{(6C_0)^{-1}}(\pi_N( x))\big) \geq \frac{\vol_N(B^N_{(6C_0)^{-1}}(\pi_N( x)))}{2(2C_0)^{n - 1}} \geq v(C_0,N) > 0
\end{equation}

Using Cheeger's Lemma (see \cite{Che67},\cite[Lem.~11.4.9]{Pet16}), there is a lower bound $0 < r_0 = r_0(N,C_0) \leq 1/2$ on the injectivity radius. Since $B_{1/2}^{\lambda X}(\lambda x)\subseteq [1/2,3/2]\times N$ embeds isometrically into this closed manifold, $d_{\lambda X}(\lambda x,\cdot)$ is smooth on $B_{r_0}^{\lambda X}(\lambda x)$.

It remains to define the new warping function $\tilde f_\lambda$ on $[0,2]$. Instead of writing an explicit formula for $\tilde f_\lambda$, we simply specify a list of requirements that can be satisfied. Since $C^{-1}_0 \leq f'_\lambda \leq C_0$, we can choose $\tilde f'_\lambda \geq 0$ smooth such that 

\begin{align}
\tilde f'_\lambda(r)\begin{cases}
= 0,& r \in [0,1/2 - (4C_0^2)^{-1}]\\
\leq C_0,& r \in [1/2 - (4C_0^2)^{-1},1/2]\\
= f_\lambda'(r),& r \in [1/2,3/2]
\end{cases}\,;&\\
&\text{while}\quad  |\tilde f''_\lambda| \leq 6C_0^3\,,\quad r \in [1/2 - (4C_0^2)^{-1},1/2]\,.
\end{align}

Requiring $\tilde f_\lambda(1/2) = f_\lambda(1/2)$, which is $\geq (2C_0)^{-1}$, we can integrate the upper bound $\tilde f_\lambda' \leq C_0$ above on $[1/2 - (4C_0^2)^{-1},1/2]$ to obtain $\tilde f_\lambda(1/2 - (4C_0^2)^{-1}) \geq (4C_0)^{-1}$. Similarly, we can choose $\tilde f_\lambda'$ on $[3/2,2]$ such that $\tilde f_\lambda'(r) = 0$ for $r \in [7/4,2]$ and $|\tilde f''_\lambda| \leq 6C_0$ while preserving $0 \leq \tilde f_\lambda' \leq C_0$. The sectional curvature bound $\leq K(N,C_0)$ holds on $([0,2]\times N, d\rho^2 + \tilde f_\lambda^2(\rho)\cdot g_N)$ by the formula for the sectional curvature of a warped product \cite[Sec.~4.2]{Pet16}. Note in particular that such sectional curvature bounds are guaranteed as long as $\tilde f_\lambda \geq C(C_0)^{-1} > 0$ and $|\tilde f_\lambda'|\,,|\tilde f_\lambda''|\, \leq C(C_0)$, which we have ensured by the above requirements.\qed
\medskip

We now safely compute:

\begin{align}
&\phantom{=} \frac{1}{(12\pi \tau)^{n/2}} \frac{1}{2\tau} \int_{B_{r_0}(\lambda x)} \left[-n + \frac{d^2_{\lambda X}(\lambda x,z)}{6\tau}\right]e^{-\frac{d^2_{\lambda X}(\lambda x,z)}{12\tau}}\,dm_{\lambda X}(z) \\
&=\frac{1}{(12\pi \tau)^{n/2}} \frac{1}{2\tau} \int_{B_{r_0}(\lambda x)} \left[-n + \frac{1}{2}\Delta d^2_{\lambda X}(\lambda x,z)\right]e^{-\frac{d^2_{\lambda X}(\lambda x,z)}{12\tau}}\,dm_{\lambda X}(z) \\
&\phantom{=}-\frac{1}{(12\pi \tau)^{n/2}} \frac{r_0}{2\tau} e^{-\frac{r_0^2}{12\tau}}\sigma_{\lambda X}(\partial B_{r_0}(\lambda x)) = O_{N,C_0}(1) \label{eq:scalar-conv-almost-cone}\,.
\end{align}

Indeed, the second line is bounded (from both sides) by Hessian comparison in the presence of the sectional curvature bound $K(N,C_0)$ to obtain $|\Delta d^2_{\lambda X}(\lambda x,\cdot) - 2n| \leq C(N, C_0)d^2_{\lambda X}(\lambda x,\cdot)$, along with a $(N,C_0)$-uniform bound on the scale-invariant volume of distance spheres $\sigma_{\lambda X}(\partial B_\rho(\lambda x))\leq C(N,C_0)\rho^{n-1}$ for $\rho \leq r_0$. The LHS of the third line is bounded using the previous statement with $\rho = r_0$. \\

The portion of the integral over the complement of $B_{r_0}^{\lambda X}(\lambda x)$ is part of the tail and uniformly bounded, again using the $\leq C(N,C_0) \rho^{n - 1}$ bound on volumes of $\rho$-distance spheres. We have $(C_0,N)$-uniformly bounded the second factor on the RHS of \eqref{eq:scaling-almost-cone}, which proves \eqref{eq:susp-unif-conv}. Now we use dominated convergence to compute
\begin{align}
    \KK(K)&=\lim_{s\searrow0}\int_{K} -\partial_s\a_s(x) d\vol_h(x)  \\
    &= \int_{K\setminus \{o\}} \frac{R_N\circ \pi_N}{(f\circ r)^2}\,d\vol_h -\left(2(n - 1)\frac{f''}{f} +(n - 1)(n - 2)\frac{(f')^2}{f^2}\right)\circ r \,.
\end{align}
To be more precise, the pointwise convergence for points $ x\neq o$ follows from Theorem \ref{thm:smoothcase} and the formula for the scalar curvature of a warped product, while the dominating function is the RHS of \eqref{eq:susp-unif-conv} which is integrable using $\sigma_{X}(\partial B_r(o))\leq C(N,C_0) r^{n - 1}$.
\end{proof}

\begin{proof}[Proof of Theorem \ref{thm:higherdcone}]
    Thanks to the gluing theorem \ref{thm:C0gluing} it is enough to directly work on a ball around the vertex of a cone over a closed manifold $(M,g)$, that we denote with $B_R(o)\subseteq C(M)$. The formula follows from Lemma \ref{lem:n-geq-3-susp} with $f(r) := r$. 
\end{proof}

\begin{proof}[Proof of Proposition \ref{prop:higherdconeapprox}]
Let, for $i=1,...,k$, $U_i\subset X$ be a neighbourhood around the singular point $z_i\in X$. Then we can find $L_i>0$, $\rho_i>0$ s.t. $U_i\cong (([0,L_i]\times S^{n-1})/\sim, dr^2 + (\rho_i r)^2 g_{S^{n-1}})$. Let $\phi(r)=\rho_i r$, dropping the $i$ index for the sake of notation, be the warping function. In order to approximate $\phi'(r)\equiv \rho_i$, let $\psi \in C^1([0,\infty))$ monotone such that $\psi(0)=1$, $\psi(r)=\rho_i$ for every $r\in [1,\infty)$ and $\psi'(0) = 0$. Rescale then $\psi_\eps(r)=\psi(r/\eps)$ and consider the primitives 
\[ \phi_\eps(r) = \int_0^r\psi_\eps(s)ds \,. \]
Consider now $(M_\eps,g_\eps)$ obtained by $X$ by substituting every conical neighbourhood $U_i$ with $U_i^\eps$, where $\widetilde U_i^\eps\cong (([0,L_i]\times S^{n-1})/\sim, dr^2 + \phi_\eps^2(r)g_{S^{n-1}})$, for $\eps < L$. Then, for $\eps$ small enough, $M_\eps$ is $C^2$ regular because $\phi_\eps'(0)=1$, $\phi_\eps''(0) = 0$. Moreover $\phi_\eps\to\phi$ uniformly and so do the distance functions $d_{g_\eps}\to d_g\in C^0(X^2)$.

By smoothness of $M_\eps$, 
\begin{align}
    \KK(M_\eps) &=\int_{M_\eps} R_{g_\eps} d\vol_{g_\eps} \\
    &= \int_{X\setminus\bigcup_i U_i} R_g d\vol_g + \sum_{i=1}^k\int_{U_i}R_{g_\eps} d\vol_{g_\eps}\,.
\end{align}
In order to prove $\KK_\eps\to\KK$, we want to pass the limit $\eps\to 0$ inside the integrals over $U_i$, since $R_{g_\eps} \sqrt{|g_\eps|}$ converges pointwise to the analogue quantity for the cone for every $x\in U_i\setminus\{z_i\}$. This follows by dominated convergence since 
\begin{align}\label{eq:scalcurv_warped}
    R_{g_\eps}\sqrt{|g_\eps|}
= \left[-2 (n-1) \frac{\phi_\eps''}{\phi_\eps}
+ (n-1)(n-2) \frac{1 - (\phi_\eps')^2}{\phi_\eps^2}\right]\phi_\eps^{n-1}
\end{align}
are dominated by $Cr^{n-3}$ for $n\geq3$. Indeed, $|\phi_\eps'(r)|=|\psi(r/\eps)|\leq C$ by construction, hence $|\phi_\eps(r)|\leq Cr$, $|\phi_\eps''(r)| = |\psi'(r/\eps)| \frac{1}{\eps} \ind_{[0,\eps]}(r) \leq \frac{C}{r}\ind_{[0,\eps]}(r)$, and the absolute value of both terms of \eqref{eq:scalcurv_warped} are bounded by $Cr^{n-3}$.
\end{proof}

\begin{corollary}
Let $M$ be an $(n-1)$ Einstein manifold with $\Ric=(n-2)g$, $n\geq3$. Then the metric completion of the cone $C(M)$ over $M$ is Ricci flat in the sense that it is $\RCD(0,n)$ and $\KK^{(n)}\le0$.
\end{corollary}

\begin{proof}
    It follows from Ricci flatness outside the tip and the fact that $\KK$ does not see the tip in higher dimensional cones.
\end{proof}

\begin{example}
    Non smooth examples are given by choosing $M={\mathbb S}^{n-1}(1)/\Z_2$, $n\ge 3$. Recall that Eguchi and Hanson \cite{EguchiHanson} constructed a smooth Ricci flat metric on $T^*S^2$ which provides an approximation for the former space with $n=4$, cf.~\cite[pag.~11]{CheegerNotes2001} and \cite[Rmk.~5.4]{sturm2017remarks}.

    Another example is $M={\mathbb S}^2(1/\sqrt 3)\times {\mathbb S}^2(1/\sqrt 3)$.
\end{example}

\subsection{Surfaces of convex polytopes}
We now consider the case when $(X,d,m)$ is the surface of an $(n+1)$-dimensional compact (convex) polytope $\hat X \subset \R^{n+1}$.
\begin{definition}
A {\it closed} (resp. open) {\it polytope} is a set $K\subseteq \R^{n + 1}$ with nonempty interior, and of the form

\begin{equation}
K = \bigcap_{j = 1}^N\{v\in \R^{n + 1}\mid v\cdot \xi_j \leq \lambda_j\} \qquad \text{(resp. $\{\cdots < \lambda_j\}$)}\,,
\end{equation}

for a (finite) collection of real numbers $\lambda_1,\ldots,\lambda_N \in \R{}$ and unit vectors $\xi_j \in S_1^{n}\subseteq\R^{n + 1}$. Moreover, we require the full signed-span condition $\sum_{j = 1}^N \R{}_+\cdot\xi_j = \R^{n + 1}$ (to enforce boundedness).
\end{definition}

\begin{proposition} 
Let $(X,d,m)$ be the boundary of a polytope. Then it is an $\RCD(0,n)$ space, but it is not minimal.
\end{proposition}

\begin{proof}
Because $(X,d,m)$ bounds a convex subset of $\R^{n+1}$, it is an Alexandrov space, see \cite[Sec.~10.2.3]{BBI}, of (Hausdorff) dimension $n$, and in particular it is $\RCD(0,n)$, see \cite{KMS01} and \cite{PetruninLVS}.

We claim that $(X,d)$ has a nonempty codimension $2$ singular stratum. More precisely, there is an open set $U \subseteq (X,d)$ which embeds isometrically $U \hookrightarrow (\R^{n - 2},d\mathbf{x}^2)\times C(S^1_\theta)$ where $\theta \in (0,1)$. 

\medskip

1) We first argue that for every integer $0\leq m \leq n + 1$, there is a point $p \in K$ such that

\begin{equation}
\dim\left(L(p)\right) = n + 1 - m\,,\qquad \text{where }L(p) := \bigcap_{\{j\mid p\cdot\xi_j = \lambda_j\}}\{v\in\R^{n + 1}\mid v\cdot\xi_j = \lambda_j\}
\end{equation}

We will argue by induction on $m$; the base case $m = 0$ follows from the non-empty interior assumption on $K$ in the definition of polytope, and $L(p) = \bigcap_{j\in\emptyset}\{v\mid v\cdot\xi_j = \lambda_j\} = \R^{n + 1}$ for any interior point $p \in K^\circ$. So assume we have found a point $p$ satisfying the inductive hypothesis for $m - 1 < n + 1$. Now consider a half-line $\ell_\omega$ starting at $p$ lying inside the affine subspace $L(p)$ of dimension $n + 1 - (m - 1) > 0$. Here $\ell_\omega$ is parametrized by the point $\omega \in \partial B_\eps(p)\cap L(p) \subseteq K$ (take $0 < \eps \ll 1$ small enough) that the half-line passes through. By the intermediate value theorem and the full signed-span condition, there is a first point $q_\omega\in \ell_\omega \cap K$ and at least one index $k \notin \{j\mid p\cdot\xi_j = \lambda_j\}$ so that $q_\omega \cdot \xi_{k} = \lambda_{k}$. Thus $L(q_\omega) \subseteq L(p)\cap \{v\mid v\cdot \xi_k = \lambda_k\}$, and the latter has dimension $= n + 1 - m$ because it is the intersection of transverse affine subspaces of dimension $n + 1 - (m - 1)$, $n$. 

Thus $\dim(L(q_\omega)) \leq n + 1 - m$, but the inequality may be strict, so we must argue that there is some choice of $\omega$ for which equality holds. If not, then $\dim(L(q_\omega)) \leq n -m$ for each $\omega$, but note that there are at most finitely many choices for $L(q_\omega)$, given by the $2^N$ possible indexing subsets of $\{1,\ldots, N\}$ appearing in the definition of $L(q)$. Thus the radial volume-non-decreasing map $\omega \mapsto q_\omega$ sends the round $(n -m + 1)$-sphere $\partial B_\eps(p)\cap L(p)$ into a finite union of affine subspaces with dimension $\leq n - m$, which is a contradiction. We conclude that there is at least one $q_\omega$ that satisfies $\dim(L(q_\omega)) = n + 1 - m$. 

\medskip

2) We are interested in the case $m = 3$ of the first step, so $p \in K$ with $\dim L(p) = n - 2$ and thus also $p \in \partial K = X$. By picking $U$ small enough, we can write

\begin{equation}
U\cap \partial K  = U\cap \bigcup_{\{j\mid p\cdot\xi_j = \lambda_j\}}\left(\{v\mid v\cdot \xi_j = \lambda_j\}\cap \bigcap_{\{k\mid p\cdot\xi_k = \lambda_k\}}\{v\mid v\cdot \xi_k \leq \lambda_k\}\right)\,.
\end{equation}

Note that all sets appearing in the union on the RHS are $(L(p) - p)$-invariant, since they are affine/half spaces containing $L(p)$. Thus after an affine change of coordinates in $\R^{n + 1}$ that sends $p$ to $0$ and $L(p)$ to $\R^{n - 2}\times\{0^3\}$, the RHS is a neighborhood of $0^{n + 1} \in \R^{n - 2}\times Z\subseteq \R^{n + 1}$ where $Z\subseteq\R^{3}$ is the boundary of an intersection of halfspaces that does not contain a line (else $L(p)$ would contain this line, violating $\dim L(p) = n - 2$). In particular $Z$ is isometric to $C(S^1_\theta)$ for some $\theta \in (0,1)$, which proves the claim. 

3) The claim is proved, but now we need to argue that $X$ is not minimal. This follows from the product splitting of Gaussians, namely

\begin{align}
&\phantom{=}\frac{1}{(4\pi s)^{n/2}}\int_{X\times Y}\int_{X\times Y}e^{-\frac{d^2_{X\times Y}(x,y)}{4s}}\,d(\mu\otimes \nu)(x)d(\mu\otimes \nu)(y) \\
&= \left(\frac{1}{(4\pi s)^{k/2}}\int_{X}\int_Xe^{-\frac{d^2_X(x,y)}{4s}}d\mu(x)d\mu(y)\right)\left(\frac{1}{(4\pi s)^{m/2}}\int_{Y}\int_Ye^{-\frac{d^2_Y(x,y)}{4s}}d\nu(x)d\nu(y)\right)
\end{align}

for any choice of $k + m = n$. In this case we take $k = n - 2$, $m = 2$, and obtain a nonzero contribution

\begin{align}
\lim_{s\to 0}\partial_s \int_{U\cap X}\a_s^X(x)\,d\mu(x) &\leq \lim_{s\to 0}\partial_s \left(\int_{U'}\a_s^{\R^{n - 2}}(y)\,dy^{n - 2}\right)\left(\int_{U''}\a_s^{C(S^1_\theta)}(z)\,d\mu^{C(S^1_\theta)}(z)\right) \\
&= \vol(U')\cdot\lim_{s\to 0}\partial_s \left(\int_{U''}\a_s^{C(S^1_\theta)}(z)\,d\mu^{C(S^1_\theta)}(z)\right) < 0\,.
\end{align}

where the last line follows from the product rule, then Theorems \ref{thm:smoothcase} (with flat $\R^{n - 2}$), \ref{2d-cone}, and Lemma \ref{lem:limitsto0}. Here $U$ is chosen as in step 2 above so that $U \cap X$ compactly contains an open set isometric to $U'\times U''$ for neighborhoods $0^{n - 2}\in U' \subseteq \R^{n -2}$, $o \in U''\subseteq C(S_\theta^1)$.
\end{proof}

\section{Time dependent setting}\label{sec:timedep}
In the sequel, let $I \subseteq \R$ denote a non-empty open interval.
 
   \subsection{The Gaussian volume functional for time dependent spaces}
Consider a time-dependent family $(M,d_t,m_t)_{t\in I}$ of mm-spaces satisfying, at every fixed $t$, the sub-exponential growth of balls assumption \eqref{eq:expgrowth}. We will impose two assumptions: one is about the existence of the density (at time fixed, analogous to the static case), one on the time-regularity of the distance.

\paragraph{Assumption A:} Assume that for all $t\in I$ there exists the $n$-dimensional volume densities \begin{equation*}\label{densities}
\rho_t(x):=\lim_{r\searrow0}\frac{m_t(B^{d_t}_r(x))}{\omega_nr^n}
\end{equation*}
and the ratio is bounded uniformly in $(x,r)$.

\paragraph{Assumption B:} Assume that for all $t,t+s\in I$
\begin{align}\label{eq:dtholder}
    d_{t+s}(x,y) = d_t(x,y) + o_t(\sqrt{s})
\end{align} as $s\to 0$, with $o_t(\sqrt{s})$ uniform in $(x,y)$, but possibly depending on $t$.

Note that the condition scales parabolically. It is for example satisfied in the case that $d_t$ is Lipschitz continuous in time. All the examples we will consider satisfy this assumption.

\begin{definition}
    We define the \emph{time-dependent Gaussian volume functional} by
$$\a_{s}(t,x):=(12\pi s)^{-n/2}\int e^{-\frac{d^2_{t+s}(x,y)}{12s}}m_t(dy)\,,$$
and when $m_t$ is finite
\emph{the time-dependent Gaussian double integral}  by
  $${\AA}_s(t) :=  (12\pi s)^{-n/2}\int\int e^{-\frac{d^2_{t+s}(x,y)}{12s}}\, m_t(dy)\,m_t(dx)\,.
  $$
\end{definition}

As in the static case (Lemma \ref{lem:limitsto0}), we get the following for $s\to0$.
\begin{lemma}
    If $(X,d_t,m_t)$ satisfies assumptions A and B, then
    \begin{align}
    \a_0(t,x) := \lim_{s\searrow0} \a_s(t,x)=\rho_t(x)
    \end{align}
    and 
    \begin{align}
        {\AA}_0(t):=\lim_{s\searrow0}{\AA}_s(t)=\int\rho_t\,dm_t
    \end{align}
\end{lemma}
\begin{proof}
Set $\eps(s)=\lVert d_{t+s}(x,y) - d_t(x,y)\rVert_\infty$, by integrating by parts as in Lemma \ref{lem:limitsto0} we have
\begin{align}
    \a_s(t,x)&\leq(12\pi s)^{-n/2}\int e^{-\frac{d^2_t(x,y)-2\eps(s)d_t(x,y)+\eps^2(s)}{12s}}dm_t(y) \\
    &=  (12\pi s)^{-n/2}\int_0^\infty e^{-\frac{r^2}{12s}}e^{\frac{2\eps(s)r-\eps^2(s)}{12s}}\left(\frac{2r}{12s} - \frac{\eps(s)}{12s}  \right)\left[\rho_t(x)\omega_nr^n+o(r^n)\right] dr \\
    &= \rho_t(x) + o(1)\,.
\end{align}
Note that we used that $e^{\frac{2\eps(s)r-\eps^2(s)}{12s}}\rightsquigarrow (1 + o(1))^{1 + r}$ after parabolic rescaling and $\frac{\eps(s)}{12s} =\frac{1}{\sqrt s}o(1)$ by \eqref{eq:dtholder}. The bound $\a_s(t,x)\geq \rho_t(x) + o(1)$ follows analogously.
\end{proof}

\begin{definition}
Define \emph{the initial negative slope of $\a$} by
    $$ \k(t,x) := \liminf_{s\searrow 0}\frac1s \Big[\a_{0}(t,x)-\a_{s}(t,x)
\Big]$$
and for $m_t$ finite \emph{the initial negative slope of $\AA$} by
    $${\KK}(t):=\liminf_{s\searrow0}\frac1s\left({\AA}_0(t)-{\AA}_s(t)
    \right).$$    
\end{definition}

The following is the analogue of Lemma \ref{upp/low}, and will facilitate the computation of $\k$.
\begin{lemma}
    Assume that, for $(t,x)$ fixed, $\a_s(t,x)$ is differentiable in $s$, for all $s\in(0,s_0)$ for some $s_0>0$ then
    \begin{align}
        -\limsup_{s\searrow0} \partial_s\a_s(t,x)\le \k(t,x)\le -\liminf_{s\searrow0} \partial_s\a_s(t,x)\,.
    \end{align}
    If $m_t$ is finite and $\AA_s(t)$ is differentiable in $s$, for all $s\in(0,s_0)$ for some $s_0>0$ then
    \begin{align}
        -\limsup_{s\searrow0} \partial_s\AA_s(t)\le \KK(t)\le -\liminf_{s\searrow0} \partial_s\AA_s(t)\,. 
    \end{align}
\end{lemma}

\begin{theorem}\label{thm:smoothtimedep}
Assume that $(M,g_t)_{t\in I}$ is a smooth complete time-dependent Riemannian manifold without boundary satisfying \eqref{eq:expgrowth} for all $t\in I$ and set $\Sic_t = -\frac{1}{2}\partial _tg_t$, $S_t=\tr \Sic_t$. Assume furthermore that there is $C>0$ such that $\lVert \Sic_{t,y} - \Sic_{t,x} \rVert \leq 1 + C d^2_t(x,y)$ for all $x,y\in M$, $t\in I$. Then
\begin{align}
    \k(t,x) = R_t(x)-S_t(x)\,,
\end{align}
and if $M$ is compact
\begin{align}
    \KK(t) = \int (R_t-S_t)d\vol_t\,.
\end{align}
\end{theorem}

\begin{proof}
Write first
\begin{align}
    \partial_s\a_s(t,x) 
    &=\int \partial_s \left[(12\pi s)^{-n/2} e^{-\frac{d^2_{t+s}(x,y)}{12s}}\right]d\vol_t(y) \\
    &= \int \partial_s \left[(12\pi s)^{-n/2} e^{-\frac{d^2_{t+s}(x,y)}{12s}}\right]d\vol_{t+s}(y) \\
    &\qquad + \int \partial_s \left[(12\pi s)^{-n/2} e^{-\frac{d^2_{t+s}(x,y)}{12s}}\right] (d\vol_t(y)-d\vol_{t+s}(y)) \\
    &=I+II
\end{align}
By standard computations we have that for $d^2_t(x,y)\leq C$,
    \begin{align}\label{eq:derivd2}
        \partial_t d^2_t(x,y) &= -2d_t(x,y)\int_0^{d_t(x,y)}\Sic_t(\dot\gamma_r,\dot\gamma_r)dr \\
        &= -2d^2_t(x,y)\Sic_{t,x}(\dot\gamma_0,\dot\gamma_0) + O(d^3_t(x,y)) 
    \end{align}
    where $\gamma$ is a unit speed geodesic from $x$ to $y$. Moreover, by the assumed bound on $\Sic_t$, $|\partial_t d^2_t(x,y)| \leq Cd^2_t(x,y)(1+d^2_t(x,y))$ globally, which in the sequel allows to expand $\partial_t d^2_t(x,y)$ around $y=x$ in the Gaussian integrals.

Then, computing as in Theorem \ref{thm:smoothcase}, we have
    \begin{align}
        I
        &= \partial_s \left[ (12\pi s)^{-n/2}\int e^{-\frac{d^2_{r}(x,y)}{12s}}d\vol_r(y)  \right]_{r=t+s} - (12\pi s)^{-n/2}\int e^{-\frac{d^2_{t+s}(x,y)}{12s}} \frac{\partial_r d_r^2(x,y)}{12s} |_{r=t+s} d\vol_{t+s}(y) \\
        &= -R_{t+s}(x) + O(s) + (12\pi s)^{-n/2}\int e^{-\frac{d^2_{t+s}(x,y)}{12s}} \frac{1}{12s} \left[2d^2_{t+s}(x,y)\Sic_{(t+s,x)}(\dot\gamma_0^y,\dot\gamma_0^y) + O(d^3_{t+s}(x,y))\right] d\vol_{t+s}(y) \\
        &= -R_{t+s}(x) + O(s) + (12\pi s)^{-n/2} \int_0^\infty \int_{S^{n-1}} e^{-\frac{\rho^2}{12s}} \frac{1}{12s} \left[2\rho^2\Sic_{(t+s,x)}(\theta,\theta) + O(\rho^3) \right] \rho^{n-1}dg_{S^{n-1}}(1+O(\rho^2))d\rho \\
        &= -R_{t+s}(x) + O(s) + \int_{S^{n-1}}\Sic_{(t+s,x)}(\theta,\theta)dg_{S^{n-1}} (12\pi s)^{-n/2} \int_0^\infty e^{-\frac{\rho^2}{12s}} \frac{1}{12s} \left[2\rho^2 + O(\rho^3) \right] \rho^{n-1}d\rho \\
        &= -R_{t+s}(x) + S_{t+s}(x) + O(\sqrt s)\,. 
    \end{align}

Finally, we show that the second integral does not contribute. Using $\frac{d}{dr}d\vol_r = -S_rd\vol_r$ and the bound $|S_r(y)-S_r(x)|\leq \sqrt{n} (1+Cd^2_r(x,y))$
\begin{align}\label{eq:control_tails_timedep}
    II&=\int_M \partial_s \left[(12\pi s)^{-n/2} e^{-\frac{d^2_{t+s}(x,y)}{12s}}\right] \left(e^{\int_t^{t+s}S_r(y)dr} -1 \right) d\vol_{t+s}(y) \\
    &= (12\pi s)^{-n/2}\int_M \left[-\frac{n}{2s}+\frac{d_{t+s}^2(x,y)}{12s^2} -\frac{\partial_s d_{t+s}^2(x,y)}{12s}\right]e^{-\frac{d^2_{t+s}(x,y)}{12s}} \\
    &\hspace{5cm}\times \left(e^{\int_t^{t+s}[S_r(x)+O(1+d_r^2(x,y))]dr} -1 \right) d\vol_{t+s}(y) \,.
\end{align}
After writing $d_r^2(x,y)=(d_{t+s}(x,y)+o(\sqrt{s}))^2$, we notice that in the above integral, as $s\to 0$, the tails do not contribute, hence we can integrate on $B_1(x)$ and expand around $y=x$ using smoothness:
\begin{align}
    II &= o(1)+ (12\pi s)^{-n/2}\int_{B_1(x)} \left[-\frac{n}{2s}+\frac{d_{t+s}^2(x,y)}{12s^2} -\frac{\partial_s d_{t+s}^2(x,y)}{12s}\right]e^{-\frac{d^2_{t+s}(x,y)}{12s}} \\
    &\hspace{5cm}\times \left(\int_t^{t+s}[S_r(x)+O(d_r(x,y))]dr + O(s^2)\right) d\vol_{t+s}(y) \,.
\end{align}
The term $-\frac{n}{2s}+\frac{d_{t+s}^2(x,y)}{12s^2}$ multiplied by the (now $y$-independent) term $\int_t^{t+s}S_r(x)dr$ has a cancellation as in the static case \eqref{eq:vol_and_scal_curv} and is then of order $O(s)$; the other terms are of order $O(\sqrt s)$ as $s\to 0$. 
\end{proof}

\begin{corollary}\label{cor:minimalsrf}
Assume that the smooth, closed time-dependent Riemannian manifold $(M,g_t)_{t\in I}$ evolves as a super-Ricci flow, that is,
$\Ric_{t}+\frac12\partial_t g_t \ge0$ on $I\times M$. Then for every $t\in I$
$${\KK}_t \ge0\,. $$

Moreover, $\Ric_{t}+\frac12\partial_t g_t=0$ on $\{t\}\times M$ if and only if $(M,g_t)$ is minimal at $t$, in the sense that ${\KK}_t \leq 0$.
\end{corollary}

\begin{remark}
    Note that in the definition of $\a$ we are only varying the distance, and not the measure. If we defined
    \begin{align}
        \tilde \a_{s}(t,x):=(12\pi s)^{-n/2}\int e^{-\frac{d^2_{t+s}(x,y)}{12s}}m_{t+s}(dy)\,,
    \end{align}
    then \emph{for any} smooth $(M,g_t)$ it holds
    \begin{align}
        \lim_{s\searrow0}\partial_s \tilde\a_s(t,x) = -R_t(x)\,,
    \end{align}
    i.e. the variation of $d_{t+s}$ and of $d\vol_{t+s}$ cancel each other. Indeed, proceeding as in Theorem \ref{thm:smoothtimedep}
    \begin{align}
        \partial_s\tilde\a_s(t,x) &= -R_{t+s}(x) +  S_{t+s}(x) + O(\sqrt{s}) +(12\pi s)^{-n/2}\int e^{-\frac{d^2_{t+s}(x,y)}{12s}}\partial_rd\vol_r(y)|_{r=t+s} \\
        &= -R_{t+s}(x) +  S_{t+s}(x) + O(\sqrt{s}) - (12\pi s)^{-n/2} \int e^{-\frac{d^2_{t+s}(x,y)}{12s}} S_{t+s}(y)d\vol_{t+s}(y) \\
        &= -R_{t+s}(x) + O(\sqrt s) \,.
    \end{align}
\end{remark}

\subsection{Directional Gaussian volume for time dependent spaces}
Similarly to the time-independent case, for given $t$ and $x$ we
define the set
 $$Z(t,x):=\Big\{ \exp_{t,x}(rz): z\in Z, r\in [0,\delta_{t,x}]\Big\}$$
 for any measurable subset $Z\subset S_{t,x}M$, and put
 $$\a_s(t,x, Z):=(12\pi s)^{-n/2}\int_{Z(t,x)} e^{-\frac{d_{t+s}^2(x,y)}{12s}}m_t(dy)$$
 and $$\k(t,x,Z):=\liminf_{s\to0} \frac1s\Big( \a_0(t,x, Z)-\a_s(t,x, Z)\Big).$$

\begin{theorem}\label{ktZ}
If $(M,g_t)_{t\in I}$ is a smooth complete time-dependent Riemannian manifold with $\partial_tg_t =-2\Sic_t$, then
$$\k(t,x,Z)=
\frac n{|S_{t,x}|}\int_{Z}\Big[\Ric(z)-\Sic(z)\Big]\,d\sigma_{t,x}(z).$$
\end{theorem}
\begin{proof}
The proof is analogous to the one of Theorems \ref{k-ric} and \ref{thm:smoothtimedep}:
\begin{align}
    \partial_s\a_s(t,x,Z) 
    &=\int_Z \partial_s \left[(12\pi s)^{-n/2} e^{-\frac{d^2_{t+s}(x,y)}{12s}}\right]d\vol_t(y) \\
    &= \int_Z \partial_s \left[(12\pi s)^{-n/2} e^{-\frac{d^2_{t+s}(x,y)}{12s}}\right]d\vol_{t+s}(y) + O(\sqrt s) \\
    &= - \frac{1}{\omega_n} \int_Z \Ric(z)d\sigma(z) + \frac{1}{\omega_n}\int_Z \Sic_{t+s,x}(z) d\sigma(z) + O(\sqrt s) \\
    &= - \frac{1}{\omega_n} \int_Z \left[\Ric(z)- \Sic_{t+s,x}(z)\right]d\sigma(z) + O(\sqrt s)\,.
\end{align}
\end{proof}

Now we want to give characterization of the SRF inequality in terms of a Laplace comparison (which is similar to \cite[Thm 3.5]{BamlerEntropy}) and of $\k$.

\begin{lemma}\label{lem:newlaplcomp}
    Let $(M,g)$ be a smooth (static) Riemannian manifold, $x,y\in M$ and let $\gamma$ be a unit speed minimizing geodesic with $\gamma_0=x$, $\gamma_r=y$. Then
\begin{align}\frac1{2}\Delta_y d^2(x,y) &\le n-\frac1{r}\int_0^r \Ric(\dot\gamma_q,\dot\gamma_q)\,q^2\,dq\\
&=n-\frac2r\int_0^r \int_q^r \Ric(\dot\gamma_p,\dot\gamma_p) dp\,q\,dq.
\end{align}
in the barrier sense.
\end{lemma}
\begin{proof}
    We first assume that $x,y$ are not in the cut locus of each other. Given $\gamma(s)$, $s\in[0,r]$, consider $\{\dot\gamma(s),e_2(s),...,e_{n}(s)\}$ an orthonormal frame along $\gamma(s)$, let $Y_i(s)=se_i(s)$. 
    Consider $\gamma_u^i(s)$ a variation generated by $Y_i$, then by the second variation formula
    \begin{align}
        \frac{d^2}{du^2}|_{(u=0)}E(\gamma_u) = \int_0^r \left(1-s^2\Rm(\dot\gamma,e_i,e_i,\dot\gamma)\right) ds \,.
    \end{align}
    Since $\gamma$ is a geodesic, and $d(x,\cdot)$ is smooth around $y$,
    \begin{align}
        \frac{d^2}{du^2}|_{(u=0)}E(\gamma_u^i) &\geq \frac{1}{2r}\frac{d^2}{du^2}|_{(u=0)}d^2(\gamma(0),\gamma_u^i(r)) \\
        &= \frac{r}{2}\Hess_y d^2(x,y) (e_i,e_i)\,.
    \end{align}
Summing on $i=2,...,n$
\begin{align}
    \sum_{i=2}^n \frac{d^2}{du^2}|_{(u=0)}E(\gamma_u^i) \geq \frac{r}{2}\Delta_y d^2(x,y) - \frac{r}{2}\Hess_y d^2(x,y) (\dot\gamma, \dot\gamma)
\end{align}
    and we then have
    \begin{align}
        \Delta_y d^2(x,y) &\leq 2 + \frac{2}{r}\sum_{i=2}^n \frac{d^2}{du^2}|_{(u=0)}E(\gamma_u^i) \\
        &=2 + \frac{2}{r} \int_0^r (n-1) - s^2 \Ric(\dot\gamma_s,\dot\gamma_s)ds \\
        &= 2n - \frac{2}{r}\int_0^rs^2 \Ric(\dot\gamma_s,\dot\gamma_s)ds\,.
    \end{align}

    If $y$ is in the cut locus of $x$ we can construct an upper barrier $b^{\eps}(y)=\eps + d(\gamma_\eps,y)$ as in the classical case \cite[Lem.~7.1.9]{Pet16}.
\end{proof}

\begin{theorem}[Novel Characterizations of SRFs]\label{thm:charactSRF} For every smooth time-dependent Riemannian manifold $(M,g_t)_{t\in I}$, the following are equivalent:
\begin{itemize}
\item[(i)] For every $t$ and $x$, $$\Ric_{t,x}+\frac12\partial_t g_{t,x}\ge0.$$
\item[(ii)]
For every $t$ and every unit speed $d_t$-geodesic $\gamma$, 
$$\frac1{2}\Delta_y d_t^2(\gamma_0,\gamma_r) \le n+\frac2r\int_0^r \partial_{t}d_t(\gamma_q,\gamma_r)\,q\,dq.$$
\item[(iii)]
For every $t$, $x$, and $Z$
$$\k(t,x,Z)\ge0.$$
\end{itemize}
\end{theorem}
\begin{proof}
$(i)\implies (ii)$ By Lemma \ref{lem:newlaplcomp} and the SRF inequality, 
\begin{align}
    \frac1{2}(\Delta_t)_y d^2(\gamma_0,\gamma_r) &\le n-\frac2r\int_0^r \int_q^r \Ric_t(\dot\gamma_p,\dot\gamma_p) dp\,q\,dq \\
    &\leq n+\frac1r\int_0^r \int_q^r \partial_t g_t(\dot\gamma_p,\dot\gamma_p) dp\,q\,dq \\
    &= n+\frac2r\int_0^r \partial_t d_t(\gamma_q,\gamma_r) q\,dq \,.
\end{align}

$(ii) \implies (iii)$ We have
\begin{align}
    &\partial_s\a_s(t,x,Z) \\
    &= \int_Z \partial_s \left[(12\pi s)^{-n/2} e^{-\frac{d^2_{t+s}(x,y)}{12s}}\right]d\vol_{t+s}(y) + O(\sqrt s) \\
    &= (12\pi s)^{-n/2}\frac1{2s}\int_{Z(x)} \bigg[-n+\frac12\Delta d^2(x,y) -\frac{1}{6}\partial_s d^2_{t+s}(x,y) \bigg]e^{-\frac{d^2(x,y)}{12s}}d\vol(y) +O(\sqrt s)\\
    &\leq (12\pi s)^{-n/2}\frac1{2s}\int_{Z(x)} \bigg[\frac{1}{2r}\int_0^r\partial_tg_t(\dot\gamma_p,\dot\gamma_p)p^2dp -\frac{r}{6}\int_0^r\partial_t g_t(\dot\gamma_p,\dot\gamma_p)dp \bigg]\\
    &\hspace{10cm}\times e^{-\frac{d^2(x,y)}{12s}}d\vol(y) +O(\sqrt s) \\
    &=(12\pi s)^{-n/2}\frac1{2s}\int_{Z(x)} \bigg[\int_0^r\partial_tg_t(\dot\gamma_p,\dot\gamma_p)\left(\frac{3p^2-r^2}{6r}\right)dp \bigg]e^{-\frac{d^2(x,y)}{12s}}d\vol(y) +O(\sqrt s) \\
    &=(12\pi s)^{-n/2}\frac1{2s}\int_{Z(x)} \bigg[\lVert\partial_t^2g_t\rVert_{C^0(Z_r(x))} O(r^3) \bigg]e^{-\frac{d^2(x,y)}{12s}}d\vol(y) +O(\sqrt s) \,.
\end{align}

$(iii)\implies (i)$ It follows from Theorem \ref{ktZ} with $Z_\eps$ converging to a point $\theta\in S_xM$.
\end{proof}

\section{Examples in the time dependent setting}\label{sec:time-dep-ex}
Throughout the sequel, $I\subseteq \R$ will denote a non-empty open interval and $(M,d_t,m_t)_{t\in I}$ a time-dependent family of mm-spaces  with $n$-dimensional volume densities $\rho_t$ according to \eqref{densities}. The following first definition is justified by \cite{SturmSRF}, see also \cite{mccanntopp, KopferSturm18, KopferSturm21}.
 
\begin{definition}\label{mSRF} (i) We say that  $(M,d_t,m_t)_{t\in I}$ is a \emph{super-Ricci flow (SRF in short)} if for a.e. $t\in I$, for any $(\mu^a)_{a\in[0,1]}$ $W_t$-geodesic, 
\begin{align}\label{eq:dynamicconv}
\partial_a^+\mathrm{Ent}(\mu^{a}|m_t)\big|_{a=1^-}-\partial_a^- \mathrm{Ent}(\mu^{a}|m_t)\big|_{a=0^+}
\ge- \frac 12\partial_t^- W_{t^-}^2(\mu^0,\mu^1)
\end{align}
for all $\mu,\nu\in \mathcal P(M)$.
 
(ii) We say that $(M,d_t,m_t)_{t\in I}$, with $m_t$ finite measures is a \emph{minimal super-Ricci flow} if in addition for all $t\in I$,
$${\KK}_t\le0.$$
\end{definition}
Recall that in the smooth Riemannian case,  $(M,d_t,m_t)_{t\in I}$ is a minimal super-Ricci flow if and only if it is a Ricci flow.
Let us consider this concept of minimal SRF for a number of singular spaces.

\begin{remark}
    Consider a static, infinitesimally Hilbertian space $(X,d_t,m_t)_{t\in I}$, $d_t \equiv d$, $m_t\equiv m$. Then $(X,d_t,m_t)$ is a SRF iff $(X,d,m)$ is $\RCD(0,\infty)$. It is a minimal SRF iff $(X,d,m)$ is minimal in the sense of Definition \ref{def:minimalRCD}.
\end{remark}

\subsection{Time-dependent weighted Riemannian manifolds}

Consider $(M,g_t,m_t)_{t\in I}$ a smooth family of complete weighted Riemannian manifolds without boundary, $m_t=\rho_t=e^{f_t}d\vol_t$, satisfying \eqref{eq:expgrowth} for all $t\in I$, and set $\partial _tg_t=-2\Sic_t$, $S_t=\tr \Sic_t$.
Then $\a_0(t,x)=\rho_t(x)$, $\AA_0(t)=\int \rho_t^2d\vol_t$. Assume also that $\lVert \Sic_t(y) -\Sic_t(x) \rVert \leq (1+Cd^2_t(x,y))$ and $|\partial_rf_r(y)-\partial_rf_r(x)|\leq (1+Cd^2_t(x,y))$ for all $x,y\in M$, $t\in I$.

\begin{theorem}
    In the given setting,
    \begin{align}
        \k(t,x) = [R_t(x)-S_t(x)]\rho_t(x) - 3\Delta \rho_t(x)
    \end{align}
    and if $M$ is compact
    \begin{align}
        \KK(t) = \int \left[(R_t-S_t)\rho_t^2 -3 (\Delta\rho_t)\rho_t\right] d\vol(t)\,.
\end{align}
\end{theorem}

\begin{proof}
We essentially combine the calculations in Theorems \ref{thm:smoothtimedep} and \ref{thm:weightedRm}:
\begin{align}
    -\partial_s\a_s(t,x) = - R_{t+s}(x)\rho_{t+s}(x) + 3\Delta\rho_{t+s}(x) + S_{t+s}(x)\rho_{t+s}(x) + O(s^{1/2}) \,.
\end{align}
Note that the assumed control on $\partial_r f_r$ is used in the weighted analogue of \eqref{eq:control_tails_timedep}.
\end{proof}

\begin{corollary} In the given setting:
\begin{itemize}
    \item[(i)] If
$\Ric_{t}+\frac12\partial_t g_t \, {\ge0}$, then
$$ \k(t,x) \ge-3\Delta {\rho_t}(x)$$
and if $M$ is compact
$${\KK}_t
\ge 3\int|\nabla\rho_t|^2d\vol_t.$$
\item[(ii)] If
$\Ric_{t}-\text{\rm Hess}_t\, \log\rho_t+\frac12\partial_t g_t \, \ge0$, then
$$\k(t,x) \ge-3\Delta {\rho_t}(x) +\rho_t\Delta\log\rho_t= -2\Delta {\rho_t}(x)-\frac1{\rho_t}|\nabla\rho_t|^2
$$
and if $M$ is compact
$${\KK}_t =\int \k(t,\cdot) \rho_t\,d\vol_t
\ge \int|\nabla\rho_t|^2d\vol_t.$$
\end{itemize}
\end{corollary}

\begin{corollary}  Assume that  $M$ is compact and satisfies either $\Ric_{g_t}+\frac12\partial_t g_t\ge0$ or  
 $\Ric_{t}-\text{Hess}_t\, \log\rho_t+\frac12\partial_t g_t \, \ge0$. Then, for every $t\in I$, it is a minimal SRF in the sense of Definition \ref{mSRF}, i.e. ${\KK}_t=0$ if and only if
    $\rho_t$ is constant in $x$ and $\Ric_{g_t}+\frac12\partial_t g_t=0$.
 \end{corollary}

\subsection{Suspensions of circles}
Let $N:=\R/{L \Z}\simeq[0,L)$ be a circle of length $L$ and consider the spherical suspension $M:=(0,2\pi) \times_f N$, with $f(r):=\sin r$. In the sequel, with $\KK$ of the suspension $M$ we are referring to $\KK$ computed on the metric completion.
\begin{proposition}

(i) If $L<2\pi$ then $(M,g_t)_{t\in I}$ with $I=(0,\frac1{2\lambda})$ and $g_t=(1-2\lambda t)\,g$ is a SRF for any $\lambda\leq 1$ but it is never minimal (in the case $\lambda\leq0$, $I=\R_+$).

(ii) If $L>2\pi$ then there exists no rescaling $\alpha:(0,\delta)\to\R_+$ s.t. $(M,\alpha_tg)_{t\in (0,\delta)}$ is a SRF.
\end{proposition}

\begin{proof}
    (i) It is a SRF because it is so on the smooth part, and as showed in \cite[Thm. 3.3]{BacherSturm14}, no Wasserstein geodesic passes through the vertex. However it is not minimal in the sense that for a bounded neighbourhood $U$ of the vertex, 
    \begin{align}
        \KK(U)=\liminf_{s\searrow0}\frac1s\left({\AA}_0(U)-{\AA}_t(U) \right) > 0.
    \end{align}
    Indeed, if $\lambda < 1$, it is not minimal on the regular part. If $\lambda=1$, on the regular part it solves the Ricci flow equation (so $\k=0$ there), but by Theorem \ref{2d-sphsusp}, the tip gives a positive contribution to $\KK$.

    (ii) In this case, for any fixed $t$, we can find two open sets $U_1$, $U_2$ such that the transport between $\mu_i=\frac{\ind_{U_i}}{|U_i|}$ concentrates at the vertex. Then $s\mapsto \mathrm{Ent}_t(\mu_s)$ is not continuous and cannot satisfy \eqref{eq:dynamicconv} for any $\alpha_t$.
\end{proof}

\subsection{Suspensions in $n\ge3$}

In this Subsection, we exclusively consider spaces with dimensional parameter $n \geq 3$.

\begin{proposition} Let $(M,g)$ be the spherical suspension over the round sphere $N:={\mathbb S}^{n-1}(\rho)$ of dimension $n-1$ and radius $\rho$.

(i) If $\rho<1$ then the linear shrinking space $(M,g_t)_{t\in I}$ with $I=(0,\frac1{2\lambda})$ and   $g_t=(1-2\lambda t)\,g$, $\lambda \leq n-1$, is a SRF but it is not minimal.

(ii)  If $\rho>1$ then there exists no rescaling $\alpha:(0,\delta)\to\R_+$ s.t. $(M,\alpha_tg)_{t\in (0,\delta)}$ is a SRF.
\end{proposition}

\begin{proof}
    The curvature of $M$ is, for $X\in TN$, 
    \begin{align}
        \Ric(\partial_r+X) = (n-1)|\partial_r|^2 + \frac{n-2+\rho^2-(n-1)\rho^2\cos^2(r)}{\rho^2\sin^2r}g_M(X,X)\,.
    \end{align}
    
    (i) If $\rho<1$, the curvature is lower bounded by $n-1$, hence the regular part is a SRF for any $\lambda\leq n-1$. As in the $\dim = 2$ case, no Wasserstein geodesic passes through the vertex by \cite[Thm.~3.3]{BacherSturm14} so it is a SRF in the sense of Definition \ref{mSRF}. However it is not minimal because it does not satisfy the Ricci flow equation on the regular part, and the tips do not contribute, see Lemma \ref{lem:n-geq-3-susp}.

    (ii) If $\rho>1$ the curvature approaches $-\infty$ close to the tip, hence no rescaling can satisfy the SRF inequality.
\end{proof}

\begin{proposition} Let $(M,g)$ be the spherical suspension over $(N,h)$, a compact $(n-1)$-dimensional Einstein manifold with constant $(n-2)$. Put $I=(0,\frac1{2(n-1)})$ and $g_t:=(1-2(n-1)t)g$. Then $(M,g_t)_{t\in I}$ is a minimal SRF.
\end{proposition}

\begin{example}
One can choose for example 
\begin{itemize}
\item $N={\mathbb S}^{n-1}(1)/\Z_2$ for some $n\ge3$,
\item
or $N={\mathbb S}^2(1/\sqrt 3)\times {\mathbb S}^2(1/\sqrt 3)$ and $n=5$.
\end{itemize}
Note that in these cases $X$ is not a topological manifold.
\end{example}

\begin{proof}
    All these spaces satisfy $\Ric=(n-1)g$, hence the defined flow solves the RF equation on the regular part. Since the Wasserstein geodesics do not pass through the tips \cite[Thm. 3.3]{BacherSturm14}, it is a SRF.

    In order to show that it is minimal we show that the tips do not contribute in $\KK(t)$.

    Write the time-dependent scaling factor $\lambda_t := \sqrt{1 - 2(n - 1)t}$, and write $(X,d_t = \lambda_td,m_t = \lambda_t^n m)$ explicitly as the metric completion of the Riemannian manifold
    
    \begin{equation}
    (M,g_t) := \big((0,\pi)_r\times N, \lambda_t^2(dr^2 + \sin^2(r)\cdot g_N)\big)\,.
    \end{equation}
    
    Note that after changing coordinates $\rho := \lambda_t r$ we land in the setting covered by Lemma \ref{lem:n-geq-3-susp}, with $[0,R(t)] = [0,\lambda_t\pi]$ and $f(t,r) := \lambda_t\sin(r/\lambda_t)$. As noted previously, it suffices to show that the tips $o_0 := r^{-1}(0)$, $o_\pi := r^{-1}(\pi)$ do not contribute to $\KK(t)$, so let $x \in B_\eps(o_0)$ for an $\eps > 0$ that we will eventually send to $0$. Following the proof and notation of the smooth case Theorem \ref{thm:smoothtimedep}, we can write 

    \begin{equation}
    \partial_s \a_s(t,x) = I + II\,,
    \end{equation}

    where

    \begin{align}
    I &= \partial_s \left[(12 \pi s)^{-n/2} \int e^{-\frac{d^2_\tau(x,y)}{12 s}} d\vol_r (y)\right]_{r = t + s} \\
    &\phantom{=}- (12\pi s)^{-n/2}\int e^{-\frac{d^2_{s + t}(x,y)}{12 s}}\left.\frac{\partial_r d_r^2(x,y)}{12 s}\right|_{r = t + s}\,d\vol_{t + s}(y) \\
    &= O_{N,t}(d_{s + t}^{-2}(o_0,x)) + 2\frac{(n - 1)}{\lambda_{t + s}^2}\int_X \frac{d^2_{t+ s}(x,y)}{12s} e^{-\frac{d^2_{t+ s}(x,y)}{12s}}\frac{d\vol_{t + s}(y)}{(12\pi s)^{n/2}}\,.
    \end{align}

    Note that to pass to the last line, we used the uniform estimates guaranteed by \eqref{eq:susp-unif-conv}, along with the definition of $d_{r}$. The second term on the RHS above is uniformly bounded $\leq C(N,t)$ jointly in $(s,x) \in (0,\big((2(n - 1))^{-1} - t\big)/2] \times X$. We thus obtain

    \begin{align}
        \int_{B^{d_{s + t}}_\eps(\{o_0,o_\pi\})}|I|\,dm_t(x) &= 2\int_{B_\eps^{d_{s + t}}(o_0)} |I| \,dm(x) \\
        &\leq \int_{B_\eps^{d_{s + t}}(o_0)} \frac{C(N,t)}{d_{t + s}^2(o_0,x)} + C(N,t)\,dm(x) = O_{N,t}(\eps^{n - 2})\,.
    \end{align}

The other term $\int_{B_\eps(\{o_0,o_\pi\})}|II|\,dm_t(x)$ can be estimated by $O_N(\sqrt{s})$ as in the proof of Theorem \ref{thm:smoothtimedep}. Sending $s\to 0$, then $\eps \to 0$ shows that there is no contribution from the tips, i.e. $\KK(t) \equiv 0$.
\end{proof}
\addtocontents{toc}{\protect\setcounter{tocdepth}{1}}
\begin{appendices}
\section{Proof of Theorem \ref{thm:C0gluing}}\label{app:C0}

Let $(\Omega_1^n,g_1)$, $(\Omega_2^n,g_2)$ be smooth compact Riemannian manifolds with boundary $\iota_i: \partial\Omega_i \hookrightarrow \Omega_i$, $i = 1,2$, that are isometric via $\Phi: (\partial\Omega_1,\iota_1^*g_1)\stackrel{\cong}{\to} (\partial \Omega_2,\iota_2^* g_2)$. Write their gluing $(X^n := \Omega_1 \cup_{\Phi} \Omega_2,d,\mu)$: a smooth closed manifold such that $\Omega_1, \Omega_2 \hookrightarrow X$ embed smoothly. The distance and measure are induced by a Riemannian metric $g$ that has $\left.g\right|_{\Omega_i} = g_i$ for $i = 1,2$ (i.e. $g$ is smooth on $\Omega_1$ and $\Omega_2$ separately, and globally Lipschitz). The differential structure of $X$ is given by using $\Phi$ to paste together tubular neighborhoods $B_\eps^{g_i}(\partial \Omega_i)\subseteq \Omega_i$, $i = 1,2$, such that the (signed) distance function to the boundary is smooth across it (see \cite{Kos02, Sch12}). We wish to compute $\left.\partial_s\right|_{s = 0}\int_X \a_s\,d\mu$ in this setting.

The strategy is to compute the Laplacian of the distance function on the other side of the boundary from any sufficiently close basepoint. In practice, this means that one computes how the boundary affects the Jacobi fields of geodesics passing through it; one observes a distributional acceleration (``impulse'') given by the second fundamental form of the boundary. Luckily, only geodesics that pass through the boundary at most once will contribute to the curvature functional.

\paragraph{Geometric preliminaries}

We decompose $X = \Omega_1^\circ \sqcup \Sigma \sqcup \Omega_2^\circ$, where $\Sigma$ is the image of $\partial\Omega_1,\partial\Omega_2$ under the quotient map defining $X$. We let $\eps > 0$ be a fixed parameter, which will be constrained when necessary throughout the proof by requiring $\eps \leq \eps_0(g_1,g_2)$. Fix a point $x\in B_\eps(\Sigma)\cap \Omega_1^\circ$, and denote $r:= d(\cdot, x):X\to [0,\infty)$ the radial distance function. We will use the terminology {\it $d$-geodesics} to refer to (rectifiable) minimizing geodesics for $(X,d)$\footnote{The example given by gluing two copies of $\R^{2}\setminus B_1^2(0^2)$ exhibits branching of $d$-geodesics. However, branching of $d$-geodesics, or more generally $C^1$ curves $\gamma$ that are smooth and solve the geodesic equations on $\gamma \cap\Omega_1^\circ$, $\gamma \cap\Omega_2^\circ$, can only occur at tangential intersections with $\Sigma$. Otherwise, unique solvability of the geodesic equations given initial data would be violated on the smooth Riemannian manifolds $(\Omega_1, g_1)$ or $(\Omega_2, g_2)$.}. Such curves are globally $C^1$ (see \cite{LytYam06}), and smooth on the $C^\infty$ Riemannian manifolds $\Omega_1^\circ, \Omega_2^\circ$ where they solve the $g_1$ resp. $g_2$ geodesic equations. We claim that we can decompose $X$ as follows, and delay the proof until the final step (in \eqref{eq:Ex-def}) 

\begin{claim}\label{claim:geo}

There are sets $E_x \subseteq X$, $W_x := \Omega_1^\circ\setminus E_x$, $V_x := \Omega_2^\circ \setminus E_x$ satisfying the following properties:

\begin{itemize}
\item $W_x$ is open, $ \subseteq B_\eps(x)$, and $\left.r^2\right|_{\overline{W_x}}$ is smooth. For each $z \in W_x$ there is a unique unit-speed $d$-geodesic $\gamma^z$ connecting $x$ to $z$, and it does not intersect $\Sigma$ (i.e. lies entirely in $\Omega^\circ_1$). Moreover, $r^2$ is smooth on a neighborhood of $\gamma^z$.
\item $V_x$ is open, $\subseteq B_\eps(x)$, and $\left.r\right|_{\overline{V_x}}$ is smooth. For each $y \in V_x$, there is a unique unit-speed $d$-geodesic $\gamma^y$ connecting $x$ to $y$. $\gamma^y$ passes through $\Sigma$ transversely at a single point $\omega_y = \gamma^y_{\rho_y}$ in the effective sense

\begin{equation}\label{eq:small-inner-prod-small-vol}
r(y) \leq \eps\La \dot\gamma^y_{\rho_y}, \nu_{\omega_y}\Ra\,,\qquad \forall y \in V_x\,,
\end{equation}

where $\nu\in \Gamma(\left.TX\right|_\Sigma)$ is the normal vector pointing into $\Omega_2^\circ$. There is even a gap property: $\gamma^y$ is unique among $C^1$ unit-speed length $\ell \leq 2r(y)$ curves $\gamma$ connecting $x$ to $y$, that have a single time of intersection $t_0$ with $r(y)\leq 2\eps\La \dot\gamma_{t_0}, \nu_{\gamma_{t_0}}\Ra$, and that are smooth and solve the $g_1$ and $g_2$ geodesic equations on $\gamma\cap \Omega_1^\circ$ resp. $\gamma\cap \Omega_2^\circ$. Moreover, $r^2$ is smooth on a neighborhood of $\left.\gamma^y\right|_{[0,\rho_y]} \subseteq \Omega_1$ and $\left.\gamma^y\right|_{[\rho_y,r(y)]}\subseteq \Omega_2$, and $C^{1,1}$ on a neighborhood of $\gamma^y$. 
\item $E_x$ is the closure of an open set with piecewise smooth boundary. For every $\omega \in \Sigma \setminus E_x$, there is $y \in V_x$ such that $\omega = \omega_y$. There is a monotone function $f = f_\eps:(0,1] \to (0,\infty)$ with $\lim_{r\searrow 0}f(r) = 0$ and such that

\begin{equation}\label{eq:error-volume-decay}
    B_r(x)\cap E_x =\emptyset\,,\quad \forall r \in (0,\eps]\,, x \in \Omega_1^\circ \setminus B_{f(r)r}(\Sigma)\,.
\end{equation}
Finally, $\{(x,y)\mid x \in B_\eps(\Sigma) \cap \Omega_1^\circ\,,\,y\in X\setminus E_x\}\subseteq \Omega_1^\circ\times \Omega_2^\circ$ is open.
\end{itemize}
\end{claim}

The bulk of the proof will be devoted to computing in $V_x$, which is the source of the boundary term in Theorem \ref{thm:C0gluing}. The contribution from $W_x$ is the scalar curvature already familiar from the smooth setting, but integrated over an $\eps$-small set. Finally, $E_x$ will not contribute at all.

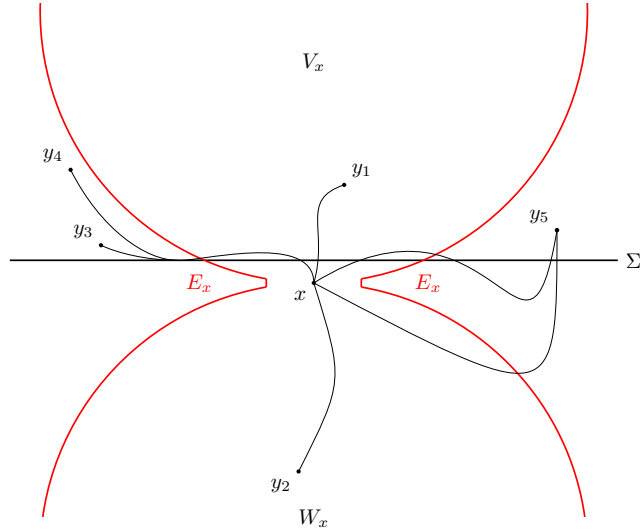
\begin{figure}[h]\label{fig:EVW-regions}
\centering
\scalebox{0.8}{\begin{tikzpicture}[scale=0.5]

\def\h{0.75} 
\def\r{9} 
\def\a{10} 

\coordinate (x) at (0,-\h);

\draw[black, thick] (-10,0) -- (10,0)
    node[below, right] {$\Sigma$}
    node[midway , below=4cm, text=black] {$W_x$}
    node[midway, above=3cm, text=black] {$V_x$};

\begin{scope}
    \clip (-10,-8.5) rectangle (10,8.5);
\draw[red, thick]
    (x) ++({\r*sin(\a)}, {\r*(1 - cos(\a))}) arc (-{90+\a}:{270-\a}:\r);
\draw[red, thick]
    (x) ++({\r*sin(\a)}, {\r*(1 - cos(\a))}) --++(0,-{2*\r*(1 -cos(\a))})
    node[midway, right=0.75cm] {$E_x$};
\draw[red, thick]
    (x) ++({-\r*sin(\a)}, {\r*(cos(\a) - 1)}) arc ({-270+\a}:{90-\a}:\r);
\draw[red, thick]
    (x) ++({-\r*sin(\a)}, {\r*(cos(\a) - 1)}) --++(0,{2*\r*(1 -cos(\a))}) node[midway, left=0.75cm] {$E_x$};
\end{scope}

\node[below left] at (x) {$x$};
\fill (x) circle (2pt);

\coordinate (y1) at (1,2.5);
\draw (x) .. controls (0.5,1) and (-0.5,2) .. (y1);
\node[above right] at (y1) {$y_1$};
\fill (y1) circle (2pt);

\coordinate (y2) at (-0.5,-7);
\draw (x) .. controls (1,-4) .. (y2);
\node[below left] at (y2) {$y_2$};
\fill (y2) circle (2pt);

\coordinate (y3) at (-7,0.5);
\draw (x) .. controls (0,1) and (-4,-0.1) .. (-4.5,0) .. controls (-5.5,0) and (-6.5,0.25) .. (y3);
\node[above left] at (y3) {$y_3$};
\fill (y3) circle (2pt);
\coordinate (y4) at (-8,3);
\draw (-4.5,0) .. controls (-5.5,0) and (-7, 1) .. (y4);
\node[above left] at (y4) {$y_4$};
\fill (y4) circle (2pt);

\coordinate (y5) at (8,1);
\fill (y5) circle (2pt);
\draw (x) .. controls (8,-5) .. (y5);
\draw (x) .. controls (6,3) and (7,-5) .. (y5);
\node[above left] at (y5) {$y_5$};
\fill (y5) circle (2pt);

\end{tikzpicture}}
\caption{Radial $d$-geodesics ending in the regions $W_x$, $V_x$, and $E_x$.}\label{cap:EVW-regions}
\end{figure}

\paragraph{Hessian and Jacobi fields at the boundary} Fix a point $y\in V_x$, and let $\gamma = \gamma^y:[0,\ell] \to X$ be the unit-speed $d$-geodesic with $\gamma_0 = x$, $\omega = \omega_y = \gamma_\rho \in \Sigma$, and $\gamma_{\ell} = y$. We wish to compute the Hessian $({\nabla^2}r^2)_y$ to high enough order in $r$. To do so, we define two collections of ``Jacobi-like'' vector fields along $\gamma$. 

On one hand, let

\begin{equation}\label{eq:fake-Jacobi}
\tilde J_i:\begin{cases}
\tilde J_i '' + \Rm(\tilde J_i,\gamma')\gamma' = 0\,,\\
\tilde J_i(0) = 0\,,\, \tilde J_i(\rho) = v_i\,,\end{cases}
\end{equation}

where $v_1 := \nabla r$, while $v_2,\ldots,v_n$ are chosen to form an ONB of $T_\omega \Sigma$. At $\gamma_\rho$ where there is a jump discontinuity in $\Rm$, \eqref{eq:fake-Jacobi} is interpreted as restarting the ODE w.r.t $g_2$ at time $\rho^+$ with initial data given by the final data of the ODE w.r.t $g_1$ at time $\rho^-$. Note that we needed to assume $\ell \leq\eps \leq \mathrm{conj}(g_1)$ to find $ \tilde J_i'(0)$ so that $\tilde J_i(\rho) = v_i$.

On the other hand, let $J_i$ be vector fields that satisfy $J_i(0) = 0$, $J_i(\rho) = v_i$, given by variations $\left.\partial_\eta\right|_{\eta =0} \gamma^{\eta}$ among constant-speed $d$-geodesics $\gamma^{\eta}$, $\gamma^{0} = \gamma$. More carefully, define $C^1$ curves $\gamma^{\eta}:[0,\ell] \to X$ with the desired $\left.\partial_\eta\right|_{\eta =0} \gamma^{\eta}_\rho = v_i$ and $\gamma^\eta_0 \equiv x$ by integrating the Jacobi field $\tilde J_i$ on $[0,\rho]$ to obtain a $g_1$-geodesic variation $\left.\gamma^\eta\right|_{[0,\rho]}\subseteq \Omega_1$, then define $\gamma_s^\eta := \exp_{\gamma^\eta_{\rho}}^{g_2}((s - \rho)\cdot\dot\gamma_\rho^\eta)$ for $s\in[\rho,\ell]$. These curves are $d$-geodesics for $\eta$ small enough in view of the gap property from Claim \ref{claim:geo}. By smoothness of the metric and $\gamma$ away from $\Sigma$, $\left.J_i\right|_{[0,\rho]}$ and $\left.J_i\right|_{[\rho, \ell]}$ are smooth and solve the Jacobi equation of \eqref{eq:fake-Jacobi}. At time $\rho$, $J_i$ is continuous by construction, but $J_i'$ may not be. Indeed for $i > 1$ there is such a discontinuity, which can be explicitly calculated. Noting that $\left.[r\nabla r,J_i]\right|_{\gamma_{[\rho,\ell]}}=\left.[s\partial_s \gamma_s^\eta,\partial_\eta \gamma_s^\eta]\right|_{\gamma_{[\rho,\ell]}}=0$ for any extension of $J_i$ to $\Omega_2$ (by commutativity of mixed partials of the map $(s,\eta)\mapsto \gamma_s^\eta \in X$),

\begin{align}
\La J_i', J_j\Ra (\rho^+)&= \frac{1}{r}\La \prescript{g_2}{}{\nabla_{r\nabla r}J_i},J_j\Ra(\rho) \\
&= \frac{1}{2r}\La \prescript{g_2}{}{\nabla_{J_i}\nabla r^2}, J_j\Ra(\rho) \\
&= \frac{1}{2r}\big(D_{J_i}D_{J_j}r^2 - \La \prescript{g_2}{}{\nabla_{J_i}J_j}, \nabla r^2\Ra\big) (\rho)\,.
\end{align}
Computing the same quantity at $\rho^-$, because the term $D_{J_i}D_{J_j}r^2$ is in common we get
\begin{align}
\La J_i', J_j\Ra (\rho^+)&= \La J_i', J_j\Ra (\rho^-) + \frac{1}{2r}\Big(\La \prescript{g_1}{}{\nabla_{J_i}J_j}, \nabla r^2\Ra - \La \prescript{g_2}{}{\nabla_{J_i}J_j}, \nabla r^2\Ra\Big)(\rho) \\
&= \La \tilde J_i', J_j\Ra(\rho) - \La \nu_\omega,\nabla r\Ra(II_{\partial\Omega_2}+ II_{\partial \Omega_1})_\omega(v_i,J_j)(\rho)\\
 &\phantom{=}+(II_{\partial \Omega_2} + II_{\partial \Omega_1})_\omega(v_i, \nabla r)\La \nu_\omega, J_j\Ra(\rho)\,,\qquad i = 2,\ldots, n\,;\, j=1,\ldots, n\,.
\end{align}

The second equality follows from writing $J_j$, $\nabla r^2$ as sums of parts $\bot, \top$ to $\Sigma$, then computing that the $\top,\top$ and $\bot,\bot$ terms are intrinsic to $(\left.TX\right|_\Sigma,\left.g_1\right|_{\Sigma} = \left.g_2\right|_{\Sigma})$ and cancel while the mixed terms amount to the boundary defect in the RHS above. Recall that our sign convention when computing the second fundamental form of $\partial\Omega_i\hookrightarrow \Omega_i$ for both $i = 1,2$ is to use the respective outward-pointing normal vector. For notational conciseness, we label the defect terms

\begin{equation}\label{eq:boundary-error-term}
V_i:=  (II_{\partial \Omega_2} + II_{\partial\Omega_1})_\omega(v_i,(\nabla r)_\omega)\nu_\omega-\La \nu_\omega,(\nabla r)_\omega\Ra (II_{\partial\Omega_2}+ II_{\partial \Omega_1})_\omega(v_i,\cdot)^\sharp\,,\qquad i=2,\ldots,n\,.
\end{equation}

The above is in principle only defined at $\omega = \gamma_\rho$, but we extend $V_i$ and $v_i$ along $\gamma$ using parallel transport. We use this computation at $\omega = \gamma_\rho$ along with smoothness of $g_2$ to approximate the $J_i'$ and $J_i$ further along $\gamma$ for $i = 2,\ldots, n$ by

\begin{align}\label{eq:J_i_tilde}
\frac{1}{2r}(\nabla^2r^2)_{\gamma(\rho + t)}(J_i,\cdot)^\sharp = J_i'(\rho + t) &=  \tilde J_i'(\rho + t) +V_i + O(t^2)\,,\\ 
J_i(\rho + t) &= \tilde J_i(\rho + t) + tV_i + O(t^3) \label{eq:fake-real-jac-field}\,.
\end{align}

Note the $O(t^2)$ resp. $O(t^3)$ improvement above, which follows from \cite[Thm~6.5.3]{Jos17}, then \cite[Cor~6.5.1]{Jos17} applied to the Jacobi field with initial data $(J_i - \tilde J_i)(\rho) = 0$, $(J_i - \tilde J_i)'(\rho^+) = V_i$. Here and later, the implicit constant in such asymptotic notation is allowed to depend on $g_1,g_2$, but is otherwise uniform. For the final direction $J_1(t +\rho) = (r/\rho)\nabla r$, we have 
\begin{align}
    J_1' = \frac{1}{2r}(\nabla^2 r^2)((r/\rho)\nabla r,\cdot)^\sharp = \frac{r}{\rho}\nabla_{\nabla r}\nabla r + \frac{1}{\rho}|\nabla r|^2\nabla r = \frac{1}{\rho}\nabla r = \tilde J_1'
\end{align}
because $r$ is a distance function. To avoid treating $i = 1$ as a separate case, it is notationally convenient to define $V_1 := 0$.

\paragraph{The Euclidean and boundary correction terms}

To compute $\Delta r^2$ it suffices to raise the $(0,2)$ Hessian to a $(1,1)$ Hessian, i.e. multiply the matrix representation of $\nabla^2 r^2$ by the inverse of the metric in the chosen basis, then take the trace. Our chosen basis is $J_i$, so we must compute and invert the metric in this basis.

Using bounded curvature of $g_1$ on $[0,\rho]$ then bounded curvature of $g_2$ on $[\rho,\ell]$, the following very rough ``almost-Euclidean'' estimates for the $\tilde J_i$ hold 
\begin{align}\label{eq:bad-Jacobi-ests}
\tilde J_i(\ell) &= \frac{1}{\rho}\big(\ell v_i + O(\ell^3)\big)\,,\qquad \tilde J_i'(\ell) = \frac{1}{\rho}\big(v_i + O(\ell^2)\big)\,,\qquad i = 1,\ldots, n\,. 
\end{align}

(More precisely, by definition $\tilde J_i(0) = 0$ and $\tilde J_i(\rho) = v_i$, so applying \cite[Cor~6.5.1]{Jos17} to the initial data $\tilde J_i(0)$, $ \tilde J_i'(0)$ yields $\rho\tilde J_i'(\rho) = v_i + O(\rho^2)$. Applying now \cite[Thm~6.5.3]{Jos17} on the other side of the boundary with initial data $\rho\tilde J_i(\rho) = \rho v_i$ and $\rho \tilde J_i'(\rho) =v_i + O(\rho^2)$ (or more precisely to the colinear initial data $(\rho v_i, v_i)$ and $(0,O(\rho^2))$ separately), we obtain $\rho \tilde J_i(\ell) = \rho v_i + t(v_i + O(\rho^2)) + (O(t^3) + \rho O(t^2)) = \ell v_i + O(t\ell^2)$. Then, \cite[Cor~6.5.1]{Jos17} (or more precisely a variant of the statement with nonzero zeroth order initial data $J(0) \neq 0$ and the conclusion $|J(t) - P_t(J(0)) - tJ'(t)| \leq C|J(t)|t^2$, which uses the same proof) yields $\rho \tilde J_i'(\ell) = v_i + O(\ell^2)$ as desired.)

\medskip

From now on, we let $[A_{ij}]_{ij}$ be the $n\times n$ matrix with entry $A_{ij}$ in the $i^\text{th}$ row, $j^\text{th}$ column. We name the metric tensor in the respective basis of vector fields as $M := [\La J_i,J_j \Ra]_{ij}$. We fix $t := \ell - \rho = r(y) - r(\omega)$, and are careful when accumulating errors to avoid any implicit dependence on the (uncontrolled) ratios $(\rho/t)^\pm$, $\ell/\rho$, or $\ell/t$, though we may use $t/\ell,\rho /\ell \leq 1$ to simplify the smallness of errors. Similarly, we cannot treat $\La \nu_\omega,(\nabla r )_\omega\Ra^{-1}$ as uniformly bounded, but we may use $\ell \leq \eps \La \nu_\omega,(\nabla r )_\omega \Ra$ from \eqref{eq:small-inner-prod-small-vol} to simplify errors. We now compute, using \eqref{eq:fake-real-jac-field} and \eqref{eq:bad-Jacobi-ests},

\begin{align}
M &= [\La \tilde J_i + tV_i + O(t^3), \tilde J_j + tV_j +O(t^3)\Ra]_{ij} \\
&= [\La (\ell/\rho)v_i + tV_i + O(\ell^3)/\rho+ O(t^3), (\ell/\rho)v_j + tV_j + O(\ell^3)/\rho +O(t^3)\Ra]_{ij} \\
&= (\ell/\rho)^2[\La v_i,v_j\Ra]_{ij} + t(\ell/\rho)[\La v_i,V_j\Ra + \La V_i,v_j \Ra]_{ij} + (\ell/\rho)^2O(\ell^2) \\
&= (\ell/\rho)^2\Big([\La v_i,v_j\Ra]_{ij} - 2t(\rho/\ell)\La \nu_\omega,(\nabla r)_\omega\Ra[\big(0 \oplus (II_{\partial \Omega_1} + II_{\partial \Omega_2})_\omega\big)(v_i,v_j)]_{ij} + O(\ell^2)\Big) \\
&= (\ell/\rho)^2\Big([\La v_i,v_j\Ra]_{ij} + O(\ell \La \nu_\omega,(\nabla r)_\omega\Ra) + O(\ell^2)\Big)
\end{align}

Note that to evaluate $[\La V_i,v_j\Ra]_{ij}$, we used $\La v_1, V_i\Ra = 0$ for $i = 1,\ldots, n$, while $v_j$ is orthogonal to the first term in \eqref{eq:boundary-error-term} for $j = 2,\ldots, n$. We now compute the first term of the RHS above. In the (generic) case when $(\nabla r)_\omega \cancel{\perp} T_\omega\Sigma$, it is convenient to constrain the choice of $v_i$ for $i =2,\ldots, n$ by requiring

\begin{equation}
v_i: \begin{cases}
    = \frac{(\nabla r)_\omega - \La (\nabla r)_\omega,\nu_\omega\Ra\nu_\omega}{\sqrt{1 - \La (\nabla r)_\omega, \nu_\omega\Ra^2}}\,,& i =2\,,\\
    \in T_\omega \Sigma \cap \La(\nabla r)_\omega^\perp \Ra\,,&i=3,\ldots, n\,.
\end{cases}
\end{equation}

With this choice of $v_i$, we have

\begin{align}\label{eq:euclidean-matrix}
[\La v_i,v_j\Ra]_{ij} &= 
\mat{1 & a & 0 &\cdots & 0 \\
a &  &  &  & \\ 0 &  &&&\\
\vdots &  && \mathrm{Id}_{n - 1}& \\0 &} \qquad\text{where}\quad a := \sqrt{1 - \La \nu_\omega,(\nabla r)_\omega\Ra^2}\,.
\end{align}

This matrix is easily invertible (this was the reason for making the particular choices for $v_i$) and has smallest eigenvalue $1 - a \geq \La \nu_\omega, (\nabla r)_\omega\Ra^2/2$. In order to invert $M$, we use the inverse approximation $(A + B)^{-1} = (\Id_n + A^{-1} B)^{-1}A^{-1}=A^{-1} + O(\|A^{-1}\|^2\|B\|))$, which holds as long as $\|A^{-1}\| \|B\| < 1$. For us $A = [\La v_i,v_j\Ra]_{ij}$, while $B = O(\ell \La \nu_\omega, (\nabla r)_\omega\Ra + O(\ell^2))$. From the smallest eigenvalue lower bound we have $\|[\La v_i,v_j\Ra]_{ij}^{-1}\| \leq 2\La \nu_\omega, (\nabla r)_\omega\Ra^{-2}$, so \eqref{eq:small-inner-prod-small-vol} shows that $\|A^{-1}\| \|B\| < 1$ is satisfied as long as $\eps \leq \eps_0(g_1,g_2)$. The inverse approximation then gives 

\begin{equation}
M^{-1} = (\rho/\ell)^2\big([\La v_i,v_j\Ra]_{ij}^{-1} + \eps O(\La \nu_\omega, (\nabla r)_\omega\Ra^{-2})\big)\,.\label{eq:J-metric-inverse-approx}
\end{equation}

\begin{remark}\label{rmk:no-conj-points}
Recall that the matrix $M$ inverted in \eqref{eq:J-metric-inverse-approx} is the metric $g_2$ expressed in the collection of vector fields $J_i$ arising from variations of $\gamma$ among $d$-geodesics. Since $M$ is nonsingular, the $J_i$ form a (full rank) basis for $T_{\gamma}X$. This last fact holds in larger generality than we have assumed here, without changing the argument. We used that $\gamma$ was a $C^1$ curve beginning at $x$ and intersecting $\Sigma$ transversely, solving the geodesic equation on $\gamma \cap \Omega_1^\circ$ and $\gamma \cap \Omega_2^\circ$, that the (smooth on $\gamma \cap \Omega_1$ and $\gamma \cap \Omega_2$, globally Lipschitz) $J_i$ were produced by a variation among such curves, and that there was some (smooth on $\Omega_1$ and $\Omega_2$, globally $C^{1,1}$) function $r$ defined on a neighborhood of $\gamma$ with $( r\circ \gamma)(t) = t$, $\dot\gamma_t = (\nabla r)_{\gamma_t}$, and $[J_i, r\nabla r] = 0$ for any extension of $J_i$ to $ \Omega_1$, $ \Omega_2$. We did not use that $\gamma$ was a (minimizing) $d$-geodesic, nor did we use that $r$ was given by $d(x,\cdot)$. However, $\eps \leq \eps_0(g_1,g_2)$ in \eqref{eq:small-inner-prod-small-vol} was crucial to obtaining the nonsingularity\footnote{This condition is geometrically meaningful. It is ensuring that we do not encounter a conjugate point, since there is no unconditional lower bound on the injectivity radius that holds in this setting. The simplest example is the doubling of $B^2_1(0^2)\subseteq \R^2$, where there are choices of points $x,y$ with distance $\to 0$ that are conjugate, and also have $\La \nu_\omega,(\nabla r)_\omega\Ra \to 0$.} of $M$.
\end{remark}

\paragraph{The Laplacian of $d(x,\cdot)$ to high enough order}

We have a formula for $\nabla^2 r^2$ on $\gamma_{(\rho, \ell]}\subseteq\Omega_2^\circ$, written in terms of a frame $J_i$ for which we also have an approximation. We now multiply it by $M^{-1}$, computed to acceptable error in the previous step, then take the trace.

First, we relate $J_i'$ and $J_i$ using \eqref{eq:bad-Jacobi-ests} and \eqref{eq:fake-real-jac-field}:

\begin{align}
J_i' &= \tilde J_i' + V_i + O(t^2) \\
&= v_i/\rho + V_i + O(\ell^2)/\rho + O(t^2) \\
&= \tilde J_i/\ell + V_i + O(\ell^2)/\rho \\
&= (\tilde J_i + tV_i)/\ell + (\rho/\ell)V_i + O(\ell^2)/\rho = J_i/\ell + (\rho/\ell)V_i + (\ell/\rho)O(\ell) \,.
\end{align}

We can thus compute, using the definition of $M$, \eqref{eq:bad-Jacobi-ests}, and \eqref{eq:J-metric-inverse-approx} (along with the inverse bound $\|[\La v_i,v_j\Ra]_{ij}^{-1}\| \leq 2\La \nu_\omega, (\nabla r)_\omega\Ra^{-2}$),

\begin{align}
M^{-1}\cdot[\La J_i', J_j\Ra]_{ij} &= M^{-1}\cdot[\La J_i/\ell + (\rho/\ell)V_i + (\ell/\rho)O(\ell), J_j\Ra]_{ij} \\
&= \Id_n/\ell + (\rho/\ell)M^{-1}\cdot[\La V_i, (\ell/\rho)v_i + O(\ell^3)/\rho\Ra]_{ij} +(\ell/\rho)^2 O(\ell\|M^{-1}\|) \\
&=\Id_n/\ell + (\rho/\ell)^2([\La v_i,v_j\Ra]_{ij}^{-1} + \eps O(\La \nu_\omega, (\nabla r)_\omega\Ra^{-2}))\cdot[\La V_i, v_j \Ra]_{ij} \\
&\phantom{=}+\eps O(\La \nu_\omega, (\nabla r)_\omega\Ra^{-1})\,. \\
\end{align}

In order to continue, we need to use \eqref{eq:boundary-error-term} as before in the form $[\La V_i, v_j\Ra]_{ij} = -\La \nu_\omega,(\nabla r)_\omega\Ra[\big(0 \oplus (II_{\partial \Omega_1} + II_{\partial \Omega_2})_\omega\big)(v_i,v_j)]_{ij}$. This concludes the $(1,1)$ Hessian calculation up to acceptable error

\begin{align}
M^{-1}\cdot[(1/2r)\nabla^2 r(J_i,J_j)]_{ij} &= M^{-1}\cdot[\La J_i', J_j\Ra]_{ij} \\
&= \Id_n/\ell - (\rho/\ell)^2\La \nu_\omega,(\nabla r)_\omega\Ra[\La v_i,v_j\Ra]_{ij}^{-1}\cdot[\big(0 \oplus (II_{\partial \Omega_1} + II_{\partial \Omega_2})\big)(v_i,v_j)]_{ij} \\
&\phantom{=}+\eps O(\La \nu_\omega, (\nabla r)_\omega\Ra^{-1})\,.
\end{align}

Taking the trace of both sides (leisurely computing the inverse of $[\La v_i,v_j\Ra]_{ij}$ by examining \eqref{eq:euclidean-matrix}) now gives

\begin{align}
\frac{1}{2r(y)}\Delta r^2(y) &= n/\ell - (\rho/\ell)^2\La \nu_\omega,(\nabla r)_\omega\Ra\left(m_{\partial\Omega_1} + m_{\partial\Omega_2} + \left(\frac{1}{\La \nu,\nabla r\Ra^2} - 1\right)(II_{\partial\Omega_1} + II_{\partial\Omega_2})(v_2,v_2)\right)_\omega \\
&\phantom{=}+ \eps O(\La \nu_\omega,(\nabla r)_\omega\Ra^{-1})\,.
\end{align}

Here $m_{\Omega_i}$ denotes the trace of $II_{\partial\Omega_i}$ (in particular, the sum of eigenvalues has not been normalized by the dimension).

\paragraph{The $s \to 0$ asymptotics of $\int_X \a_s(x)\,d\mu(x)$}

Now on to the differentiated integral, restricted to points $x$ within some small but fixed $\eps > 0$ of the boundary: 

\begin{align}
\partial_s \int_{B_\eps(\Sigma)\cap \Omega_1^\circ} \a_s(x)\,d\mu(x) &=\frac{1}{s}\int_{B_\eps(\Sigma)\cap \Omega_1^\circ} \int_{W_x\cup V_x \cup E_x} \left(-\frac{n}{2} + \frac{d^2(x,y)}{12s}\right)\frac{e^{-d^2(x,y)/12s}}{(12\pi s)^{n/2}}\,d\mu(y)d\mu(x) \\
&= \frac{1}{s}\int_{B_\eps(\Sigma)\cap \Omega_1^\circ} \int_{W_x\cup V_x} \left(-\frac{n}{2} + \frac{\Delta_y d^2(x,y)}{4}\right)\frac{e^{-d^2(x,y)/12s}}{(12\pi s)^{n/2}}\,d\mu(y)d\mu(x) \\
&\phantom{=}+ \frac{1}{s}\int_{B_\eps(\Sigma)\cap \Omega_1^\circ} \int_{E_x} \left(-\frac{n}{2} + \frac{d^2(x,y)}{12s}\right)\frac{e^{-d^2(x,y)/12s}}{(12\pi s)^{n/2}}\,d\mu(y)d\mu(x)\label{eq:error-region-II}\\
&\phantom{=}- \frac{1}{s}\int_{B_\eps(\Sigma) \cap \Omega_1^\circ}\int_{\partial E_x}\frac{d(x,y)}{2}\La (\nabla d(x,\cdot))_y, \tilde{\nu}_y \Ra\\
&\hspace{6cm}\times \frac{e^{-d^2(x,y)/12s}}{(12\pi s)^{n/2}}\,d\sigma(y) d\mu(x)\label{eq:error-region-III} \\
&=: I + II + III\,.
\end{align}
Note that we integrated by parts over $y \in W_x$, $V_x$ (as in \eqref{eq:IBP}), but did not obtain any boundary term from $\partial V_x \cap \partial W_x = \Sigma \setminus E_x^\circ$. Recall that every $\omega \in \Sigma\setminus E_x$ is on the interior of the unique $d$-geodesic $\gamma^y$ connecting $x$ to $y$ for some $y \in V_x$, and $d^2(x,\cdot)$ is $C^{1,1}$ in a neighborhood of $\gamma^y$, so the boundary terms along $\partial V_x \cap \partial W_x$ cancel (the unit normals are opposite, as in the proof of Theorem \ref{thm:multiple-copies}). $\tilde\nu$ denotes the inward-pointing unit normal to $\partial E_x$. Using now the approximation for $\Delta r^2(y) = \Delta_y d^2(x,y)$ on $V_x$ from the previous step (all quantities including $r(y)$ depending implicitly on both $x$ and $y$), we can continue with

\begin{align}\label{eq:asympt}
I &= \frac{1}{s}\int_{B_\eps(\Sigma)\cap \Omega_1^\circ} \int_{W_x\cup V_x} \left(-\frac{n}{2} + \frac{
\Delta r^2(y)}{4}\right)\frac{e^{-r^2(y)/12s}}{(12\pi s)^{n/2}}\,d\mu(y) d\mu(x) \\
&= -\frac{1}{2s}\int_{B_\eps(\Sigma)\cap \Omega_1^\circ} \int_{V_x} \frac{\rho^2}{r(y)}\left(m_{\partial\Omega_1} + m_{\partial\Omega_2} + \frac{1 - \La \nu,\nabla r\Ra^2}{\La \nu,\nabla r\Ra^2}(II_{\partial\Omega_1} + II_{\partial\Omega_2})(v_2,v_2)\right)_\omega\\
&\hspace{9cm}\times\La \nu_\omega,(\nabla r)_\omega\Ra\frac{e^{-r^2(y)/12s}}{(12\pi s)^{n/2}}\,d\mu(y) d\mu(x) \\
&\phantom{=}+ \int_{B_\eps(\Sigma)\cap \Omega_1^\circ} \left[ \int_{V_x}\frac{r(y)O(\eps)}{s\La \nu_\omega,(\nabla r)_\omega\Ra}\frac{e^{-r^2(y)/12s}}{(12\pi s)^{n/2}}\,d\mu(y) + \int_{W_x} \frac{O(r^2(y))}{s}\frac{e^{-r^2(y)/12s}}{(12\pi s)^{n/2}}\,d\mu(y)\right]d\mu(x) \,.
\end{align}

The last term contributes $\leq C\cdot \vol(B_\eps(\Sigma)) = O(\eps)$ uniformly in $s$. The second-to-last term is a bit more worrisome, but we notice that its integrand has the same form as the $(II_{\partial \Omega_1} + II_{\Omega_2})(v_2,v_2)$ term in the previous line, whose integral is in fact uniformly bounded in $s$ (we will explicitly compute it in the next step). Hence the entire last line is $O(\eps)$, and will vanish when taking $\eps \to 0$ at the end.

We claim also that $II$, $III$ of \eqref{eq:error-region-II} resp. \eqref{eq:error-region-III} vanish as $s\to 0$ with $\eps > 0$ fixed. This will be verified in the last step of the proof.

\paragraph{Computing the constants}

We have already written $\partial_s \int_{B_\eps(\Sigma)\cap \Omega_1^\circ}\a_s(x)\,d\mu(x)$, up to terms that vanish as $\eps,s \to 0$, as a sum of integrals involving the second fundamental form. It remains to compute the constants.

It is convenient now to change coordinates to a system better suited to parabolic rescaling

\begin{align}
&\Phi(\omega,\rho,\vartheta,t) := \big(\exp^{g_1}_\omega (- \rho \cdot\vartheta), \exp_{\omega}^{g_2}(t\cdot \vartheta)\big) \in \Omega_1^\circ\cup \Omega_2^\circ \,, \\
&\omega \in \Sigma\,; \, \rho,t \in (0,\eps)\,;\, \vartheta \in S_\omega X \cap \{\xi \in T_\omega X\mid \La \xi,\nu_\omega\Ra > 0\}\,.
\end{align}
Note that the variables $t = d(y,\omega)$, $\rho = d(x,\omega)$ agree with their previous usage. For any $\omega \in \Sigma$, write the corresponding domain as

\begin{align}
A_{\omega} &:=  \Big\{(\rho,t,\vartheta)\mid \gamma = \gamma^{\omega,\rho,\vartheta,t}:x\mapsto y \text{ is minimizing };\,x\in B_\eps(\Omega)\cap\Omega_1^\circ\,,\, y \in V_x\Big\}\,,\\
&\text{where} \quad\gamma^{\omega,\rho,\vartheta,t}:[-\rho, t] \to X\,, \quad \eta\mapsto \begin{cases}
\exp^{g_1}_\omega(\eta \cdot\vartheta)\,,& \eta \in [-\rho,0]\,,\\
\exp_\omega^{g_2}(\eta\cdot\vartheta)\,,& \eta \in [0,t]\,.
\end{cases}
\end{align}

To resolve a slight ambiguity of notation, we do not allow parameters causing $\gamma^{\omega,\rho,\vartheta,t}$ to hit $\Sigma$ for $\eta \in[-\rho,t]\setminus\{0\}$. $\Phi$ maps $D:=\bigcup_{\omega \in \Sigma} (\{\omega\}\times A_\omega)$ bijectively onto $\bigcup_{x \in B_\eps(\Sigma)\cap \Omega_1^\circ} (\{x\} \times V_x) = \Phi(D)$ because every pair in the latter are joined by a unique $d$-geodesic that meets $\Sigma$ once, and transversely. However, to perform any sort of calculus we would like $D$ to be open. Since $\Phi(D)$ is open and the exponential maps are continuous in all parameters (where defined), $\Phi$ maps an open neighborhood $U\supseteq D$ into $\Phi(D)$. The gap property from the first step implies that $\gamma^{\omega,\rho,\vartheta,t}$ being a $d$-geodesic is a stable property under perturbations of $\omega,\rho,\vartheta,t$ that satisfy $\Phi(\omega,\rho,\vartheta,t) =(\gamma^{\omega,\rho,\vartheta,t}_{-\rho}, \gamma^{\omega,\rho,\vartheta,t}_{t}) \in \Phi(D)$, so we can assume it holds for all $(\omega,\rho,\vartheta,t)\in U$. Thus $U \subseteq D \subseteq U$ is open, as desired.

We define the parabolically rescaled domain

\begin{align}
s^{-1/2}A_\omega &:= \{(\rho,t,\vartheta)\mid (s^{1/2}\rho,s^{1/2}t,\vartheta) \in A_\omega\} \overset{s\to 0}{\longrightarrow} (0,\infty)^2 \times \big( S_{\omega}X \cap\{\xi\mid \La \xi,\nu_\omega\Ra > 0\}\big)\,.\label{eq:domain-convergence}
\end{align}

The above convergence is meant in the sense of limits of open subsets (pointwise convergence of indicator functions). To show that it holds, we can observe that

\begin{equation}\label{eq:bad-inclusion}
s^{-1/2}A_{\omega} \supseteq (c^{-1}f(Rs^{1/2}) R ,R/2) \times (0,(R/2)) \times (S_\omega X \cap \{\xi\mid \La \xi,\nu_\omega\Ra \geq \delta\})\,,
\end{equation}

whenever $0 < R \leq s^{-1/2}\min(R_0(\delta,g_1,g_2), \eps\delta)$. Here $f$ is the function from \eqref{eq:error-volume-decay}, while $R_0 := \min_{\xi}\max\{r > 0 \mid r\mapsto d(\exp^{g_1}_\omega(-r\xi), \Sigma) \text{ is nondecreasing}\}$ and $c = c(\delta,g_1,g_2):= \min_{\xi}\min_{r \in[0,R_0]} d(\exp^{g_1}_\omega(-r\xi),\Sigma)/r$, with both $\min$s taken over unit vectors $\xi$ with $\La \xi,\nu_\omega\Ra \geq \delta$. Indeed, \eqref{eq:error-volume-decay} and the constraint on $R$ implies that for $(s^{-1/2}\rho,s^{-1/2}t,\vartheta)$ in the RHS of \eqref{eq:bad-inclusion}, we have $x:= \gamma^{\omega,\rho,\vartheta,t}_{-\rho} \in B_\eps(\Sigma)\cap \Omega_1^\circ$ and $\gamma^{\omega,\rho,\vartheta,t} \subseteq X\setminus E_x$ (and in particular $\gamma^{\omega,\rho,\vartheta,t}_t \in V_x$). Moreover, $\gamma^{\omega,\rho,\vartheta,t}$ cannot first fail to be minimizing at $\eta \in [-\rho,0)$ because $\gamma^{\omega,\rho,\vartheta,t}_\eta \in W_x$ is connected to $x$ by a $d$-geodesic lying entirely in $\Omega_1^\circ$ by Claim \ref{claim:geo}, which if distinct would violate $\rho \leq \eps \leq \mathrm{inj}(g_1)$. It also cannot first fail to be minimizing at $\eta \in [0,t]$ in view of the gap property for $V_x$ in Claim \ref{claim:geo}, along with the second constraint on $R$. Thus $\gamma^{\omega,\rho,\vartheta,t}$ is minimizing, and the inclusion \eqref{eq:bad-inclusion} holds. Now take $s\to 0$, then $R \to \infty$ and $\delta \to 0$ to obtain the desired convergence. 

We now consider the Jacobian determinant of $\Phi$, defining its parabolic normalization

\begin{align}\label{eq:prelimit-Jacobian}
&|D\Phi|_{(\omega,s^{1/2}\rho, \vartheta, s^{1/2}t)} \\
&= s^{\frac{n - 1}{2}}\left|\begin{array}{c|c|c|c}
\partial_\omega \exp^{g_1}_\omega(-s^{1/2}\rho\cdot\vartheta) & -(D\exp_\omega^{g_1})_{-s^{1/2}\rho\cdot\vartheta}(\vartheta) & -\rho(D\exp_\omega^{g_1})_{-s^{1/2}\rho\cdot\vartheta}\circ P_{\La \vartheta\Ra^\perp} & 0 \\
\hline
\partial_\omega \exp^{g_2}_\omega(s^{1/2}t\cdot\vartheta) & 0 & t(D\exp_\omega^{g_2})_{s^{1/2}t\cdot\vartheta}\circ P_{\La \vartheta\Ra^\perp} & (D\exp_\omega^{g_2})_{s^{1/2}t\cdot\vartheta}(\vartheta)
\end{array}\right|\\
&=: s^{\frac{n - 1}{2}}|L_s|\,.
\end{align}

Computing $|L_s|$ directly is quite messy. Thus consider a smooth local coordinate system on $X$ near $\omega$ and the induced coordinate system on $\left.SX\right|_{\Sigma}$, requiring $g_{ij}(\omega) = \delta_{ij}$, and that $\nu_\omega$ is identified with $e_n(\omega)$. We then recover the Euclidean case $\omega \in \Sigma\subseteq X,\, \vartheta \in S_\omega X\rightsquigarrow 0^n \in \{0\} \times \R^{n - 1} \subseteq \R^n,\vartheta \in S^{n-1}_1$ in the limit $s\to 0$:

\begin{equation}
L_0 :=\lim_{s\to 0}L_s = \left[\begin{array}{cc|c|ccc|c}
& &-\vartheta^1& &-\rho \cdot(P_{ \vartheta^\perp})^1&&0 \\
&\Id_{n - 1}&\vdots&&\vdots&&\vdots\\
& &-\vartheta^{n - 1}& &-\rho \cdot(P_{ \vartheta^\perp})^{n - 1}&&0\\
\hline  
&0&-\vartheta^n&&-\rho\cdot (P_{\vartheta^\perp})^n&&0 \\
\hline
& &0& &t \cdot (P_{\vartheta^\perp})^{1}&&\vartheta^1 \\
&\Id_{n - 1}&\vdots& & \vdots&& \vdots\\
& &0&&t \cdot (P_{\vartheta^\perp})^{n - 1}&&\vartheta^{n - 1}\\
\hline  
&0&0&&t\cdot (P_{\vartheta^\perp})^n&&\vartheta^n
\end{array}\right]\,.
\end{equation}

By applying elementary row and column operations, we can reduce the Jacobian determinant to that of a block upper-triangular matrix

\begin{align}
    |L_0|=\lim_{s\to 0}|L_s| &= \left|\begin{array}{cc|c|ccc|c}
& &-\vartheta^1& &-\rho \cdot(P_{ \vartheta^\perp})^1&&0 \\
&\Id_{n - 1}&\vdots&&\vdots&&\vdots\\
& &-\vartheta^{n - 1}& &-\rho \cdot(P_{ \vartheta^\perp})^{n - 1}&&0\\
\hline  
&0&-\vartheta^n&&-\rho\cdot (P_{\vartheta^\perp})^n&&0 \\
\hline
& &0& &(\rho + t) \cdot (P_{\vartheta^\perp})^{1}&&\vartheta^1 \\
&\Large{0}&\vdots& & \vdots&& \vdots\\
& &0&&(\rho + t) \cdot (P_{\vartheta^\perp})^{n - 1}&&\vartheta^{n - 1}\\
\hline  
&0&0&&(\rho + t)\cdot (P_{\vartheta^\perp})^n&&\vartheta^n
\end{array}\right| = |\vartheta^n| (\rho + t)^{n - 1}\,.
\end{align}

Writing $\cos \theta := \La (\nabla r)_\omega ,\nu_\omega \Ra = |\La \vartheta, \nu_\omega\Ra| = |\vartheta^n|$, we are finally in a position to compute the terms from the previous step. In particular, applying the change of coordinates $\Phi$ we can rewrite part of the integral on the second line of \eqref{eq:asympt} and parabolically rescale $(\rho, t)\leadsto (s^{1/2}\rho,s^{1/2}t)$ to obtain

\begin{align}
&\phantom{=}\frac{1}{2s}\int_{\{(x,y)\mid B_\eps(\Sigma)\cap \Omega_1^\circ\,,\,y\in V_x\}} \frac{\rho^2}{r(y)}(m_{\partial\Omega_1} + m_{\partial\Omega_2})_\omega\La \nu_\omega,(\nabla r)_\omega\Ra\frac{ e^{-r^2(y)/12s}}{(12\pi s)^{n/2}}\,d\mu(y)d\mu(x) \\
&= \frac{1}{2s}\int_{\Sigma}\int_{A_\omega}\frac{\rho^2}{(t + \rho)}(m_{\partial\Omega_1} + m_{\partial\Omega_2})_\omega\cos\theta\frac{e^{-(t + \rho)^2/12s} }{(12\pi s)^{n/2}}\, |D\Phi|_{(\omega,\rho,\vartheta,t)}dt d\rho d\vartheta d\omega \\
&= \frac{1}{(12\pi)^{n/2}}\int_{\Sigma} \int_{(0,\infty)^2 \times (S_\omega X \cap \{\xi\mid \La \xi, \nu_\omega\Ra > 0\})} \frac{(m_{\partial\Omega_1} + m_{\partial\Omega_2})_\omega}{2}\rho^2(t + \rho)^{n - 2}(\cos\theta)^2e^{-(t + \rho)^2/12}\\
&\hspace{9cm}\times \ind_{s^{-1/2}A_\omega}\cdot \left(\frac{|L_s|}{|\vartheta^n|(\rho + t)^{n-1}}\right)\,dt d\rho d\vartheta d\omega\,.
\end{align}

Dominated convergence can be applied to the above integral as $s\to 0$, which amounts to replacing $\ind_{s^{-1/2}A_\omega}$ and the normalized Jacobian factor on the RHS with $1$. (To be careful with the unbounded denominators, one should note that the eigenvalues of $L_0^{-1}$ are $1,\, (\rho + t)^{-1},\,-\La (\nabla r)_\omega ,\nu_\omega \Ra^{-1}$, while $s^{1/2}(\rho + t) \leq \eps\,,(\rho + t)\,, \eps \La (\nabla r)_\omega ,\nu_\omega \Ra$ on $s^{-1/2}A_\omega$ by \eqref{eq:small-inner-prod-small-vol}, then write $|L_s| = |L_0 + O(s^{1/2}(\rho + t))| = |L_0||\Id_n + O(\|L_0^{-1}\|s^{1/2}(\rho + t))| = |L_0|O(1)$.) The variables are now separated, and all that remains is calculus

\begin{align}
&\phantom{=}\lim_{s\to 0}\frac{1}{2s}\int_{\{(x,y)\mid B_\eps(\Sigma)\cap \Omega_1^\circ\,,\,y\in V_x\}} \frac{\rho^2}{r(y)}(m_{\partial\Omega_1} + m_{\partial\Omega_2})_\omega\La \nu_\omega,(\nabla r)_\omega\Ra\frac{ e^{-r^2(y)/12s}}{(12\pi s)^{n/2}}\,d\mu(y)d\mu(x) \\
&= \frac{1}{(12\pi)^{n/2}}\left(\vol(S^{n-2}_1)\int_0^{\pi/2}\sin^{n - 2}(\theta)\cos^2(\theta)\,d\theta\right) \left(\int_0^{\infty}\rho^2\int_0^\infty (t + \rho)^{n - 2}e^{-(t + \rho)^2/12}\,dtd\rho\right) \\
&\hspace{11cm}\times\int_{\Sigma}\frac{(m_{\partial\Omega_1} + m_{\partial\Omega_2})_\omega}{2}\,d\omega \\
&= 2\pi\frac{\vol(S_1^{n - 2})}{\vol(S_1^{n + 1})}\left(\frac{\vol(S^{n - 1}_1)}{\vol(S^{n - 2}_1)} - \frac{\vol(S^{n + 1})}{\vol(S^{n}_1)}\right)\int_{\Sigma}\frac{(m_{\partial\Omega_1} + m_{\partial\Omega_2})_\omega}{2}\,d\omega\,,
\end{align}

where we used

\begin{align}
\int_0^{\infty}\rho^2\int_0^\infty (t + \rho)^{n - 2}e^{-(t + \rho)^2/12}\,dtd\rho &= \int_0^{\infty}\rho^2\int_\rho^\infty t^{n - 2}e^{-t^2/12}\,dtd\rho \\
&= \int_0^\infty t^{n - 2}e^{-t^2/12}\int_0^t \rho^2\,d\rho dt \\
&= \frac{1}{3}\int_0^\infty t^{n + 1} e^{-t^2/12}\,dt = \frac{4\pi(12\pi)^{n/2}}{\vol(S^{n + 1}_1)}\,,\\
\int_0^{\pi/2}\sin^{n - 2}(\theta)\cos^2(\theta)\,d\theta &= \int_0^{\pi/2} (\sin^{n -2}(\theta)-\sin^n(\theta))\,d\theta \\
&= \frac{\vol(S^{n - 1}_1)}{2\vol(S^{n - 2}_1)} - \frac{\vol(S^{n + 1}_1)}{2\vol(S^{n}_1)}\,.
\end{align}

Similarly, we now compute the second term from the previous step in \eqref{eq:asympt}. The vector $v_2$ and unpleasant-looking factor appearing there can be made sense of here as

\begin{align}
S_\omega X \cap \{\xi\mid \La \xi, \nu_\omega\Ra > 0\} \ni\vartheta &= (\cos\theta) \nu_\omega + (\sin\theta)v_2\,, \qquad v_2 \in S_\omega \Sigma \,; \\
\frac{1-\La \nu_\omega,(\nabla r)_\omega\Ra^2}{\La \nu_\omega,(\nabla r)_\omega\Ra^2}&= \tan^2(\theta)\,.
\end{align}

Applying now the same manipulations as for the first term, we obtain

\begin{align}
&\lim_{s \to 0}\frac{1}{2s}\int_{\{(x,y)\mid B_\eps(\Sigma)\cap \Omega_1^\circ\,,\,y\in V_x\}}\frac{\rho^2}{r(y)} \frac{1 - \La \nu_\omega,(\nabla r)_\omega\Ra^2}{\La \nu_\omega,(\nabla r)_\omega\Ra^2}(II_{\partial\Omega_1} + II_{\partial\Omega_2})_\omega(v_2,v_2)\La \nu_\omega,(\nabla r)_\omega\Ra\\
&\hspace{11cm}\times\frac{e^{-r^2(y)/12s}}{(12\pi s)^{n/2}}\,d\mu(y) d\mu(x) \\
&= \frac{1}{(12\pi)^{n/2}}\left(\int_0^{\pi/2}\sin^{n}(\theta)\,d\theta\right) \left(\int_0^{\infty}\rho^2\int_0^\infty (t + \rho)^{n - 2}e^{-(t + \rho)^2/12}\,dtd\rho\right)\\
&\hspace{10cm} \times\int_{\Sigma}\int_{S_\omega \Sigma}\frac{(II_{\partial\Omega_1} + II_{\partial\Omega_2})_\omega(\sigma,\sigma)}{2}\,d\sigma d\omega \\
&= \frac{2\pi}{\vol(S_1^n)}\frac{\vol(S^{n - 2}_1)}{n - 1}\int_{\Sigma}\frac{(m_{\partial\Omega_1} + m_{\partial\Omega_2})_\omega}{2}\,d\omega\,.
\end{align}

We finally conclude, using $2\pi\vol(S^{n - 2}_1) = (n - 1) \vol(S^{n}_1)$

\begin{align}
\lim_{\eps \to 0}\lim_{s \to 0}\partial_s \int_{B_\eps(\Sigma)\cap \Omega_1^\circ} \a_s(x)\,d\mu(x) &= 2\pi\left(\frac{n - 2}{n - 1}\frac{\vol(S^{n - 2}_1)}{\vol(S_1^{n})} - \frac{\vol(S^{n - 1}_1)}{\vol(S_1^{n + 1})}\right)\int_{\Sigma}\frac{(m_{\partial\Omega_1} + m_{\partial\Omega_2})_{\omega}}{2}\,d\omega \\
&= -\int_{\Sigma}(m_{\partial\Omega_1} + m_{\partial\Omega_2})_\omega\,d\omega\,.
\end{align}

Using that all previous reasoning was symmetric under exchanging $\Omega_1$ and $\Omega_2$, along with the computations in the smooth setting from Theorem \ref{thm:smoothcase}, the integral over the whole space is

\begin{align}
&\lim_{s \to 0}\partial_s \int_{X} \a_s(x)\,d\mu(x) \\
&= \lim_{\eps \to 0}\lim_{s \to 0}\partial_s \int_{B_\eps(\Sigma
) \sqcup (\Omega_1^\circ \setminus B_\eps(\Sigma
))\sqcup (\Omega_2^\circ \setminus B_\eps(\Sigma
))} \a_s(x)\,d\mu(x)\\
&= -2\int_{\Sigma}(m_{\partial\Omega_1} + m_{\partial\Omega_2})_\omega\,d\omega -\lim_{\eps\to 0}\int_{\Omega_1^\circ \setminus B_\eps(\Sigma
)} R_{g_1}\,d\mu - \lim_{\eps\to 0}\int_{\Omega_2^\circ \setminus B_\eps(\Sigma
)}R_{g_2}\,d\mu\\
&= -2\int_{\Sigma}(m_{\partial\Omega_1} + m_{\partial\Omega_2})_\omega\,d\omega - \int_{X\setminus \Sigma} R_g\,d\mu\,.
\end{align}

\paragraph{Isolating bad points}
The following technical lemma is meant to section off the ``bad'' points at one end of a $d_g$-geodesic that may branch, be non-unique, or pass through the boundary multiple times.

\begin{lemma}\label{lem:technical-lemma}
    Pick $0 < \lambda \leq \lambda_0(n)$ and suppose that $g$ is a Lipschitz Riemannian metric on $B_{20}^{g_E}(0^n)\subseteq \R^n$, that is smooth on $B_{20}^{g_E}(0^n)\setminus (\{0\} \times \R^{n -1})$ up to the boundary approached from either side, such that $e_1$ is a unit normal vector to $\{0\} \times \R^{n - 1}$ w.r.t $g$, and such that 
    \begin{equation}\label{eq:metric-reg-bounds}
    \frac{1}{4}g_{E} \leq g \leq 4g_{E}\,; \quad|\partial_k g_{ij}(z)| \leq \lambda\,,\quad |\partial_{k\ell}g_{ij}(z)| \leq \lambda^2\,,\quad \forall z \in B_{20}^{g_E}(0^n)\setminus (\{0\} \times \R^{n})\,,
    \end{equation}

    with $g_E$ the Euclidean metric on $\R^{n}$. Fix any $x\in B_4^{g_E}(0^n) \cap (\R_{< 0} \times \R^{n - 1})$ and define

    \begin{align}
      F_x&:=
    \begin{cases}
    \emptyset& \text{if}\quad |x^1| \geq 5L\lambda\,,\\
    T_{\lambda L}(x)& \text{else}\,,
    \end{cases} \\
    &\text{with}\quad T_{\alpha}(x):= \{y : |y^1 - x^1|/\|y - x\| \leq \alpha\}\,,\quad \forall \alpha \geq 0\,.
    \end{align}
    
    Then there is a constant $L_0(n)$ such that $L \geq L_0(n)$ implies the following statements: 
    \begin{itemize}
    \item For every $y \in B_1^{g_E}(x)\setminus F_x$, there is a unique unit-speed $d_g$-geodesic $\gamma$ joining $x$ to $y$, and it crosses $\{0\} \times \R^{n -1}$ at most once and transversely in the effective sense $\La\dot\gamma_t, e_1\Ra_g = \dot\gamma_t^1 \geq L\lambda/16$ if $\gamma_t^1 = 0$. In fact, there is no other curve $\gamma$ that is: $C^1$, unit-speed, of $g$-length $\ell\leq 4$, connecting $x$ to $y$, smooth and solves the $g$-geodesic equations on $\gamma \cap (\R_{\lessgtr} \times \R^{n - 1})$, and if there is $t\in[0,\ell]$ such that $\gamma^1_t=0$ then such a $t$ is unique and $\La\dot\gamma_t, e_1\Ra_g \geq L\lambda/128$. 
    
    \item $d_g^2(x,\cdot)$ is smooth in a neighborhood of $\gamma_{[0,t]}\subseteq \R_{\leq 0}\times \R^{n - 1}$, $\gamma_{[t,\ell]} \subseteq \R_{\geq 0} \times \R^{n - 1}$ and is $C^{1,1}$ in a neighborhood of $\gamma\subseteq \R^n$. Also, $d_g^2(x,\cdot)$ is smooth on $B_1^{g_E}(x)\setminus \big(F_x\cup (\{0\} \times \R^{n - 1})\big)$ up to the boundary approached from either side. 

    \item For any $y \in F_x \cap B_1^{g_E}(x)$ and $\gamma$ a $g$-length $\ell$ $d_g$-geodesic connecting $x$ to $y$, we have $|\dot\gamma^1_\ell| \leq 4L\lambda$.
    \end{itemize}
\end{lemma}

\begin{proof}

Consider a point $x \in B_4^{g_{E}}(0^n) \cap (\R_{< 0} \times \R^{n - 1})$. Let $\gamma: [0,\ell] \to B_{8}^{g_E}(0^n)$ with $\gamma_0 = x$ be a $g$-length $\ell \leq 2$ unit-speed $d_g$-geodesic, which recall is $C^1$. Let $\alpha: [0,l] \to B_{20}^{g_E}(0^n)$ with $\alpha_0 = x$ be a $C^1$ $g$-length $l \leq 8$ unit-speed curve, that is smooth and solves the $g$-geodesic equations on $\alpha \cap (\R_{\lessgtr}\times \R^{n - 1})$ (for example, but not necessarily, $\alpha = \gamma$).

\medskip

1) We first show that there is a $C_1 = C_1(n)$ independent of $\lambda$ such that for any choice $L_1 \geq 2C_1$, if $|\dot\alpha^1_0| \geq L_1\lambda$ then $\alpha$ intersects $\{0\} \times \R^{n - 1}$ at most once, and any such intersection $\alpha^1_{t_0} = 0$ is transverse with $\dot\alpha^1_{t_0} \geq (L_1 - C_1)\lambda$. As a partial converse, if $\alpha_{t_0}^1 = 0$ and $\dot\alpha^1_{t_0} \geq L_1\lambda \geq 2C_1\lambda$ for some $t_0 \in [0,l]$, then $\dot\alpha^1_0 \geq (L_1 - C_1)\lambda$.

\medskip

Let $t_0 = \inf\{{s \in [0,\ell]}\mid \alpha^1_s = 0\}$; if no such $t_0$ exists we are done. We have the following crude estimate on $\ddot\alpha_t^1$, which we will use repeatedly (using the first two items of \eqref{eq:metric-reg-bounds}): 

\begin{equation}\label{eq:crude-est}
\ddot\alpha^1_t= -\dot\alpha^i_t\dot\alpha^j_t\Gamma_{ij}^1(\alpha_t) = O(\lambda) \implies \dot{\alpha}_t^1 \geq \dot\alpha_0^1 - C(n)\lambda t \geq \dot\alpha_0^1 - 8C\lambda\,,\quad\forall t\in [0,t_0]\subseteq [0,l]\,.
\end{equation}

Thus, if we require $\dot\alpha_0^1 \geq L_1\lambda \geq 16C\lambda$, then we see that $\alpha_{t_0} \in \{0\} \times \R^{n - 1}$ is a transverse intersection with $\dot\alpha_{t_0}^1 \geq (L_1 - 8C)\lambda \geq 8C\lambda$. If $t_1 > t_0$ is the next time that $\alpha$ intersects $\{0\} \times \R^{n - 1}$, then $\dot\alpha_t^1 = 0$ for some $t_0 < t \leq t_1$ by the Mean Value Theorem, while

\begin{equation}
0 =\dot\alpha_t^1 \geq \dot\alpha_{t_0}^1 - C\lambda(t - t_0) \geq 8C\lambda - C\lambda(t - t_0)\,,
\end{equation}

so we conclude that $t \geq t_0 + 8 > 8$. But $\alpha_t$ had $g$-length $l \leq 8$, so we conclude that there is no second intersection point. 

A similar argument, this time with the roles of $t_0$ and $t_1$ in the argument above played by $0$ resp. the first time of intersection with $\{0\} \times \R^{n - 1}$, shows that $\dot{\alpha}^1_0 \leq -L_1\lambda \leq -16C\lambda \implies$ $\alpha \subseteq (-\infty, x^1] \times \R^{n - 1}$ for $t \in [0,l]$. Again a similar argument, this time with the roles of $t_0$ and $t_1$ in the argument above played by the intersection time $t_0$ with $\dot\alpha^1_{t_0} \geq L_1\lambda \geq 16 C\lambda$ resp. a potential previous time of intersection, then the roles of $0$ and $t_0$ played by $t_0$ resp. $0$, shows the partial converse. We conclude with $C_1 := 16C$.

\medskip 

2) Suppose that $\beta,\eta$ are curves satisfying the constraints for $\alpha$, that have only transverse intersections with $\{0\} \times \R^{n - 1}$, and that coincide at some point $\beta_t = \eta_s$. We show that if $\lambda \leq \lambda_0(n)$ is small enough, then $\La \dot\beta_t, \dot \eta_s\Ra_{g} > 0$ and there is $C_2 = C_2(n)> 0$ such that $\left||\dot\beta_0^1| - |\dot\eta_0^1|\right|\leq C_2\lambda$.

\medskip

Write the $g$-unit vectors $v:= \dot\beta_0\,,\,w:= \dot\eta_0 \in \R^n$. Since $\beta$ is $C^1$ on a compact interval and has no tangential intersections with $\{0\} \times \R^{n - 1}$, it can have at most finitely many (transverse) intersections. Reasoning as in  \eqref{eq:crude-est} on each interval between intersections, on each component of $\beta$, we have $\dot\beta_t = v + O(\lambda)$ (the implicit constant depending only on $n$). Integrating, we have $(\beta_t - x) = t(v + O(\lambda))$.
This reasoning also holds for $\eta$, and taking $g_x$ norms of the assumed intersection $\beta_t = \eta_s$ (identifying manifold and tangent space via $T_xB_{20}^{g_E}(0^n) \simeq \R^n$), gives

\begin{equation}
|t - s| = ||tv|_{g_x} - |sv|_{g_x}| \leq |tv - sw|_{g_x} = |(\beta_t - x) - (\eta_s - x) + (t + s)O(\lambda)|_{g_x} = (s + t)O(\lambda)\,.
\end{equation}

With this comparability, we have

\begin{equation}
|v - w|_{g_x} = \frac{|2(tv - sw) + (s - t)(v + w)|_{g_x}}{s + t} \leq \frac{2|tv - sw|_{g_x} + |s - t||v + w|_{g_x}}{s + t}=O(\lambda)\,.
\end{equation}

The first claim now follows from Cauchy-Schwarz, the second item of \eqref{eq:metric-reg-bounds}, and again $\lambda \leq \lambda_0(n)$

\begin{align}
 \La v,w\Ra_{g_x} &= |w|^2_{g_x} + \La v - w,w \Ra_{g_x} \geq 1 - O(\lambda) > 1/2\,, \\
\La \dot \beta_t,\dot\eta_s\Ra_{g_{\beta_t}} &= \La v + O(\lambda), w + O(\lambda)\Ra_{g_{\beta_t}} = \La v, w \Ra_{g_{x}} + O(\lambda) > 0\,,
\end{align}

while for the second claim, $\left||v^1| - |w^1|\right| \leq |v^1 - w^1| \leq  \|v - w\| = O(\lambda)$.
\medskip

3) Let us now show that there is $C_3 = C_3(n)> 0$ independent of $\lambda$ such that for any choice $L_2 > 0$, if there is some $t_0 \in (0,\ell]$ with $\gamma_{t_0} \in T_{L_2\lambda}(x)$, then $\gamma \subseteq T_{(4L_2 + C_3)\lambda}(x)$.

\medskip

Assume there is such a $t_0$. Using the first item of \eqref{eq:metric-reg-bounds} (i.e. that $d_g$ and $d_{g_E}$ are $2$-biLipschitz), along with the fact that $\gamma$ is a $g$-arclength-parameterized $d_g$-geodesic, we have $t \geq \|\gamma_t - \gamma_0\|/2$ for any $t \in [0,\ell]$. Since $\gamma_{t_0} \in T_{L_2 \lambda}(x)$, we therefore have

\begin{equation}
\frac{|\gamma_{t_0}^1 - \gamma_0^1|}{t_0} = \frac{|\gamma_{t_0}^1 - x^1|}{t_0} \leq 2\frac{|\gamma_{t_0}^1 - x^1|}{\|\gamma_{t_0} - x\|} \leq 2L_2\lambda\,.
\end{equation}

Now suppose all intersections with $\{0\} \times \R^{n - 1}$ are transverse, i.e. $\gamma^1_t = 0 \implies \dot\gamma_t^1 \neq 0$ for $t \in [0,\ell]$. By the Mean Value Theorem applied to the above inequality, there is $t_1 \in [0,t_0]$ such that $|\dot\gamma^1_{t_1}| \leq 2L_2\lambda$. Since $\gamma$ is $C^1$ on a compact interval, absence of tangential intersections implies finitely many intersections. Arguing as in 1), and more specifically \eqref{eq:crude-est}, on each interval between intersections, we obtain that $|\dot\gamma_t^1| \leq |\dot\gamma_{t_1}^1| + |\dot\gamma_t^1 -\dot \gamma_{t_1}^1| \leq (2L_2 + 2C)\lambda$ for all $t\in[0,\ell]$. 

If not all intersections are transverse, we can write $[0,\ell] = A\cup[0,b_0) \cup (a_0,\ell] \cup\bigcup_j (a_j,b_j)$ where $t\in A$ are the times of tangential intersections $\gamma^1_t =\dot\gamma^1_t = 0$, while each $[0,b_0),(a_0,\ell], (a_j,b_j)$ contains only transverse intersections. Then simply repeat the above argument on compact exhaustions $[a_j', b_j'] \subseteq (a_i,b_j)$, $[0, b_0'] \subseteq [0,b_0)$, and $[a_0', \ell] \subseteq (a_0,\ell]$ (where $|\dot\gamma^1_{a_j'}|$ or $|\dot\gamma^1_{b_j'}|$ are arbitrarily small as the open intervals are exhausted) to obtain $|\dot\gamma_t^1| \leq 2C\lambda$ on each of these intervals. Incidentally, we have proved the final claim of the Lemma for any choice $L_2 \geq C$.

Finally, we conclude

\begin{equation}
\frac{|\gamma_t^1 - x^1|}{\| \gamma_t-x\|} \leq \frac{2}{t}\int_{0}^t|\dot\gamma^1_s|\,ds \leq 2(2L_2 + 2C)\lambda
\end{equation}

so the result is proved if we let $C_3 := 4C$.

\medskip

4) We finally use 1), 2), and 3) to prove the Lemma, for now requiring $L \geq 384C_1 + 384 C_2 + 8C_3$. 

\medskip

Suppose first that $y \in B_1^{g_E}(x) \setminus T_{L\lambda}(x)$. Since $d_g$ is $2$-biLipschitz to Euclidean distance, any $d_g$-geodesic $\gamma$ connecting $x$ and $y$ has $g$-length $\ell \leq 2$ and lie entirely in $B_8^{g_E}(0^n)$. By 3) applied with $(L - C_3)/4\to L_2$ and the assumption on $y$, no such curve can pass through $T_{(L - C_3)\lambda/4}(x)$. In particular, we have

\begin{equation}
|\dot\gamma_0^1| = \lim_{t\searrow 0} \frac{|\gamma_t^1 - \gamma_0^1|}{t} \geq \frac{1}{2}\liminf_{t\searrow 0} \frac{|\gamma_t^1 - x^1|}{\|\gamma_t - x\|} \geq \frac{L - C_3}{8}\lambda\,.
\end{equation}

This is the assumption of 1) with $(L - C_3)/8\to L_1$, and we conclude that $\gamma$ intersects $\{0\} \times \R^{n - 1}$ at most once, and transversely with $\dot\gamma_t^1 \geq [(L - C_3)/8 - C_1]\lambda \geq (L/16)\lambda$ if $\gamma^1_t = 0$. 

\medskip

If additionally $|x^1| \geq 5L\lambda > (4L + C_3)\lambda$, then $T_{ (4L + C_3)\lambda}(x) \cap (\{0\} \times \R^{n - 1}) \cap B_1^{g_E}(x) = \emptyset$. By applying 3) with $L \to L_2$, we see that if $\gamma$ intersects $\{0\} \times \R^{n - 1}$, then it cannot intersect $T_{L\lambda}(x)$. Thus in this case, all $d_g$-geodesics beginning at $x$ and ending in $T_{L\lambda}(x) \cap B_1^{g_E}(x)$ are entirely contained in $\R_{< 0} \times \R^{n - 1}$.

\medskip

Finally, we address smoothness of $d_g^2(x,\cdot)$ and the uniqueness statement.

For any $\alpha$ a $C^1$ $g$-length $l \leq 8$ curve with $\alpha_0 = x$, we have by \eqref{eq:metric-reg-bounds} that $\alpha$ lies entirely in $ B_{20}^{g_E}(0^n)$. Assume additionally that $\alpha$ is unit-speed, smooth and solves the $g$-geodesic equations on $\alpha \cap (\R_{\lessgtr} \times \R^{n - 1})$, and with $|\dot \alpha_0^1| \geq ((L/128) - C_1 - 2C_2)\lambda$. We can then apply 1) with $((L/128) - C_1 - 2C_2) \to L_1$, obtaining that $\alpha$ intersects $\{0\} \times \R^{n - 1}$ at most once and with $\dot\alpha_{t_0}^1 \geq ((L/128) - 2C_1 - 2C_2)\lambda \geq L \lambda/384$ at the time $t_0$ of intersection. This means that $\alpha$ is uniquely determined by $\alpha_0 = x$ and $\dot\alpha_0$, giving a well-defined exponential map

\begin{align}
\exp^g_x:D_x\to B_4^g(x)\,,&\quad v\mapsto \begin{cases}
\exp_x^{g_-}(v)& |v|_{g_x} \leq t_0\,;\\
\exp_{\exp_x^{g_-}(t_0v/|v|_{g_x})}^{g_+}((|v|_{g_x} - t_0)(D\exp_x^{g_-})_{t_0v/|v|_{g_x}}(\partial_r))& t_0 <|v|_{g_x} \leq 4
\end{cases}\,,\\
\text{where we define } D_x &:= B_4^{g_x}(0) \cap \{v\in\R^{n}\mid |v^1|/|v|_{g_x} \geq ((L/128) - C_1 - 2C_2)\lambda\}\,,\\
\text{and for later use } \widetilde D_x &:= B_4^{g_x}(0) \cap \{v\in\R^{n}\mid |v^1|/|v|_{g_x} \geq (L/128 - C_1)\lambda\}
\,.
\end{align}

In this definition, $t_0 = t_0(v/|v|_{g_x}) > 0$ is the smallest $t$ with $\exp_x^{g_-}(tv/|v|_{g_x}) \in \{0\} \times \R^{n - 1}$ if it exists, else $4$. The notation $g_{\pm}$ indicates the choice of one-sided limit on $\{0\} \times \R^{n - 1}$. By definition, $\exp_x^g$ is smooth on $\{v\mid |v|_{g_x} \leq t_0\}$. Then, $S:= (\exp_x^{g_-})^{-1}(\{0\} \times \R^{n - 1})$ is a smooth separating hypersurface $S\subseteq D_x$ transverse to $\partial_r$ and $V_y:= (D\exp_x^{g_-})_{(\exp_x^{g_-})^{-1}(y)}(\partial_r)$ is a smooth vector field on $y\in\{0\} \times \R^{n - 1}$ pointing into $\R_{> 0}\times \R^{n - 1}$; to invert $\exp_x^{g_-}$ smoothly we are using that the injectivity radius of $(\R_{\leq 0}\times \R^{n - 1},g)$ is bounded below by $8 \geq 4$, which requires $\lambda \leq \lambda_0(n)$. The map $\Psi: \{(r,\omega)\in [0,\infty)\times S\mid r \leq 4 - |\omega|_{g_x}\}\to D_x \setminus \{v\mid |v|_{g_x} < t_0\}\,,\, (r,\omega)\mapsto (1 + r/|\omega|_{g_x})\omega$ is a smooth diffeomorphism, and $\exp_x^g \circ \Psi: (r,\omega)\mapsto \exp_{\exp_x^{g_-}(\omega)}^{g_+}(rV_{\exp_x^{g_-}(\omega)})$ in these coordinates is manifestly smooth. Finally, $\exp_x^g$ is $C^{1,1}$ across $S$, which follows because first partial derivatives in the $\partial_r$ direction and directions tangent to $S$ are smooth up to $S$ when approached from either side, and agree on $S$.

We now argue that $\exp_x^g$ is non-singular on $D_x$. Rescaling $g\rightsquigarrow g_\lambda := g(\lambda^{-1} \cdot)$, the control  \eqref{eq:metric-reg-bounds} holds for $g_\lambda$ but with $1\to\lambda$. The conclusions obtained so far are scaling-invariant, so hold on $B_{\lambda}^{g_E}(\lambda x) \subseteq B_{2\lambda}^{g_\lambda}(\lambda x)$. If $\beta_t := \exp_{\lambda x}^{g_\lambda}(tv)$ with $v \in D_{\lambda x}:= \lambda D_x$ intersects $\{0\} \times \R^{n - 1}$ at time $t_0$, then by definition of this domain we have $\La e_1, \dot\beta_{t_0}\Ra_{g_\lambda} = \dot\beta_{t_0}^1 \geq L\lambda/384$ (recall $e_1$ is a unit normal vector of $\{0\} \times \R^{n - 1}$ w.r.t $g_\lambda$) and $\beta$ has $g_\lambda$-length $l\leq 4\lambda \leq 1536L^{-1} \La e_1, \dot\beta_{t_0}\Ra_{g_\lambda}$. If $L \geq L_0(n)$ is large enough, then \eqref{eq:small-inner-prod-small-vol} is satisfied and Remark \ref{rmk:no-conj-points} applies. Note that to apply the Remark here we locally define $r := |(\left.\exp_{\lambda x}^{g_\lambda}\right|_{U})^{-1}(\cdot)|_{(g_\lambda)_{\lambda x}}$ up to a potential singular time (identifying $\partial_r = \nabla r$ uses the Gauss Lemma, which holds even over the boundary because $\La \partial_r, J_i'(\rho^+) - J_i'(\rho^-)\Ra = 0$, see \eqref{eq:boundary-error-term}), $J_1 = (D\exp_{\lambda x}^{g_\lambda})_{v}(\partial_r)$, and $J_i := (D\exp_{\lambda x}^{g_\lambda})_{v}\big((|v|_{(g_\lambda)_{\lambda x}}/t_0)w_i\big)$ for $i = 2,\ldots, n$, where $w_i := (D\exp_{\lambda x}^{(g_\lambda)_-})_{\beta_{t_0}}^{-1}(v_i)$ and $v_2,\ldots, v_n$ is a $g$-ONB for $\{0\}\times \R^{n - 1}$ at $\beta_{t_0} \in \{0\}\times \R^{n - 1}$. Remark \ref{rmk:no-conj-points} guarantees that the $J_i$ remain a (full rank) basis, which, in view of the previous expression for the $J_i$, gives the non-singularity of $\exp_{\lambda x}^{g_\lambda}$ on $D_{\lambda x}$. Unrescaling, this is the non-singularity of $\exp_x^g$ on $D_x$.

For $0 < \rho \leq 4$ such that $\left.\exp_x^{g}\right|_{B_\rho^{g_x}(0)}$ is non-singular and injective, it is a diffeomorphism (smooth away from $S$, $C^{1,1}$ across $S$). We have seen that it is non-singular on $D_x$, so it remains to check that injectivity is not violated before $\rho = 4$. We now define $0 < \rho \leq 4$ maximal such that $\exp_x^g(v) = \exp_x^g(w) \implies v = w$ for $w \in D_x \cap B_\rho^{g_x}(0)$ and $v \in \widetilde D_x \cap B_\rho^{g_x}(0)$. If $\rho < 4$, then by compactness and because $\exp_x$ is a local diffeomorphism, we in fact have $\exp_x^g(v) = \exp_x^g(w)$ for some $v \neq w$ with $v \in \widetilde D_x\cap \bar{B}_\rho^{g_x}(0)$, $w \in D_x\cap \bar{B}_\rho^{g_x}(0)$, and $\max(|v|_{g_x}, |w|_{g_x}) = \rho$. By the choice of $D_x$, $\widetilde D_x$ and the second claim of 2), we have $w \notin \partial D_x$ (this was the reason for enlarging the domain $D_x\supseteq \tilde D_x$). By the Gauss Lemma and then the first claim of 2)

\begin{equation}\label{eq:pos-inner-prod}
\La (D\left.\exp_x^g\right|_U)_w^{-1}(D\exp_x^g)_v(\partial_r),\partial_r\Ra_{g_x} = \La (D\exp_x^g)_v(\partial_r),(D\exp_x^g)_w(\partial_r)\Ra_{g_{\exp_x^g(v)}} > 0\,,
\end{equation}

where $(D\left.\exp_x^g\right|_U)_w^{-1}$ is coming from a local inverse $\left.\exp_x^g\right|_U$ near $\exp_x^g(w)$ into an open set $w \in U \subseteq D_x^\circ$. Now pick $\delta > 0$ small enough so that with $v_\delta := (1 -\delta/|v|_{g_x})v$ we still have $w_\delta := \left.\exp_x^g\right|_U^{-1}\big(\exp_x^g(v_\delta)\big) \in U$. Recognizing that the LHS of \eqref{eq:pos-inner-prod} is $-\left.\partial_\delta\right|_{\delta = 0} |w_\delta|_{g_x}$, the positivity of the RHS implies $|w_\delta|_{g_x} < |w|_{g_x}$ for small enough $\delta > 0$, and of course $|v_\delta|_{g_x} =|(1 - \delta/|v|_{g_x})v|_{g_x} < |v|_{g_x}$. By definition, $v_\delta \in \widetilde D_x\cap \bar{B}_\rho^{g_x}(0)$, $w_\delta \in D_x\cap \bar{B}_\rho^{g_x}(0)$, and $v_\delta \neq w_\delta$ (again requiring $\delta > 0$ small), but now also $\max(|v_\delta|_{g_x}, |w_\delta|_{g_x}) < \max(|v|_{g_x}, |w|_{g_x}) = \rho$. This witnesses the non-injectivity property on a smaller radius than $\rho$, contradicting how $\rho$ was chosen, and we conclude that $\rho = 4$. Since $\widetilde D_x \subseteq D_x$, this shows in particular that $\left.\exp_x^g\right|_{\widetilde D_x}$ is injective, and thus a diffeomorphism.

Recall from above that for $y \in B_1^{g_E}(x)\setminus T_{L\lambda}(x)$, all unit-speed $d_g$-geodesics $\gamma:[0,\ell] \to X$ from $x$ to $y$ satisfy $|\dot\gamma_0^1| \geq (L - C_3)\lambda/8 \geq (L/128 - C_1)\lambda$, and thus $\dot\gamma_0 \in \widetilde D_x$ and $\gamma_t = \exp_x^g(t\dot\gamma_0)$. In particular, $B_1^{g_E}(x)\setminus T_{L\lambda}(x) \subseteq \exp_x^g(\widetilde D_x)$ and

\begin{equation}
y\mapsto d_g^2(x,y) = \ell^2 = |(\left.\exp_x^{g}\right|_{B_2^{g_x}(0)})^{-1}(\gamma_\ell)|_{g_x}^2 = |(\left.\exp_x^{g}\right|_{B_2^{g_x}(0)})^{-1}(y)|_{g_x}^2\,.
\end{equation}

is smooth on $ B_1^{g_E}(x)\setminus \big(T_{L\lambda}\cup(\{0\}\times \R^{n - 1})\big)$ up to the boundary approached from either side, and $C^{1,1}$ on $B_1^{g_E}(x)\setminus T_{L\lambda}$. Finally, by the partial converse in 1) with $L/128 \to L_1$ and the choice of $\widetilde D_x$, we obtain the uniqueness stated in the Lemma. 
\end{proof}

\paragraph{Definition and smallness of $E_x$}
We are now in a position to define $V_x$ and $W_x$ (or equivalently $E_x$), and to verify the claim that the terms $II$ and $III$ of \eqref{eq:error-region-II} resp. \eqref{eq:error-region-III} vanish as $s\to 0$.

\medskip

First cover $\Sigma$ by finitely many coordinate charts $\Phi_j :U_j \to B_{20\rho_j}(0^n)\subseteq \R^{n - 1}$ in the following way: Take a chart $\Phi_\omega: U_\omega \to \R^n$ around each $\omega \in \Sigma$ that maps $\omega$ to $0^n$, $\Sigma \cap U_\omega$ into $\{0\} \times \R^{n - 1}$, and $\nu$ to  $e_1$. Then perform a linear change of variables to obtain $g(0^n) = g_E$, and pick $\rho_\omega$ small enough so that $(1/4) g_E \leq g \leq 4g_E$, $|\partial_k g_{ij}(y)| \leq \rho_\omega^{-1}$, and $|\partial_{k\ell} g_{ij}(y)| \leq \rho_\omega^{-2}$ holds on $B_{20\rho_\omega}(0^n)\setminus (\{0\} \times \R^{n - 1})$. Since $\Sigma$ is compact, take a finite collection $\omega_j \in \Sigma$, $j = 1,\ldots,N$, so that $\Sigma \subseteq \bigcup_{j = 1}^N\Phi_j^{-1}(B_{10\rho_j}(0^n))$, and require $\eps \leq \eps_0(g_1,g_2)$ small enough so that $B_\eps^g(\Sigma)\subseteq \bigcup_{j = 1}^N \Phi_j^{-1}(B_{20\rho_j}(0^n))$.

For $x \in B_\eps^g(\Sigma) \cap \Omega_1^\circ$, we abuse notation by suppressing the chart $\Phi_j$ and identifying $x \in U_j\cong B_{20\rho_j}^{g_E}(0^n) \subseteq \R^{n}$. Fixing $L:= 32\rho_j/\eps$, we define (see also Figure \ref{cap:EVW-regions})

\begin{equation}\label{eq:Ex-def}
E_x := (X\setminus B_{\lambda_0\rho_j}^{g_E}(x)) \cup \left(\{y \in B_{\lambda_0\rho_j}^{g_E}(x): |y^1 - x^1| \leq 2(L/\rho_j)\|x - y\|^2\}\setminus B_{(|x^1|\rho_j/(5L))^{1/2}/2}^{g_E}(x)\right)\,,
\end{equation}

where we require $\eps \leq 32\rho_j/L_0$ with $L_0(n)$ and $\lambda_0(n) \leq 20$ from Lemma \ref{lem:technical-lemma}. By construction, $E_x$ satisfies \eqref{eq:error-volume-decay}, and even with linear $f(r) := (40  L/\rho_j)r$. By translating along $\{0\} \times \R^{n - 1}$, we assume $x = (x^1,0\ldots,0)$ in the chosen coordinate system.  Notice that for any $|x^1|/4 \leq \rho \leq \lambda_0\rho_j$, the rescaled metric $g_{\rho^{-1}}:= g(\rho\cdot)$ and point $x_{\rho^{-1}}:= \rho^{-1}x\subseteq B_{4}^{g_E}(0^n)$ satisfy the assumptions of Lemma \ref{lem:technical-lemma} with $\lambda := \rho/\rho_j$, which outputs a set $F_{\rho^{-1}x}$. Moreover, denoting by $A_{r,R}(z)$ the annulus $B_R(z)\setminus B_r(z)$,

\begin{align}
T_{L\lambda}(x_{\rho^{-1}})\cap A_{1/2\,,\,1}^{g_E}(x_{\rho^{-1}}) &\subseteq \rho^{-1}E_x \cap  A_{1/2\,,\,1}^{g_E}(x_{\rho^{-1}})\,,\quad \forall \rho \geq(|x^1|\rho_j/(5L))^{1/2}\,, \\
\text{while}\quad\rho \leq (|x^1|\rho_j/(5L))^{1/2} &\implies 5L\lambda = (\rho^25L/\rho_j)\rho^{-1} \leq |x^1|\rho^{-1} = |x_{\rho^{-1}}^1|\,.
\end{align}

Thus, for any $y \in B_{\lambda_0\rho_j}^{g_E}(x)\setminus E_x$ and $\rho := \|x - y\|$, we have $y_{\rho^{-1}} \in B_1^{g_E}(x_{\rho^{-1}})\setminus F_{\rho^{-1}x}$ (in the case that $\|x - y\| < |x_1|/4$ is outside the allowed range for $\rho$, also $y \in B^{g}_{|x_1|/2}(x)\subseteq \R_{<0}\times \R^{n - 1}$ is in a smooth complete precompact Riemannian ball centered at $x$, i.e. satisfies all claimed properties of $W_x$). 

The claim \eqref{eq:small-inner-prod-small-vol} follows because the unique unit-speed $d_{g_{\rho^{-1}}}$-geodesic $\gamma$ from $x_{\rho^{-1}}$ to $y_{\rho^{-1}}$ has at most one point of intersection $\gamma_t \in \{0\} \times \R^{n - 1}$ at which time $\La \dot\gamma_t,e_1\Ra \geq L\lambda /16 = L\|x - y\|/(16\rho_j) \geq \eps^{-1} d_g(x,y)$ (using repeatedly that $d_g$ and $d_{g_E}$ are $2$-biLipschitz). Also by the Lemma, if $y_{\rho^{-1}} =\gamma_\ell \in (\{0\} \times \R^{n-1})\setminus \rho^{-1}E_x$, then $\dot\gamma_{\ell}$ points into $\R_{> 0} \times \R^{n-1}$ and can be extended using the $g_{\rho^{-1}}$ exponential map on $\R_{\geq 0}\times \R^{n-1}$. Such a curve remains a $d_{g_{\rho^{-1}}}$-geodesic by the uniqueness statement of the Lemma, witnessing $y_{\rho^{-1}}$ as an interior point of a $d_{g_{\rho^{-1}}}$-geodesic ending in $(\R_{> 0} \times \R^{n - 1})\setminus \rho^{-1}E_x$. 

Suppose now that $\alpha$ is a $C^1$ unit-speed curve connecting $x_{\rho^{-1}}$ and $y_{\rho^{-1}}$, of $g_{\rho^{-1}}$-length $l \leq 2d_{g_{\rho^{-1}}}(x_{\rho^{-1}},y_{\rho^{-1}}) \leq 4$, smooth and solving the $g_{\rho^{-1}}$ geodesic equations on $\alpha \cap (\R_{\lessgtr 0}\times \R^{n - 1})$, with at most one time of intersection $\alpha_t \in \{0\} \times \R^{n - 1}$ at which time $d_{g}(x,y) \leq 2\eps\La \dot\alpha_t,e_1\Ra$. Unpacking notation, this inequality is $(L\lambda/128) \leq \La \dot\alpha_t,e_1\Ra$, which by the Lemma implies that $\alpha$ was in fact the $d_{g_{\rho^{-1}}}$-geodesic connecting $x_{\rho^{-1}}$ to $y_{\rho^{-1}}$. Unrescaling these observations and the other conclusions from Lemma \ref{lem:technical-lemma}, we obtain all properties in Claim \ref{claim:geo}. 

\bigskip

We now show that $II$ and $III$ (see \eqref{eq:error-region-II} resp. \eqref{eq:error-region-III}) vanish as $s\to 0$ for fixed $\eps > 0$. We begin with $II$:

\begin{align}
|II| &\leq \frac{1}{s}\int_{B_\eps^g(\Sigma)\cap \Omega_1^\circ}\int_{E_x} P(d^2_g(x,y)/s)\frac{e^{-d^2_g(x,y)/12s}}{(12\pi s)^{n/2}}\,d\mu(y)d\mu(x)\,,
\end{align}

where $P$ denotes a polynomial whose particular form is of no consequence. We note that the contribution from $y \in E_x\setminus B_{s^{(1 - \delta)/2}}^g(x)$, for $0 < \delta < 1$ fixed, is $O(e^{-s^{-\delta}})$ and vanishes as $s\to 0$. For $y \in B_{s^{(1-\delta)/2}}^g(x)$, we can assume that we are working in a coordinate chart around $x = (\rho,0^{n - 1}) \in B_{20\rho_j}^{g_E}(0^n)\subseteq \R^n$ as above, and estimate everything in terms of the Euclidean quantities. From now on, $C = C(\eps,g_1,g_2)$ is a constant that may change from line to line.

\begin{align}
&\phantom{\leq}\frac{1}{s}\int_{B_\eps^g(\Sigma)\cap \Omega_1^\circ}\int_{E_x \cap B_{s^{(1 - \delta)/2}}^g(x)} P(d_g^2(x,y)/s)\frac{e^{-d_g^2(x,y)/12s}}{(12\pi s)^{n/2}}\,d\mu(y)d\mu(x) \\
&\leq \frac{C}{s^{1 + n/2}}|\Sigma|\int_0^{Cs^{1-\delta}} \int_{T_{Cs^{(1 - \delta)/2}}(\rho,0^{n - 1})}P(C\|(\rho,0^{n - 1}) - y\|^2/s) e^{-\|((\rho,0^{n - 1})-y\|^2/(12Cs)}dy^nd\rho\,.
\end{align}

Note that only points $x$ with $|x^1| < Cs^{1 - \delta}$ have $E_x \cap B^{g_E}_{2s^{(1 - \delta)/2}}(x)$ nonempty, whence the upper bound of integration in the $\rho$ variable. We also used $B_r^{g_E}(x) \cap E_x \subseteq T_{Cr}(x)$ for all $r \leq C^{-1}$. Now center $y$ to remove $\rho$ from the inner integral and rescale $y\rightsquigarrow s^{(1 - \delta)/2}y$, then switch to $y$-polar coordinates around the origin

\begin{align}
\cdots &\leq Cs^{-(1 + n/2) + (1-\delta)(1 + n/2)}\int_{T_{Cs^{(1 - \delta)/2}}(0^n)} P(C\|y\|^2s^{-\delta})e^{-C^{-1}\|y\|^2s^{-\delta}}\,dy^n \\
&\leq C s^{- \delta(1 + n/2)}|\{\vartheta \in S^{n - 1} \mid \La \vartheta, e_1\Ra \leq Cs^{(1 -\delta)/2}\}|\int_0^\infty P(Cr^2)e^{-C^{-1}r^2}r^{n - 1}\,dr  \leq C s^{1/2 - \delta(3/2 + n/2)}\,.
\end{align}

So we indeed obtain $|II| = O(s^{1/2-\delta'})$ for any $\delta' > 0$ by picking $0 < \delta < 1/(3 + n)$ arbitrarily small above. We omit the estimate of $III$, because it follows the exact same pattern as for $II$. The only difference is that the smallness $|\{\vartheta \in S^{n - 1}\mid \La \vartheta, e_1\Ra \leq Cs^{(1 -\delta)/2}\}| \leq Cs^{(1 -\delta)/2}$ of the sector used to estimate the volume of $E_x$ at scale $s^{(1-\delta)/2}$ no longer holds for its boundary. Instead one uses the final claim of Lemma \ref{lem:technical-lemma} to obtain for all $y\in \partial^\pm E_x\cap B^{g_E}_{s^{(1 - \delta)/2}}(x)$ with $|x^1|\leq C s^{1 - \delta}$,

\begin{equation}
|\La\tilde\nu_y,(\nabla d_g(x,\cdot))_y\Ra_{g_y}| \leq \underbrace{|\La e_1, (\nabla d_g(x,\cdot))_y\Ra_{g_{(0,y^2,\ldots, y^n)}}|}_{|\dot\gamma_{d_g(x,y)}^1|} + 2|\pm e_1 - \tilde\nu_y|_{g_y} + 2|g_{(0,y^2,\ldots, y^n)}- g_y|\leq Cs^{(1 - \delta)/2}\,,
\end{equation}

where the LHS appears as an additional factor in the integrand of $III$.

\section{Proof of Theorem \ref{2d-cone}}\label{app:cone}
We will compute $\KK$ for cones, which by the gluing result is enough. Consider then $C(S^1)$ the cone of angle $2\pi-\alpha$, $x\in C$ a point at distance $r$ from the vertex, the challenge is then computing $|\partial B_R(x)|$.

\subsection{Case $\alpha \leq 0$.}
In this case geodesics are simply the Euclidean ones or they pass through the origin. Therefore:
\begin{itemize}
    \item when $R<r$, $B_R(x)$ is isometric to a Euclidean ball;
    \item when $R>r$, $B_R(x)$ is made of a Euclidean ball of radius $R$, $\partial B_R^{(1)}$, plus a sector of angle $|\alpha|$ and radius $R-r$, $\partial B_R^{(2)}$, see Figure \ref{alphaneg};
\end{itemize} 
and hence $|\partial B_R(x)|=2\pi R -\alpha (R-r) \ind_{{R>r}}$. 

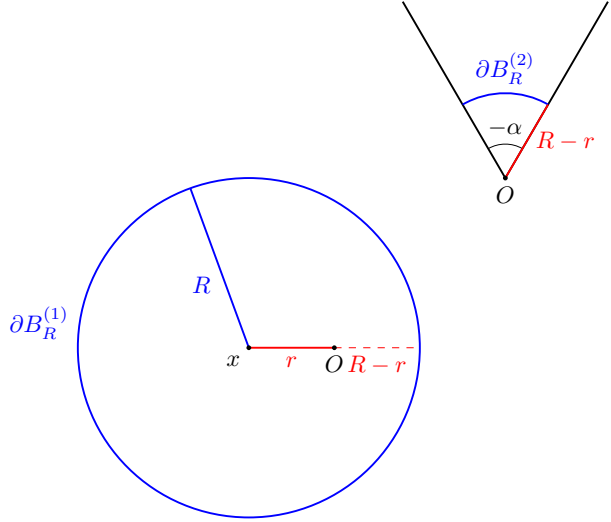
\begin{figure}[h]\label{fig:cone<0}
\centering
\scalebox{0.9}{\begin{tikzpicture}[scale=0.5]

\def\R{5}
\def\r{2.5}
\def\ang{120}
\def\L{6}
\def\l{5}

\coordinate (O) at (0,0);
\coordinate (x) at (-\r,0);
\coordinate (Op) at (5,5);

\draw[black, thick] (Op) -- ++(60:\L);
\draw[black, thick] (Op) -- ++(120:\L);

\draw[blue, thick]
  (x) ++(0:\l)
  arc (0:360:\l)
  node[midway, above left] {$\partial B_R^{(1)}$};
\draw[blue, thick]
  (Op) ++(60:\R-\r)
  arc (60:120:\R-\r)
  node[midway, above] {$\partial B_R^{(2)}$};

\draw[black]
  (Op) ++(60:1)
  arc (60:120:1)
  node[midway, above] {$-\alpha$};

\draw[red, thick] (x) -- (O)
  node[midway, below] {$r$};
\draw[red, dashed] (O) -- ($(O) + (\R-\r,0)$)
    node[midway, below] {$R-r$};
\draw[blue, thick] (x) -- ($(x) +(110:\R)$)
    node[midway, below left] {$R$};
\draw[red, thick] (Op) -- ($(Op) +(60:\R-\r)$)
    node[midway, right] {$R-r$};

\node[below left] at (x) {$x$};
\fill (x) circle (2pt);
\node[below] at (O) {$O$};
\fill (O) circle (2pt);
\node[below] at (Op) {$O$};
\fill (Op) circle (2pt);

\end{tikzpicture}}
\caption{When $\alpha < 0$, $R>r$, $\partial B_R = \partial B_R^{(1)} \sqcup \partial B_R^{(2)}$}\label{alphaneg}
\end{figure}

Let $C_\rho$ be the ball of radius $\rho$ around the vertex, we compute
\begin{align}\label{eq:comp_alpha<0}
&\int_{C_\rho}\a_s(x)d\vol(x) \\
    &= \int_{C_\rho}\frac{1}{12\pi s}\int_C e^{-d^2(x,y)/12s}d\vol(y)d\vol(x) \\
    &= \frac{1}{12\pi s} \int_0^\rho \int_0^\infty e^{-R^2/12s}|\partial B_R(x)| dR \ |\partial B_r(0)| dr \\
    &= \frac{1}{12\pi s} \int_0^\rho\int_0^\infty e^{-R^2/12s}(2\pi R -\alpha (R-r) \ind_{{R>r}})dR \ (2\pi-\alpha)r dr \\
    &= \frac{2\pi- \alpha}{2}\rho^2 - \frac{\alpha(2\pi-\alpha)}{12\pi s} \left[ \int_0^\rho\int_0^R e^{-R^2/12s}(R-r)r \ drdR + \int_\rho^\infty\int_0^\rho e^{-R^2/12s}(R-r) r \ drdR\right] \\
    &= \frac{2\pi-\alpha}{2}\rho^2 - \frac{\alpha(2\pi-\alpha)}{12\pi s} \Bigg[ \int_0^\infty\int_0^R e^{-R^2/12s}(R-r)r \ drdR - \underbrace{\int_\rho^\infty\int_\rho^R e^{-R^2/12s}(R-r) r \ drdR}_{=O(s^2)}\Bigg] \label{eq:intOs2}\\
    &= \frac{2\pi-\alpha}{2}\rho^2 - \frac{\alpha(2\pi-\alpha)}{12\pi s} \int_0^\infty e^{-R^2/12s} \frac{R^3}{6}dR + O(s^2) \\
    &= \frac{2\pi-\alpha}{2}\rho^2 - \frac{\alpha(2\pi-\alpha)}{\pi}s + O(s^2)
\end{align}
hence
\begin{align}
    \KK = \frac{\alpha(2\pi-\alpha)}{\pi} \leq0 \,.
\end{align}

In the previous computations we used at line \eqref{eq:intOs2} that the integral on $[\rho,\infty)$ is of order $O(s^2)$. More in general, for any $f(R)$ with polynomial growth $0 \leq f(R)\leq C(1+R^d)$, for $\rho>0$ fixed,
\begin{align}\label{eq:expdecayoutside0}
    \frac{1}{s}\int_\rho^\infty e^{-R^2/12s}f(R) dR &=\int_0^\infty e^{-(\rho+sv)^2/12s}f(\rho+sv)  dv \\
    &\leq e^{-\rho^2/12s}\int_0^\infty e^{-(2\rho v+sv^2)/12}C(1+(sv)^d) dv \\
    & \leq e^{-\rho^2/12s} \left[\int_0^\infty e^{-2\rho v}C dv + \int_0^\infty e^{-2\rho v}C (sv)^d dv\right] \\
    & = O(e^{-\rho^2/12s})
\end{align}

\subsection{Case $\alpha\in[0,\pi]$.}
In this case $|\partial B_R(x)|=(2\pi-\beta\ind_{{R>r}})R$, for some defect angle $\beta$, see Figure \ref{fig:Cone0pi}. In order to find $\beta$ we note that 
\begin{align}
    \tan(\alpha/2) = \frac{h}{b} = \frac{R\sin(\beta/2)}{R\cos(\beta/2)-r}.
\end{align}
Using the identity $A\sin x+ B\cos x=\sqrt{A^2+B^2}\sin(x+\arcsin(\frac{B}{\sqrt{A^2+B^2}}))$, we get
\begin{align}
    \beta^{\alpha,r,R} = \alpha - 2\arcsin\left(\frac{r}{R}\sin(\alpha/2)\right).
\end{align}

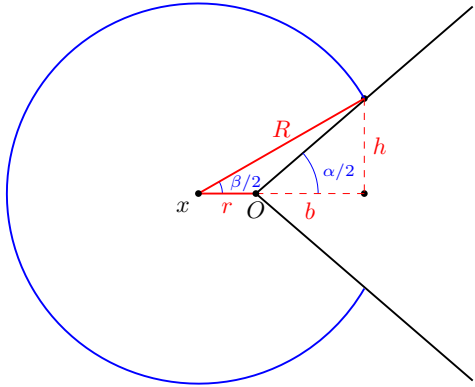
\begin{figure}[ht]
\centering
\scalebox{0.9}{\begin{tikzpicture}[scale=0.7]

\def\R{4} 
\def\r{1.2}
\def\ang{30} 
\def\L{6}
\def\l{4}

\coordinate (O) at (0,0);
\coordinate (x) at (-\r,0);

\coordinate (I) at ({\R*cos(\ang)}, {\R*sin(\ang)});

\coordinate (H) at (\r, {\R*sin(\ang)});

\draw[black, thick] (O) -- ++(\ang+11:\L);
\draw[black, thick] (O) -- ++(-\ang-11:\L);

\coordinate (A) at ($(x)+(\ang:\l)$);
\fill (A) circle (2pt);

\path let \p1 = (A) in
coordinate (Ax) at (\x1,0);
\fill (Ax) circle (2pt);

\draw[blue, thick]
  (x) ++(\ang:\l)
  arc (\ang:{360-\ang}:\l);

\draw[blue]
  (O) ++(\ang+11:1.3)
  arc (\ang+11:0:1.3)
  node[midway, right, font=\scriptsize] {$\alpha/2$};

\draw[blue]
  (x) ++(\ang:0.5)
  arc (\ang:0:0.5)
  node[midway, yshift=2pt, right, font=\scriptsize] {$\beta/2$};

\draw[red, thick] (x) -- (O)
  node[midway, below] {$r$};
\draw[red, thick] (x) -- (A)
  node[midway, above] {$R$};

\draw[red, dashed] (A) -- (Ax)
  node[midway, right] {$h$};
\draw[red, dashed] (Ax) -- (O)
  node[midway, below] {$b$};

\node[below left] at (x) {$x$};
\fill (x) circle (2pt);
\node[below] at (O) {$O$};
\fill (O) circle (2pt);
\end{tikzpicture}}
\caption{$\alpha\in[0,\pi]$, $R>r$}\label{fig:Cone0pi}
\end{figure}

Let $C_\rho$ be the ball around the vertex of radius $\rho$, then 
\begin{align}
    \int_{C_\rho}\a_s(x)d\vol(x) &= \int_{C_\rho}\frac{1}{12\pi s}\int_C e^{-d^2(x,y)/12s}d\vol(y)d\vol(x) \\
    &= \frac{1}{12\pi s} \int_0^\rho\int_0^\infty e^{-R^2/12s}|\partial B_R(x)| dR \ |\partial B_r(0)| dr \\
    &= \frac{1}{12\pi s} \int_0^\rho\int_0^\infty e^{-R^2/12s} R \left( 2\pi-\beta^{\alpha,r,R}\ind_{\{R>r\}} \right) dR \ (2\pi-\alpha) r dr \\
    &= \frac{1}{12\pi s} \int_0^\rho\int_0^\infty e^{-R^2/12s}  2\pi R \  dR \ (2\pi-\alpha) r dr \\
    &\qquad - \frac{1}{12\pi s} \int_0^\rho\int_0^\infty e^{-R^2/12s} \beta^{\alpha,r,R}\ind_{\{R>r\}} R \ dR \ (2\pi-\alpha) r dr \\
    &= \frac{2\pi-\alpha}{2}\rho^2 - \frac{1}{12\pi s} \int_0^\rho\int_0^\infty e^{-R^2/12s} \beta^{\alpha,r,R}\ind_{\{R>r\}} R \ dR \ (2\pi-\alpha) r dr
\end{align}
as the first integral is the same as in $\mathbb{R}^2$.
\begin{align}
    ...&= \frac{2\pi-\alpha}{2}\rho^2 -\frac{2\pi-\alpha}{12\pi s} \int_0^\infty e^{-R^2/12s} R \int_0^\rho \beta^{\alpha,r,R}\ind_{\{R>r\}} r dr \ dR \\
    &= \frac{2\pi-\alpha}{2}\rho^2 - \frac{2\pi-\alpha}{12\pi s} \int_0^\rho e^{-R^2/12s} R \int_0^R \beta^{\alpha,r,R} r dr \ dR -  \frac{2\pi-\alpha}{12\pi s} \int_\rho^{\infty} e^{-R^2/12s} R \int_0^\rho\beta^{\alpha,r,R} r dr \ dR \\
    &\qquad +(1-1)\frac{2\pi-\alpha}{12\pi s} \int_\rho^{\infty} e^{-R^2/12s} R \int_0^R \beta^{\alpha,r,R} r dr \ dR \\
    &= \frac{2\pi-\alpha}{2}\rho^2 - \frac{2\pi-\alpha}{12\pi s} \int_0^\infty e^{-R^2/12s} R \int_0^R \beta^{\alpha,r,R} r dr \ dR + \frac{2\pi-\alpha}{12\pi s} \int_\rho^{\infty} e^{-R^2/12s} R \int_\rho^R\beta^{\alpha,r,R} r dr \ dR
    \end{align}

The last integral is now of order $O(s^2)$ (in fact $O(e^{-1/s})$) by applying \eqref{eq:expdecayoutside0} with $f(R)=R \int_\rho^R\beta^{\alpha,r,R} r dr$, which has cubic growth:
\begin{align}
    \int_\rho^R \left[\alpha - 2\arcsin\left(\frac{r}{R}\sin(\alpha/2)\right) \right]r dr = \frac{\alpha}{2} (R^2-\rho^2) - 2\int_{\rho/R}^1 \arcsin\left(u\sin(\alpha/2)\right)uR^2du = O(R^2).
\end{align}

The first two terms instead give
\begin{align}
    &\frac{2\pi-\alpha}{2}\rho^2 - \frac{2\pi-\alpha}{12\pi s} \int_0^\infty e^{-R^2/12s} R \int_0^R \beta^{\alpha,r,R} r dr \ dR \\
    &\qquad = \frac{2\pi-\alpha}{2}\rho^2 - \frac{(2\pi-\alpha)}{12\pi s}\int_0^\infty e^{-R^2/12s}R\int_0^{R} \left(\alpha - 2\arcsin \left(\frac{r}{R}\sin(\alpha/2)\right)\right)rdrdR \\
    &\qquad = \frac{2\pi-\alpha}{2}\rho^2 - (2\pi-\alpha) \frac{3\alpha}{\pi}s +\frac{(2\pi-\alpha)}{12\pi s}\int_0^\infty e^{-R^2/12s}R \int_0^{R}2\arcsin \left(\frac{r}{R}\sin(\alpha/2)\right)r drdR
\end{align}
Now use the fact that
\begin{align}\label{eq:arcsin_int}
    \int \arcsin(ax)x dx = \arcsin(ax)\frac{x^2}{2}+ \frac{1}{4}x\sqrt{\frac{1}{a^2}-x^2}-\frac{1}{4a^2}\arcsin(ax)
\end{align}
to get that
\begin{align}
    \int_0^{R}2\arcsin \left(\frac{r}{R}\sin(\alpha/2)\right)r dr = \frac{1}{2}\left( \alpha + \frac{\cos(\alpha/2)\sin(\alpha/2)-\alpha/2}{\sin^2(\alpha/2)} \right)R^2 =: \frac{1}{2}\phi(\alpha)R^2
\end{align}
Now we can conclude
\begin{align}
    \int_{C_\rho}\a_s(x)d\vol(x) &= \frac{2\pi-\alpha}{2}\rho^2 - 3(2\pi-\alpha) \frac{\alpha}{\pi}s +\frac{(2\pi-\alpha)}{12\pi s}\int_0^\infty e^{-R^2/12s}\frac{\phi(\alpha)}{2}R^3 dR  + O(s^2)\\
    &= \frac{2\pi-\alpha}{2}\rho^2 + \frac{3(2\pi-\alpha)}{\pi} \left(-\alpha +\phi(\alpha) \right)s  + O(s^2) \\
    &= \frac{2\pi-\alpha}{2}\rho^2 + \frac{3(2\pi-\alpha)}{\pi} \left(\frac{\cos(\alpha/2)\sin(\alpha/2)-\alpha/2}{\sin^2(\alpha/2)} \right)s + O(s^2)
\end{align}

Hence we have that
\begin{align}
    \KK &= -\frac{d}{ds}|_{s=0}\int_{C_\rho} \a_s(x) d\vol(x) \\
    &= -\frac{3(2\pi-\alpha)}{\pi} \left(\frac{\cos(\alpha/2)\sin(\alpha/2)-\alpha/2}{\sin^2(\alpha/2)} \right) \\
    &= 6\left(1-\frac\alpha{2\pi}\right)
\frac{\alpha-\sin\alpha}{1-\cos\alpha}
\end{align}
Note that $\KK (\pi)=3\pi/2$ and $\KK \to 0$ as $\alpha \to 0$.

\subsection{Case $\alpha\in[\pi,2\pi)$.}
In this case the length of $\partial B_R(x)$ can take three different forms, see Figures \ref{fig:2pi1} \ref{fig:2pi2}:
\begin{align}
|\partial B_R(x)|= 
    \begin{cases}
        2\pi R & 0\leq R \leq d; \\
        2\beta R & R\geq r; \\
        2(\beta + \gamma)R & d\leq R \leq r
    \end{cases}
\end{align}
where $d=r\sin(\frac{2\pi-\alpha}{2})$, $\beta = \frac{3\pi}{2}-\frac{\alpha}{2}-\arccos(\frac{d}{R})$ and $\gamma=\arccos(\frac{R^2+r^2-l^2}{2Rr})$, $l= r\cos (\frac{2\pi-\alpha}{2})-\sqrt{R^2-r^2\sin^2(\frac{2\pi-\alpha}{2})}$. The angle $\gamma$ is found using the cosine law twice:
\begin{align}
    &R^2 = r^2 + l^2 - 2rl\cos(\frac{2\pi-\alpha}{2}) \\
    &l^2 = R^2 + r^2 - 2Rr \cos\gamma.
\end{align}

\begin{figure}[ht]
\centering
  \begin{minipage}[b]{0.35\textwidth}
    \scalebox{1}{\begin{tikzpicture}[scale=0.8]

\def\R{4}
\def\r{2.5} 
\def\ang{120}
\def\L{6}
\def\l{4}

\coordinate (O) at (0,0);
\coordinate (x) at (-\r,0);

\coordinate (I) at ({\R*cos(\ang+22.5)}, {\R*sin(\ang+22.5)});

\coordinate (H) at ($(O)!(x)!(I)$);
\fill (H) circle (2pt);
\draw[orange, dashed] (x) -- (H)
    node[midway, above left] {$d$};

\draw[black, thick] (O) -- ++(\ang+22.5:\L);
\draw[black, thick] (O) -- ++(-\ang-22.5:\L);

\coordinate (A) at ($(x)+(\ang:\l)$);
\fill (A) circle (2pt);

\draw[blue, thick]
  (x) ++(\ang:\l)
  arc (\ang:{360-\ang}:\l);

\draw[blue]
  (O) ++(\ang+22:0.5)
  arc (\ang+22:0:0.5)
  node[midway, above, font=\scriptsize] {$\alpha/2$};

\draw[blue]
  (x) ++(\ang:0.5)
  arc (\ang:180:0.5)
  node[midway, yshift=1pt, left, font=\scriptsize] {$\beta$};

\draw[red, thick] (x) -- (O)
  node[midway, below] {$r$};
\draw[red, thick] (x) -- (A)
  node[midway, left] {$R$};

\draw[red, dashed] (2,0) -- (-4.5,0);

\node[below left] at (x) {$x$};
\fill (x) circle (2pt);
\node[below] at (O) {$O$};
\fill (O) circle (2pt);

\end{tikzpicture}}
    \caption{$\alpha\in[\pi,2\pi)$, $R>r$}\label{fig:2pi1}
    \end{minipage}
  \hspace{3cm}
  \begin{minipage}[b]{0.35\textwidth}
    \scalebox{1}{\begin{tikzpicture}[scale=0.8]

\def\R{1}    
\def\r{2.5}   
\def\ang{120}  
\def\L{6}
\def\l{2}

\coordinate (O) at (0,0);    
\coordinate (x) at (-\r,0);

\draw[black, thick] (O) -- ++(\ang+22.5:\L);
\draw[black, thick] (O) -- ++(-\ang-22.5:\L);

\coordinate (A) at ($(x)+(\ang-27:\l)$);
\fill (A) circle (2pt);

\coordinate (B) at ($(x)+(12:\l)$);
\fill (B) circle (2pt);

\draw[blue, thick]
  (x) ++(\ang-27:\l)
  arc (\ang-27:{360-\ang+27}:\l);

\draw[blue, dashed]
  (x) ++(\ang-27:\l)
  arc (\ang-27:13:\l);
  
\draw[blue, dashed]
  (x) ++(-\ang+27:\l)
  arc (-\ang+27:0:\l);
  
\draw[blue, thick]
  (x) ++(-12:\l)
  arc (-12:12:\l);

\draw[blue]
  (O) ++(\ang+22:0.5)
  arc (\ang+22:0:0.5)
  node[midway, above right, font=\scriptsize] {$\alpha/2$};

\draw[blue]
  (x) ++(\ang-27:0.5)
  arc (\ang-27:180:0.5)
  node[midway, yshift=1pt, left, font=\scriptsize] {$\beta$};

\draw[orange]
  (x) ++(0:1.4)
  arc (0:12:1.4)
  node[midway, right, font=\scriptsize] {$\gamma$};

\draw[red, thick] (x) -- (O)
  node[midway, below] {$r$};
\draw[red, thick] (x) -- (A)
  node[midway, left] {$R$};
\draw[orange, dashed] (x) -- (B);
\draw[orange] (O) -- (B)
node[midway, yshift=3pt, right] {$l$};

\draw[red, dashed] (1,0) -- (-3.6,0);

\node[below left] at (x) {$x$};
\fill (x) circle (2pt);
\node[below] at (O) {$O$};
\fill (O) circle (2pt);
\end{tikzpicture}}
    \caption{$\alpha\in[\pi,2\pi)$, $d<R<r$}\label{fig:2pi2}
  \end{minipage}
\end{figure}

We then get
\begin{align}
    \int_{C_\rho} \a_s(x) d\vol(x) 
    &= \int_{C_\rho}\frac{1}{12\pi s}\int_C e^{-d^2(x,y)/12s}d\vol_g(y)d\vol_g(x) \\
    &= \frac{1}{12\pi s} \int_0^\rho \int_0^\infty e^{-R^2/12s}|\partial B_R(x)| dR \ |\partial B_r(0)| dr \\
    &= \frac{1}{12\pi s} \int_0^\rho \Bigg[ \int_0^d e^{-R^2/12s}2\pi R dR + \int_d^r e^{-R^2/12s}2(\beta^{\alpha,r,R}+\gamma^{\alpha,r,R})RdR \\ & \hspace{3cm} +  \int_r^\infty e^{-R^2/12s} 2\beta^{\alpha,r,R}RdR\Bigg](2\pi-\alpha)r dr \\
    & \ \ \ +(1-1)\frac{1}{12\pi s}\int_0^\rho \int_d^\infty e^{-R^2/12s}2\pi R dR (2\pi-\alpha)r dr
\end{align}
so that
\begin{align}
    &= \frac{1}{12\pi s} \int_0^\rho \Bigg[ \int_0^\infty e^{-R^2/12s}2\pi R dR + \int_d^r e^{-R^2/12s} 2(\beta^{\alpha,r,R}+\gamma^{\alpha,r,R}-\pi)RdR \\ & \hspace{3cm} +  \int_r^\infty e^{-R^2/12s} 2(\beta^{\alpha,r,R} - \pi)RdR\Bigg](2\pi-\alpha)r dr \\
    &= \frac{2\pi-\alpha}{2}\rho^2 + \frac{2\pi-\alpha}{12\pi s} \int_0^{\infty} \Bigg [ \int_R^{R/\sin(\frac{2\pi-\alpha}{2})} 2(\beta+\gamma -\pi)rdr + \int_0^R 2(\beta-\pi) rdr \Bigg] Re^{-R^2/12s}dR + O(s^2) \\
    &= \frac{2\pi-\alpha}{2}\rho^2 + \frac{2\pi-\alpha}{12\pi s} \int_0^\infty \Bigg [ \int_0^{R/\sin(\frac{2\pi-\alpha}{2})} 2(\beta-\pi) rdr + \int_R^{R/\sin(\frac{2\pi-\alpha}{2})} 2\gamma rdr \Bigg] Re^{-R^2/12s}dR + O(s^2)
\end{align}
Now we evaluate the two integrals. For the first one, note that 
\begin{align}
    \int x \arccos(ax) dx = \frac{x^{2}}{2}\arccos(ax) + \frac{1}{4a^{2}}\arcsin(ax)- \frac{x}{4a}\sqrt{1 - a^{2}x^{2}}
\end{align}
hence
\begin{align}
    \int_0^{R/\sin(\frac{2\pi-\alpha}{2})} 2(\beta -\pi)rdr &= \frac{(\pi-\alpha)R^2}{2\sin^2(\frac{2\pi-\alpha}{2})} -\frac{R^2}{2\sin^2(\frac{2\pi-\alpha}{2})}\frac{\pi}{2} \\
    &= -\frac{\alpha-\pi/2}{2\sin^2(\frac{2\pi-\alpha}{2})}R^2 \\
    &= -\frac{\alpha-\pi/2}{1-\cos\alpha}R^2
\end{align}
For the second integral, write $$l^2=\left(-r\cos(\alpha/2)-\sqrt{R^2-r^2\sin^2(\alpha/2)}\right)^2$$
and
\begin{align}
    \gamma = \arccos\left( \frac{r}{R}\sin^2(\alpha/2) - \cos(\alpha/2)\sqrt{1- \frac{r^2}{R^2}\sin^2(\alpha/2)} \right) \,,
\end{align}
then
\begin{align}
    &\int_R^{R/\sin(\frac{2\pi-\alpha}{2})}  2 \arccos\left( \frac{r}{R}\sin^2(\alpha/2) - \cos(\alpha/2)\sqrt{1- \frac{r^2}{R^2}\sin^2(\alpha/2)} \right) rdr \\
    &\qquad = R^2 \int_1^{1/\sin(\frac{2\pi-\alpha}{2})} \arccos\left( s\sin^2(\alpha/2) - \cos(\alpha/2)\sqrt{1- s^2 \sin^2(\alpha/2)} \right) sds \\
    &\qquad =: \eta(\alpha)R^2\,,
\end{align}
we will compute $\eta(\alpha)$ at the end of the proof.

Altogether we computed
\begin{align}
    \int_{C_\rho} \a_s(x) d\vol(x) &= \frac{2\pi-\alpha}{2}\rho^2 + \frac{2\pi-\alpha}{12\pi s} \left[ -\frac{\alpha-\pi/2}{1-\cos\alpha} +\eta(\alpha) \right] \int_0^\infty R^3 e^{-R^2/12s}dR + O(s^2) \\
    &= \frac{2\pi-\alpha}{2}\rho^2 + \frac{2\pi-\alpha}{12\pi s} \left[ -\frac{\alpha-\pi/2}{1-\cos\alpha} +\eta(\alpha) \right] 72s^2 + O(s^2) \\
    &= \frac{2\pi-\alpha}{2}\rho^2 + \frac{2\pi-\alpha}{\pi} \left[ -\frac{\alpha-\pi/2}{1-\cos\alpha} +\eta(\alpha) \right] 6s + O(s^2)\\
\end{align}
and hence
\begin{align}
    \KK = 12\left(1-\frac{\alpha}{2\pi}\right)\left[ \frac{\alpha-\pi/2}{1-\cos\alpha} -\eta(\alpha) \right] \,.
\end{align}

In order to compute $\eta(\alpha)$, we use the cosine difference identity and recognise the argument of $\arccos$ in the integrand:
\begin{align}
    \cos\left(\arcsin(s\sin\frac{\alpha}{2}) - (\pi-\frac{\alpha}{2}) \right) &= \cos(\arcsin(s\sin\frac{\alpha}{2})) \cos(\pi-\frac{\alpha}{2}) + \sin\arcsin(s\sin\frac{\alpha}{2})  \sin(\pi-\frac{\alpha}{2})\\
    &= -\cos\frac{\alpha}{2}\sqrt{1-s^2\sin^2\frac{\alpha}{2}} + s\sin^2\frac{\alpha}{2}
\end{align}
so that
\begin{align}
    \eta(\alpha)= \int_1^{1/\sin(\alpha/2)} \left(\arcsin(s\sin\frac{\alpha}{2}) - (\pi-\frac{\alpha}{2})\right)s ds \,.
\end{align}
In order to integrate this we use \eqref{eq:arcsin_int}
\begin{align}
    \eta(\alpha) &= \left[\frac{\pi}{4}\frac{1}{\sin^2\frac{\alpha}{2}} - \frac{\pi}{8}\frac{1}{\sin^2\frac{\alpha}{2}} - \frac{1}{2}(\pi-\frac{\alpha}{2}) - \frac{1}{4}\frac{|\cos\frac{\alpha}{2}|}{\sin\frac{\alpha}{2}} + \frac{1}{4}(\pi-\frac{\alpha}{2}) \frac{1}{\sin^2\frac{\alpha}{2}}\right] - \frac{1}{2}(\pi-\frac{\alpha}{2})\left(\frac{1}{\sin^2\frac{\alpha}{2}} -1 \right)\\
    &= -\frac{\pi}{8}\frac{1}{\sin^2\frac{\alpha}{2}} + \frac{\alpha}{8}\frac{1}{\sin^2\frac{\alpha}{2}} + \frac{\cos\frac{\alpha}{2}}{4\sin\frac{\alpha}{2}} \\
    &=\frac{2}{1-\cos\alpha}\left( \frac{\alpha-\pi}{8} + \frac{\sin\alpha}{8} \right) \,.
\end{align}
Altogether,
\begin{align}
    \KK &= 12\left(1-\frac{\alpha}{2\pi}\right)\left[ \frac{\alpha-\pi/2}{1-\cos\alpha} - \frac{\alpha-\pi+\sin\alpha}{4(1-\cos\alpha)} \right] \\
    &= 3 \left(1-\frac{\alpha}{2\pi}\right) \left[ \frac{3\alpha - \pi - \sin\alpha}{1-\cos\alpha} \right] \,.
\end{align}

\section{Proof of Theorem \ref{2d-sphsusp}}\label{app:sphsusp}
We divide the discussion in three cases as in the cone case. We write the integral $\int \a_s d\vol$ and show that its first derivative at zero has a contribution from the tip singularity identical to the one of the cone. For the sequel, we denote the volume of geodesic circles on a unit sphere by $C(r) = 2\pi\sin r = 2\pi r(1-\frac{r^2}{6}+O(r^4))$.

\subsection{Case $\alpha \leq 0$}
Let $S_\rho$ be the geodesic ball around the tip of the spherical suspension, let $x\in S_\rho$ at distance $r$ from $0$ and $B_R(x)$ the ball around $x$ in the suspension. As in Figure \ref{alphaneg}, $|\partial B_R(x)| = C(R) - C(R-r)\frac{\alpha}{2\pi}\ind_{\{R>r\}}$. Then, just as in the cone case
\begin{align}
    \int_{S_\rho} \a_s(x)d\vol(x) &= \frac{1}{12\pi s} \int_0^\rho \int_0^\infty e^{-R^2/12s}|\partial B_R(x)| dR \ |\partial B_r(0)| dr \\
    &=\frac{1}{12\pi s} \int_0^\rho \int_0^\infty e^{-R^2/12s} C(R) dR \ \frac{2\pi -\alpha}{2\pi} C(r) dr \\
    & \qquad -\frac{1}{12\pi s} \int_0^\rho \int_0^\infty e^{-R^2/12s} C(R-r)\frac{\alpha}{2\pi}\ind_{\{R>r\}} dR \ \frac{2\pi -\alpha}{2\pi} C(r) dr\,.
\end{align}
Using the expansion of $C(R)$ we see that the first integral gives (in the formula below $R$ is the scalar curvature and is equal to $2$) 
\begin{align}
    I = |S_\rho| - 2|S_\rho|s + O(s^2) = |S_\rho| - \left(\int_{S_\rho}Rd\vol \right) s + O(s^2) \,.
\end{align}
Expanding $C(\cdot)$ in the second integral we get, up to higher order terms, the linear contribution of the cone case (see \eqref{eq:comp_alpha<0})
\begin{align}
    II = -\frac{\alpha(2\pi-\alpha)}{12\pi s}\int_0^\infty\int_0^R (R-r) r dr\ e^{-R^2/12s} dR + O(s^2) = -\frac{\alpha(2\pi-\alpha)}{\pi}s + O(s^2)
\end{align}
Therefore the linear contribution is the same as the cone case plus the integral of the scalar curvature on $S_\rho$, just as stated in Theorem \ref{2d-cone}.

\subsection{Case $\alpha\in[0,\pi]$}
Now $|\partial B_R(x)| = (2\pi - \beta \ind_{\{R>r\}}) \frac{C(R)}{2\pi}$, where $\beta= \beta^{\alpha,r,R}$ is the angle is Figure \ref{fig:Cone0pi}, where now the geometry is distorted by the spherical geometry. However, note that in the cone case $\beta_{\R^2}^{\alpha,\lambda r,\lambda R} = \beta_{\R^2}^{\alpha,r,R}$, and by local flatness
\begin{align}\label{eq:beta_conv}
    \beta_{\mathbb S^2}^{\alpha,\lambda r,\lambda R} \to \beta_{\R^2}^{\alpha,r,R} \ \text{as $\lambda \to 0$.}
\end{align}
Now
\begin{align}
    \int_{S_\rho} \a_s(x)d\vol(x) &= \frac{1}{12\pi s} \int_0^\rho \int_0^\infty e^{-R^2/12s} \frac{C(R)}{2\pi} \left( 2\pi- \beta\ind_{\{R>r\}} \right) dR \ \frac{C(r)}{2\pi} (2\pi-\alpha) dr \\
    &= \frac{1}{12\pi s} \int_0^\rho \int_0^\infty e^{-R^2/12s} C(R) dR \ \frac{2\pi-\alpha}{2\pi}C(r) dr \\& \qquad -\frac{1}{12\pi s} \int_0^\rho\int_0^\infty e^{-R^2/12s} \frac{C(R)}{2\pi}\beta \ind_{\{R>r\}}dR \ \frac{C(r)}{2\pi}(2\pi-\alpha)dr
\end{align}
As the $\alpha<0$ case, the first integral gives
\begin{align}
    I = |S_\rho| - \left(\int_{S_\rho}Rd\vol \right) s + O(s^2) \,.
\end{align}
For the second we will use the convergence \eqref{eq:beta_conv} in order to reduce to the computations on the flat cone. Since we expect a linear contribution, we divide the second integral by $s$, use the expansion of $C$, and perform the change of variables $R=\sqrt s W$, $r=\sqrt s w$:
\begin{align}
    \frac{II}{s} &= -\frac{1}{s}\frac{2\pi-\alpha}{12\pi s}\int_0^\infty e^{-R^2/12s} R \int_0^R \beta_{\mathbb S^2}^{\alpha,r,R}rdr dR + O(s)\\
    &= - \frac{2\pi-\alpha}{12\pi} \int_0^\infty e^{-W^2/12\pi} W \int_0^W \beta_{\mathbb S^2}^{\alpha,\sqrt s w, \sqrt s W} w dw dW 
\end{align}
Now, all the terms are bounded ($\beta\leq 2\pi$) and multiplied by the exponential decaying function $e^{-W^2/12\pi}$, so uniformly (in $s$) integrable: by dominated convergence the integral converges, as $s\to 0$, to the respective one on the flat cone.

\subsection{Case $\alpha\in[\pi,2\pi)$}
This case, although more involved, is completely analogous to $\alpha\in[0,\pi]$. Note in particular that the quantities $\beta$, $d/R$, $\gamma$ on the flat cone only depend on $r/R$ (so they are scaling invariant), and the analogous quantities on the spherical suspension for $\lambda r$, $\lambda R$ converge to them as $\lambda\to 0$.
\end{appendices}

\bibliographystyle{abbrv}
\bibliography{Refs}

\begin{tabular}{@{}l@{}}%
    \textbf{E-mail addresses}: \text{\href{mailto:mflaim@uni-bonn.de}{mflaim@uni-bonn.de}, \ \href{mailto:ehupp@uni-bonn.de}{ehupp@uni-bonn.de} , \
    \href{mailto:sturm@uni-bonn.de}{sturm@uni-bonn.de}
    }
  \end{tabular}

\end{document}